\documentclass[a4paper,11pt]{amsart}
\usepackage[latin1]{inputenc}
\usepackage{amsmath, amsthm, amssymb, amsfonts, amscd}
\usepackage[english]{babel}
\usepackage[linktocpage=true]{hyperref}
\usepackage{url}
\usepackage{pinlabel}
\usepackage{color}
\extrafloats{100}
\newtheorem{thm}{Theorem}[section]
\newtheorem*{thm*}{Theorem}
\newtheorem{lemma}[thm]{Lemma}
\newtheorem{prop}[thm]{Proposition}
\newtheorem{cor}[thm]{Corollary}
\newtheorem{defi}[thm]{Definition}

\newtheorem{ques}[thm]{Question}

\theoremstyle{remark}

\newtheorem{rmk}[thm]{Remark}
\numberwithin{figure}{section}
\numberwithin{equation}{section}
\newcommand{\N}{\mathbb{N}}
\newcommand{\Z}{\mathbb{Z}}
\newcommand{\R}{\mathbb{R}}
\newcommand{\reid}{\mathcal{G} (K)}
\newcommand{\sfera}{\mathcal{G}_{S} (K)}
\newcommand{\diag}{\mathcal{D}(K)}
\newcommand{\nodi}{\mathcal{K}}
\newcommand{\Ru}{\Omega_1}
\newcommand{\Rd}{\Omega_2}
\newcommand{\Rs}{\Omega_T}

\title{The Reidemeister graph is a complete knot invariant}
\author{Agnese Barbensi and Daniele Celoria}
\date{}

\begin{document}
\begin{abstract}
We describe two locally finite graphs naturally associated to each knot type $K$, called \emph{Reidemeister graphs}. We determine several local and global properties of these graphs and prove that in one case the graph-isomorphism type is a complete knot invariant up to mirroring. Lastly, we introduce another object, relating the Reidemeister and Gordian graphs, and determine some of its properties.
\end{abstract}
\maketitle
\section{Introduction}\label{sec:intro}
The \emph{Gordian graph} is a well known graph in knot theory; its vertices are given by knot types, and two knots have an edge between them whenever they are related by a single crossing change. This graph can be thought of as describing knot theory at ``large scales".

The Gordian graph is however very ill behaved: each vertex of this graph has infinite valence, vertices at distance $2$ are connected by infinitely many distinct minimal paths \cite{baadercross}, and for every $n \ge 1$ there are embeddings of the graphs $\Z^n$ into it \cite{aboutgordian}.

This pathological nature of the Gordian graph makes it usually difficult to pinpoint its properties. For example, it is still unknown whether the Gordian graph is homogeneous \cite{budneyMO}, and figuring out the path distance between two knots is regarded as a hard problem (the computation of unknotting numbers is a subproblem).\\ 

The aim of this paper is to study the opposite point of view: instead of zooming out on the set of all knots, we will describe a way to observe ``under the microscope" each knot type.

To this end, we associate to each knot $K\subset S^3$ a graph, the \emph{Reidemeister graph} $\reid$, having a vertex for each diagram of $K$, and an edge between two diagrams whenever one can be converted into the other by a single Reidemeister move. We will make these definitions precise in the next section, but we note here that, unlike the Gordian graph, each $\reid$ is locally finite, so it is in some sense better behaved, and many of its properties can be studied through combinatorial techniques.

The definitions and proofs in this paper are quite natural and elementary in spirit: the only non-trivial tools we are going to use are Arnold's and Hass-Nowik's diagram invariants (\cite{arnold},\cite{hass2008invariants}), introduced in Section \ref{sec:grafo}. An analogous construction in a slightly different contest (cf. Section \ref{global}) has been carried out by Miyazawa in \cite{miyazawa2012distance}. Subtle differences in our initial setting will allow us to prove Theorem \ref{thm:completeness} (see also Question \ref{ques:minimo}).

The paper is structured as follows: Section \ref{sec:grafo} gives the basic definitions of the planar and $S^2$-Reidemeister graphs, collectively called $\mathcal{R}$-graphs.

In Section \ref{local} we analyse some local properties of these graphs; in particular we will classify all short paths in them (Theorem \ref{prop:lung3}), and examine the change in valence between adjacent vertices. These technical results are going to be crucial to establish the main result of the paper. As a preliminary step we will prove in Theorem \ref{distinguemosse} that the graph can detect which Reidemeister move corresponds to each of its edges, and define a related notion of \emph{diagram complexity}.

Section \ref{global} instead deals with global properties of the $\mathcal{R}$-graphs; we show that unsurprisingly they are non-planar (Proposition \ref{fig:nonplanar}) and not hyperbolic (Proposition \ref{prop:latticeembed}). In addition we show that each Reidemeister graph has only one thick end (Proposition \ref{thickend}), and compute the homology groups of an associated simplicial complex, following Miyazawa's definition (\cite{miyazawa2012distance}).

In Section \ref{s:towards} we are going to prove the main result of this paper, concerning the completeness of the $S^2$-Reidemeister graph invariant:

\begin{thm}\label{thm:completeness}
The $S^2$-Reidemeister graph is a complete knot invariant up to mirroring; that is  $\mathcal{G}_{S}(K) \equiv \mathcal{G}_{S} (K^\prime)$ \emph{iff} $K^\prime = K$ or $\overline{K}$.
\end{thm}

Indeed the proof of this theorem will guarantee a stronger result (Proposition \ref{distinguediagramma}): the isomorphism type of the graph does not only distinguish all knots, but contains enough information to recover some diagrams of the knot (up to mirroring). Moreover, all this data can be extracted from finite portions of the graph (Corollary \ref{cor:pallachar}).\\
We remark that, unlike the previously known complete invariants -such as knot complement \cite{gordon1989}, quandles \cite{joyce1982classifying} and conormal tori \cite{ekholm2016complete}- the proof of completeness for the $S^2$-graph is substantially more elementary, and self contained.

Finally, in Section \ref{blownup} we define yet another kind of graph, relating the Gordian and Reidemeister graphs by a ``blowup'' construction.\\

\textbf{Acknowledgements:} The authors would like to thank C. Collari, P. Lisca, A. Juh\'{a}sz, P. Ghiggini, M. Golla, F. Lin, and D. Eridan for stimulating conversations. The authors are extremely thankful to M. Lackenby for pointing out some mistakes in an early draft and for his continuous help. A special thanks to  M. Marengon, for his remarks, suggestions and valuable help. The first author would like to thank her other supervisors Dorothy Buck and Heather Harrington for their support.
The first author is supported by the EPSRC grant ``Algebraic and topological approaches for genomic data in molecular biology'' EP/R005125/1.

The second author has received support from the European Research Council (ERC) under the European Union's Horizon 2020 research and innovation program (grant agreement No 674978).

The authors would also like to thank the Isaac Newton Institute for Mathematical Sciences, Cambridge, for support and hospitality during the program ``Homology theories in low dimensional topology'', supported by EPSRC grant no EP/K032208/1, where work on this paper was undertaken.

\section{The graph}\label{sec:grafo}

We start by giving some precise definitions of the well known objects we are going to use extensively in the following.

As usual by \emph{knot} we mean the ambient isotopy class of a tame embedding of $S^1$ in $S^3$. The set of unoriented knot types in $S^3$ will be denoted by $\mathcal{K}$, and the set of diagrams representing a knot $K$ by $\diag$. A knot diagram $D \in \diag$ can be thought of as a 4-valent graph in $\R^2$ or $S^2$, by disregarding the crossing information. In order to avoid confusion, we are going to refer the $4$-valent graph associated to a diagram as the \emph{knot projection}. We will call an \emph{arc} each portion of a diagram or projection which connects two crossing points, and denote by $\alpha(D)$ the number of arcs in $D$. By the handshaking lemma we have $\alpha(D) = 2cr(D)$, where $cr(D)$ is the number of crossings in $D$. From now on, we are going to assume that, unless otherwise stated, each diagram $D$ contains at least one crossing.

The complement of a planar knot projection is composed by polygons, with the exception of the ``external'' region which is a punctured polygon; we will call this external part a polygon as well. As is customary we denote by $p_k (D)$ the number of polygons with $k$-faces.

We have
\begin{equation}
\alpha(D) = \frac{1}{2} \sum_{k \ge 1} k\cdot p_k(D),
\end{equation}
and the number of regions in $S^2 \setminus D$ is $\sum_{k \ge 1} p_k(D)$.

We say that a planar diagram $D$ is \emph{periodic} if there exists a non-trivial rotation of the projection plane taking $D$ to itself, and a knot $K$ is periodic if it admits a periodic diagram.
The \emph{order of periodicity} is then just the order of the rotation acting on the diagram.

A diagram on the $2$-sphere is said to be \emph{periodic} if there is a non-trivial finite order, orientation preserving self-diffeomorphism of the sphere, that takes it to itself.
The $(2n+1,2)$-torus knots are an example of knots that exhibit a $D_{2n+1}$ periodicity\footnote{Here $D_m$ denotes the dihedral group of order $2m$.} on the $2$-sphere and cyclic periodicity of order $2$ or $d$ (with $d|2(2n+1)$) on the plane.

Conversely, a knot which does not admit any periodic diagram is said to be \emph{non-periodic}.

A \emph{planar isotopy} can modify locally a knot diagram by moving slightly an arc as in Figure \ref{fig:localisot}, or by displacing a whole diagram, without creating or removing any crossing.

\begin{figure}[h]
\includegraphics[width=8cm]{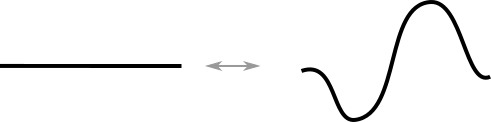}
\caption{Example of a local planar isotopy acting on an arc.}
\label{fig:localisot}
\end{figure}

Note that by considering $\R^2$ rather than $S^2$ as the ambient space we get a ``larger" set of diagrams; the two diagrams of the left trefoil minimising its crossing number  shown in Figure \ref{fig:trifogli} are planar isotopic on the 2-sphere but not on the plane.\\

\begin{figure}[h]
\includegraphics[width=6.7cm]{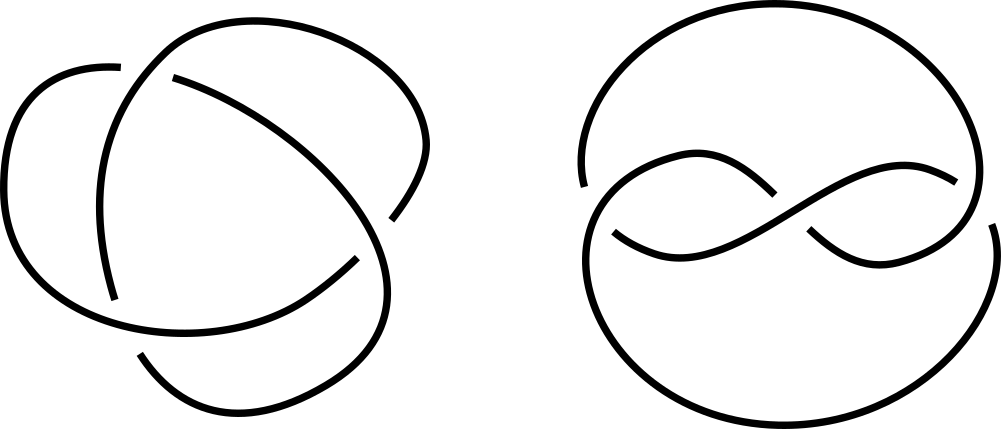}
\caption{These two diagrams of the left trefoil are planar isotopic on $S^2$, but not on $\R^2$.}
\label{fig:trifogli}
\end{figure}

Probably the most fundamental result in knot theory is Reidemeister's Theorem\footnote{We will adhere to the standard attribution of the theorem, which in fact was independently discovered by Alexander and Briggs \cite{alexbriggs}.} \cite{reidemeister}, stating that two diagrams represent the same knot type \emph{iff} they are related by a finite sequence of local moves, known as Reidemeister moves, together with planar isotopy. These moves are described in Figure \ref{reidemoves} below.
\begin{figure}[h]
\includegraphics[width=11cm]{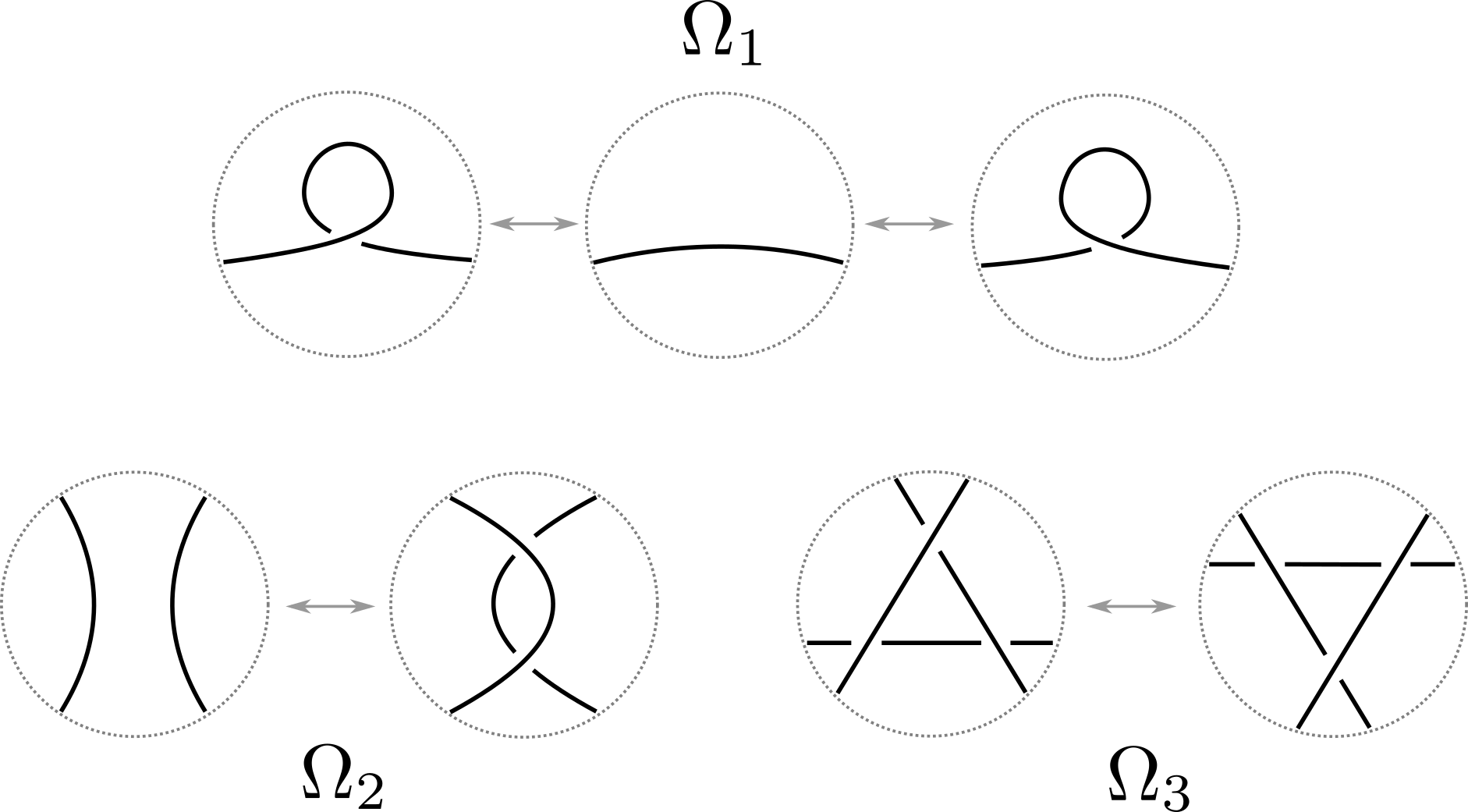}
\caption{The standard Reidemeister moves on knot diagrams.}
\label{reidemoves}
\end{figure}
Note that this set of moves is not \emph{minimal} (cf. \cite{polyak2010minimal} for the statement in the oriented case); in fact one of the two $\Omega_1$ moves could be discarded (Figure \ref{fig:triangoloreid}). However the choice of this slightly larger set will be crucial in the proof of all the upcoming results.
In what follows we will find it convenient to divide the $\Omega_2$ moves in two kinds; the first ones consist of those $\Omega_2$ moves performed on the configuration in Figure \ref{fig:tentacle}, which is called a \emph{tentacle}. We will denote them by  $\Omega_T$, and call them \emph{tentacle moves}.
\begin{figure}[h]
\includegraphics[width=6cm]{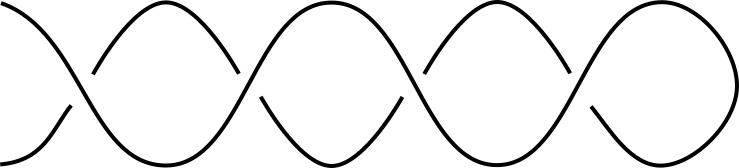}
\caption{A tentacle configuration of height $3$. Note that the crossings have alternating signs.}
\label{fig:tentacle}
\vspace{0.5cm}
\includegraphics[width=10cm]{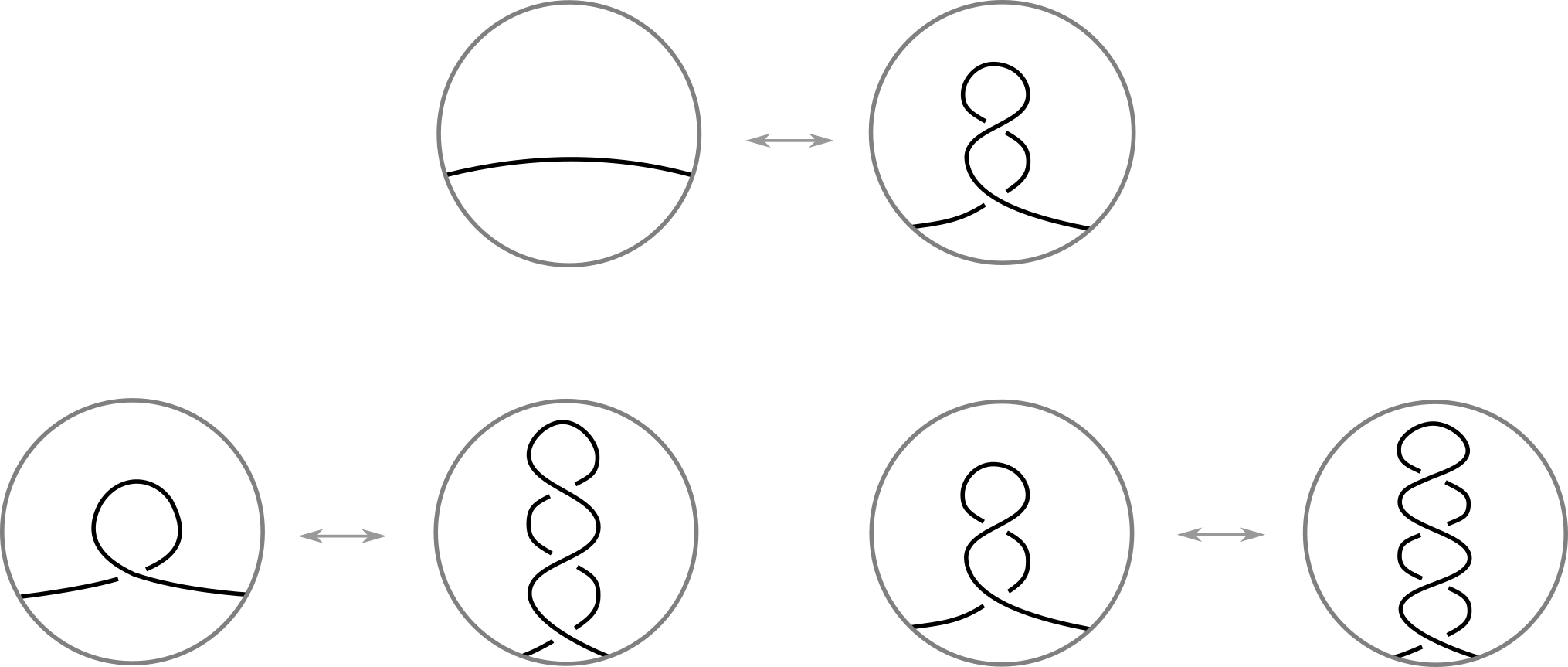}
\caption{The tentacle moves. In the top part of the figure a tentacle move creating a tentacle configuration of height $1$ is shown; such a move always arises as the superposition of an arc on itself.}
\label{fig:tentmoves}
\end{figure}
In other words, $\Rs$ moves are the $\Omega_2$ moves that create a tentacle (and their inverses), as in Figure \ref{fig:tentmoves}. We will say that a tentacle configuration has \emph{height $m$} if it can be expressed as the composition of $m+1$ $\Omega_1$ moves (with alternating signs). In particular a tentacle of height $1$ is the result of performing two $\Ru$s with opposite curls, or a $\Rs$ move in which the two affected strands belong to the same arc (top of Figure \ref{fig:tentmoves}). A tentacle of height $m$ contains $m-1$ sub-tentacles of heights $m-1, \ldots, 1$ as sub-configurations.

The other kind (which we will simply call $\Omega_2$) instead is any other Reidemeister move of type $2$.
The reason for this distinction will become apparent in the next sections (see Theorem \ref{prop:lung3}); in fact we are going to prove that tentacle moves are intrinsically distinguished from the other moves (Theorem \ref{distinguemosse}).

Additionally, if $\Omega$ denotes a Reidemeister move which is not a $\Omega_3$, there are two cases, according to whether we are doing or undoing the move. Hence, when necessary, we are going to denote a move by $\Omega^+$ or $\Omega^-$ if it increases (respectively decreases) the crossing number.

Recall that the \emph{writhe} of a diagram $D$ is the sum of the signs of the crossings; it changes by $\pm 1$ when a $\Omega_1$ move is performed, and is left unchanged by the other moves. The last definition we will need is the \emph{mirror} $\overline{D}$ of a diagram $D$, which is just the diagram obtained by switching all crossings in $D$. A knot is said to be \emph{amphichiral} if it is unchanged under mirroring of its diagrams.

We are ready to introduce the two basic versions of the object we are going to study throughout the rest of the paper.

\begin{defi}
Given a knot $K\subset S^3$ define the \emph{Reidemeister graph} of $K$, $\reid$ as the graph whose vertices are the planar diagrams of $K$ up to planar isotopy, and has an edge between two diagrams \emph{iff} they are connected by a single Reidemeister move. If we replace planar diagrams of $K$ with diagrams on the $2$-sphere (up to isotopies of $S^2$) we obtain the $S^2$-Reidemeister graph $\mathcal{G}_{S}(K)$.
\end{defi}

In what follows we will use the term ``Reidemeister graphs'' or $\mathcal{R}$-graphs to denote both $\reid$ and $\sfera$.
\begin{rmk}
It might as well happen that two diagrams are connected by two different moves (see \emph{e.g.}$\;$ Figure \ref{fig:multiedge1}), which will be considered as different edges in the graphs. On the other hand, moves coinciding up to a planar isotopy will be represented by a single edge.\\
There are a few immediate consequences of this definition; first of all the isomorphism class of the graphs $\reid$ and $\sfera$  are knot invariants, and they are unchanged under mirroring of the knot.

Also, Reidemeister's Theorem implies that, for each $K \in \mathcal{K}$ the corresponding $\mathcal{R}$-graphs are connected.
\end{rmk}

It might seem strange to define invariants that are more complicated than the object we started with. However, beyond their intrinsic interest, the $\mathcal{R}$-graphs will allow us to produce several related simple numerical invariants.

To prove many of the local structure results of Section \ref{local} for the graphs $\reid$ and $\sfera$, we will need the diagram invariant introduced by Hass and Nowik in \cite{hass2008invariants}, whose properties are concisely recalled below. In the specialized form we are going to use it, this invariant takes values in the free abelian group generated by the formal variables $\{X_s, Y_s \}_{s \in \Z}$. We call an $\Omega_1^+$ positive if the crossing (for an arbitrary choice of orientation) is positive, and negative otherwise; we will call a $\Omega_2$ move \emph{matched} if the two strands go in the same direction, and \emph{unmatched} otherwise (see Figure \ref{fig:differentir2}). Note that $\Omega_T$ moves are always unmatched.

\begin{figure}[h!]
\includegraphics[width=5cm]{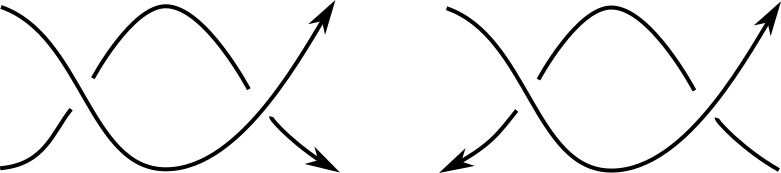}
\caption{A matched and unmatched $\Omega_2$ move (left and right respectively).}
\label{fig:differentir2}
\end{figure}
If two diagrams $D$ and $D^\prime$ differ by a single Reidemeister move, the corresponding $I_{lk}$-invariants differ as shown below, for some $n,m,r \in \Z$.
\begin{itemize}
\item $I_{lk} (D^\prime) = I_{lk} (D) + X_0$  if the move is a positive $\Omega_1^+$.
\item $I_{lk} (D^\prime) = I_{lk} (D) + Y_0$  if the move is a negative $\Omega_1^+$.
\item $I_{lk} (D^\prime) = I_{lk} (D) + X_n + Y_{n+1}$  if the move is a matched $\Omega_2^+$.
\item $I_{lk} (D^\prime) = I_{lk} (D) + X_m + Y_m$  if the move is a unmatched $\Omega_2^+$.
\item $I_{lk} (D^\prime) = I_{lk} (D) + X_0 + Y_0$  if the move is a $\Omega_T^+$.
\item $I_{lk} (D^\prime) = I_{lk} (D) +  \begin{cases}  \pm (X_{r}- X_{r+1}) \\  \pm (Y_{n}- Y_{n+1}) \end{cases}$ if the move is a $\Omega_3$.
\end{itemize}

The only other non-trivial invariants we are going to use are Arnold's \emph{perestroika} invariants $St, J^{\pm}$, first defined in \cite{arnold}. These are invariants of regular homotopy classes of immersions of $S^1$ in $\R^2$ or $S^2$. They change in a controlled way under perestroikas, that is, the analogue of Reidemeister moves for immersions $S^1 \looparrowright \R^2 $ (or $S^2$), as shown in Figure \ref{fig:perestroikas}. We will use them on the knot projections associated to the diagrams.
Note that there is no analogue of $\Omega_1$ moves for immersions, since performing it would change the \emph{index} of the curve.

\begin{figure}[h!]
\includegraphics[width=5cm]{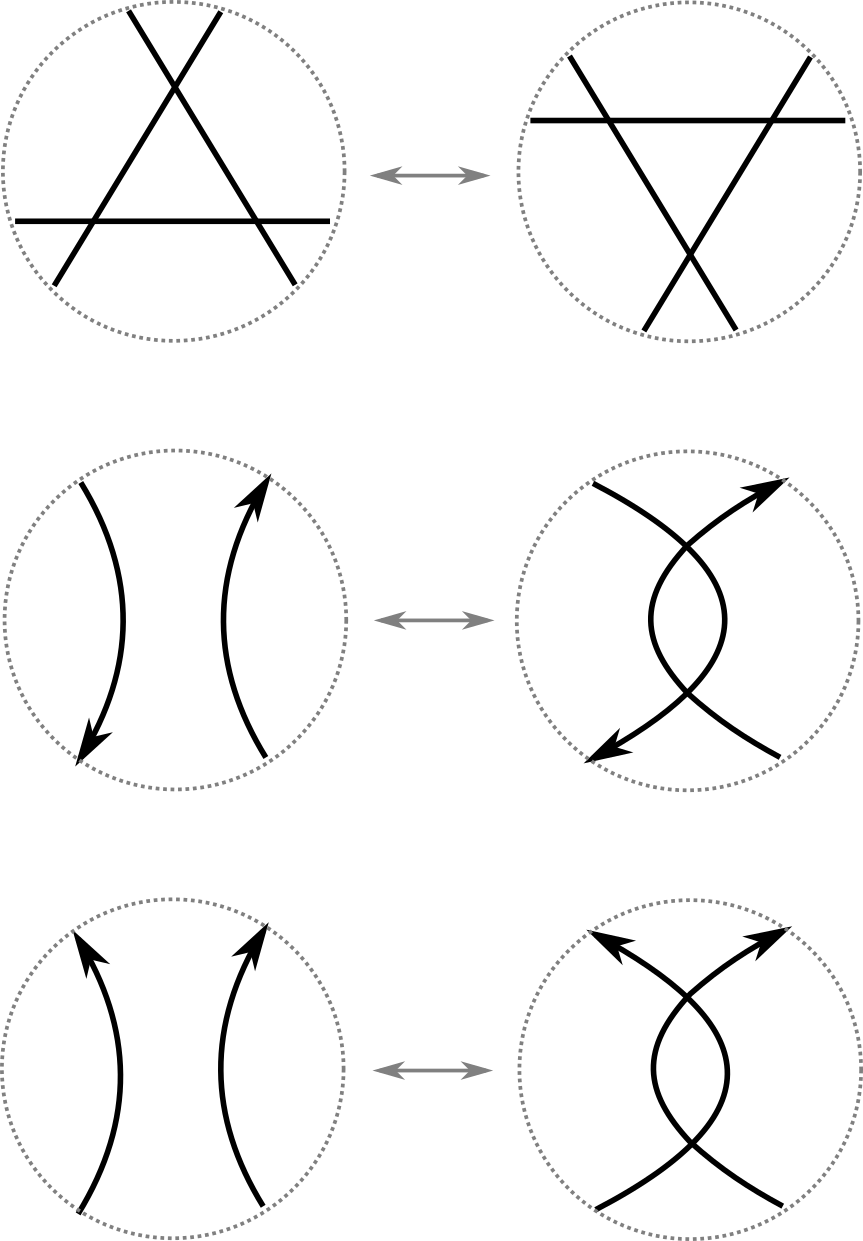}
\caption{A triple point perestroika, followed by the two possible self-tangency perestroikas.}
\label{fig:perestroikas}
\end{figure}

The invariant $St$ changes by $\pm 1$ under a triple point perestroika (which corresponds to a $\Omega_3$ move in our setting), and is left unchanged under selftangency perestroikas (corresponding to $\Omega_2$s and $\Omega_T$s). On the other hand, the invariant $J^{+}$ is unchanged under triple point perestroikas and changes by a fixed positive amount (conventionally 2), when a direct tangency perestroika is performed (that is a matched $\Omega_2^+$). The invariant $J^-$ behaves in similarly, but changes only for inverse selftangency perestroikas (that is an unmatched $\Omega_2^+$ or $\Omega_T^+$ in our case).

\section{Local properties}\label{local}

Given a knot $K \in \nodi$ and $D \in \diag$, the Reidemeister graphs can be naturally endowed with the path metric. Note that the distance induced by this metric coincides with the minimal number of Reidemeister moves connecting two diagrams. We denote
by $S(D)$ the subgraph induced by the vertices having distance $\le 1$ from  $D$. As we will see in what follows, a lot of information about a diagram $D$ can be extracted from $S(D)$.\\
The next results are aimed at understanding in detail the structure of small portions of the Reidemeister graph, in both the periodic and non-periodic cases.

We will find it convenient to denote by $\#\Omega_i^{\pm} (D)$ the number of  Reidemeister $i$ moves of type $\pm$ which can be applied to $D$.\\

This next result states that there are no ``cosmetic Reidemeister moves'', meaning that a Reidemeister move necessarily changes the diagram, even up to planar isotopy.
\begin{prop}\label{noselfedges}
The graphs $\reid$ and $\sfera$ do not contain any self-edges.
\end{prop}
\begin{proof}
Since $\Omega_1$ and $\Omega_2$ moves change the crossing number, they can be immediately ruled out. The only possibility is then to have a $\Omega_3$ move that, if performed, takes a diagram $D \in \diag$ to itself (up to planar isotopy). It is however easy to exclude this case as well using the Hass-Nowik $I_{lk}$ invariant (or Arnold's $St$): as recalled in the previous section this invariant changes in a non-trivial manner under $\Omega_3$ moves.
\end{proof}

It is easy to realize that for a given knot, its $\mathcal{R}$-graph contains infinitely many multi-edges, of any order: just take $n$ identical curls on the same arc for one diagram and $n+1$ on the other. Then there are $n+1$ edges connecting them, corresponding to the possible choices for adding another curl, as shown in the top part of Figure \ref{fig:multiedge1} for $n = 2$.\\
It is also possible to find multi-edges induced by $\Omega_2$ moves, as shown in the middle and lower parts of Figure \ref{fig:multiedge1}.
\begin{figure}[h!]
\includegraphics[width=8cm]{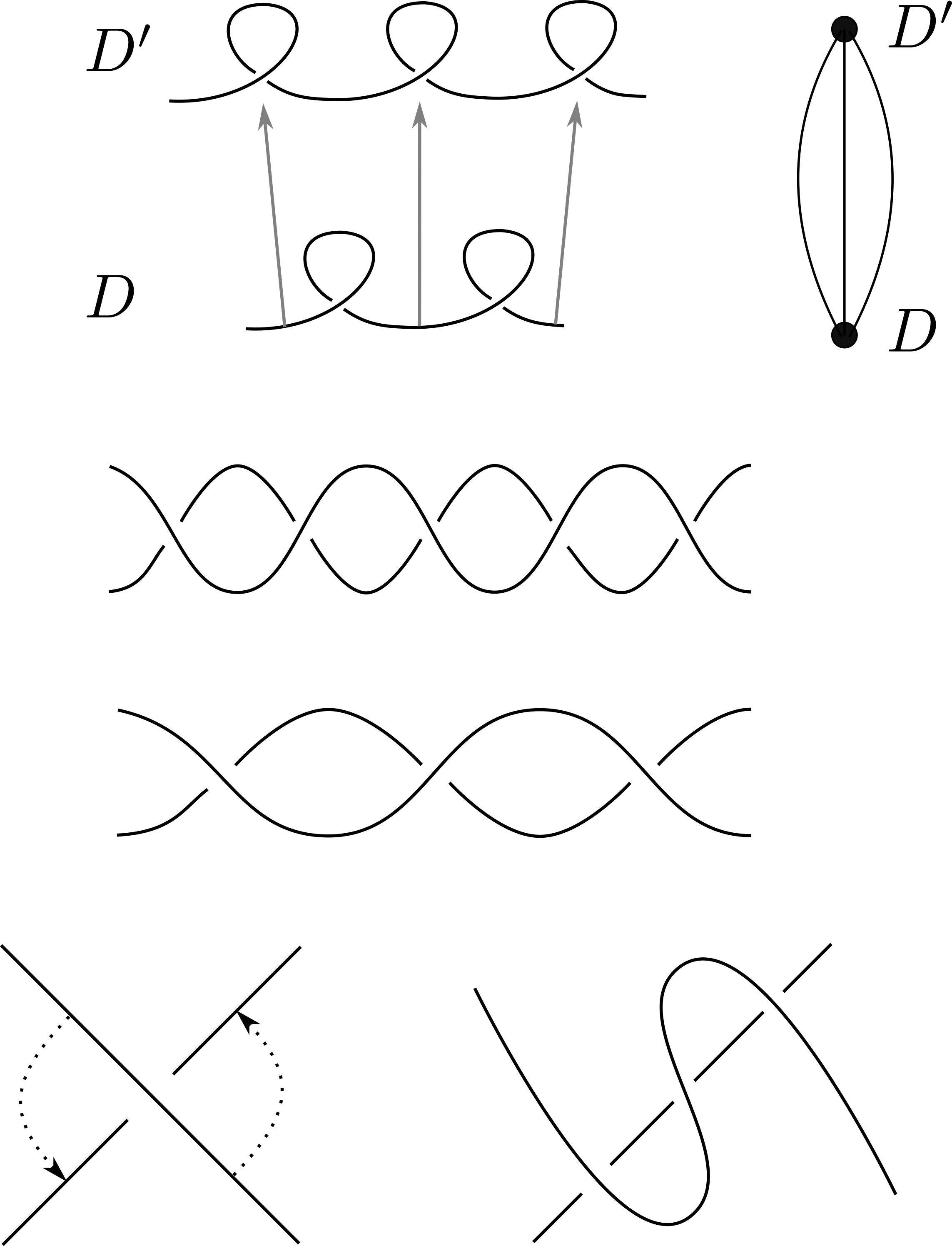}
\caption{In the top part some curls on an arc inducing a multi-edge on the graph, together with the corresponding configuration. In the central part, a multi-edge induced by $\Omega_2$ (or $\Omega_T$) moves, and in the lower part a $2$ multi-edge induced by $\Omega_2$s.}
\label{fig:multiedge1}
\end{figure}

In fact, using the configuration in the lower part of Figure \ref{fig:multiedge1}, it is immediate to show that the only radius $1$ ball not containing multi-edges is centered in the crossingless diagram of the unknot.

If the knot is periodic, one can also have multi-edges of the form shown in Figure \ref{fig:multiedge}.
\begin{figure}[h!]
\includegraphics[width=7cm]{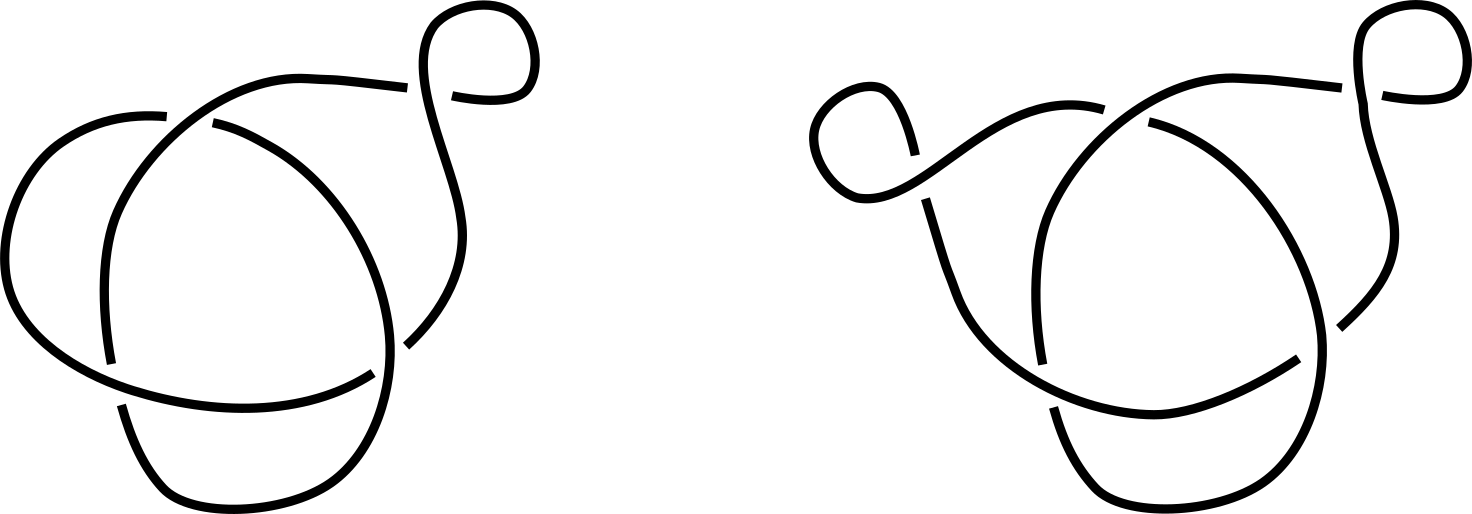}
\caption{There are two inequivalent $\Omega_1$ moves that take one diagram to the other. Note that, even if the diagrams are not periodic, they represent a periodic knot type.}
\label{fig:multiedge}
\end{figure}
It is however easy to prove\footnote{\emph{e.g.} using Arnold's invariants for $\Rd, \Omega_3$s, and $I_{lk}$ for $\Ru$s.} that each multi-edge must be composed of moves of the same kind.

We will say that a graph contains a \emph{triangle} if there are $3$ distinct vertices, such that each vertex is at distance $1$ from the other two.

We want to analyse the shape of the cycles in $S(D)$, other than the multi-edges. It is easy to find a cycle of length $3$, shown in Figure \ref{fig:triangoloreid}. Moreover, since this cycle can start from any unknotted portion of an arc, it is ubiquitous in all Reidemeister graphs.
\begin{figure}[h!]
\includegraphics[width=6cm]{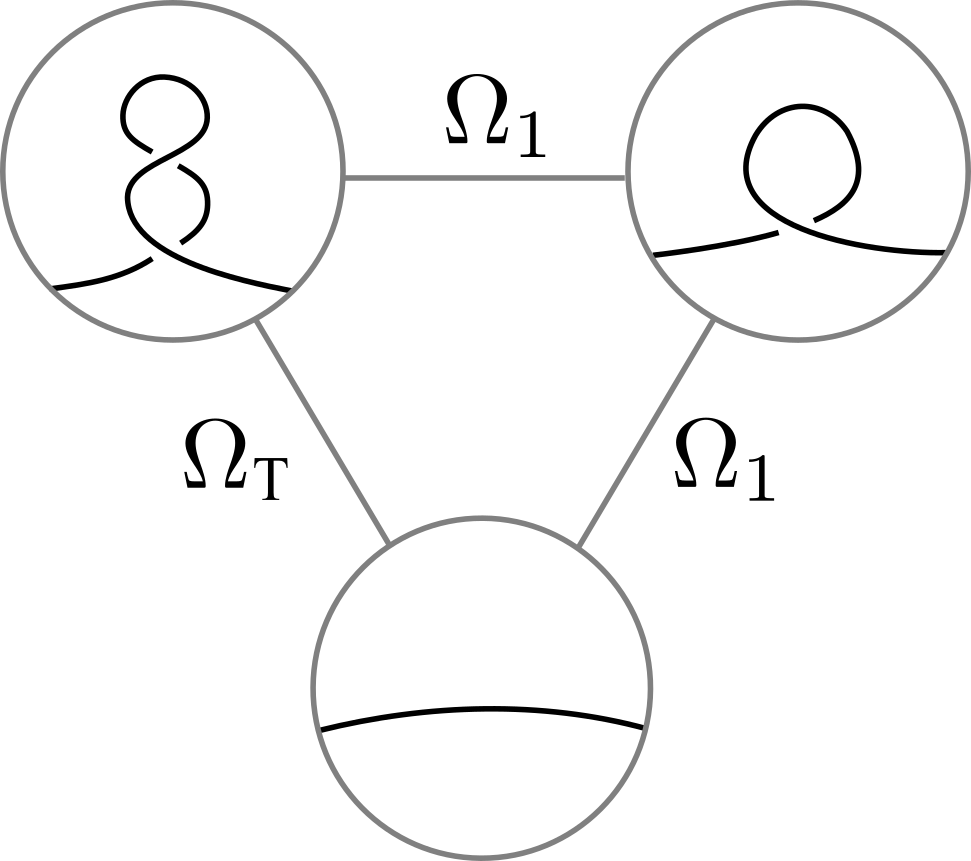}
\caption{A triangle composed by $\Omega_T^{\pm}$-$\Omega_1^{\mp}$-$\Omega_1^{\mp}$.}
\label{fig:triangoloreid}
\end{figure}
A similar and slightly more elaborate example involving a ``higher tentacle'' is shown in Figure \ref{fig:tentacletriangle}. 
The following result will establish that in some sense these are the only possible cases. Moreover it will permit us to explore the main properties of the graph. Its proof is roughly based on the following idea: the total sum of any diagram invariant has to vanish on a closed cycle. In most cases it will be sufficient to consider very simple diagram invariants, such as the crossing number.

\begin{figure}[h!]
\includegraphics[width=5cm]{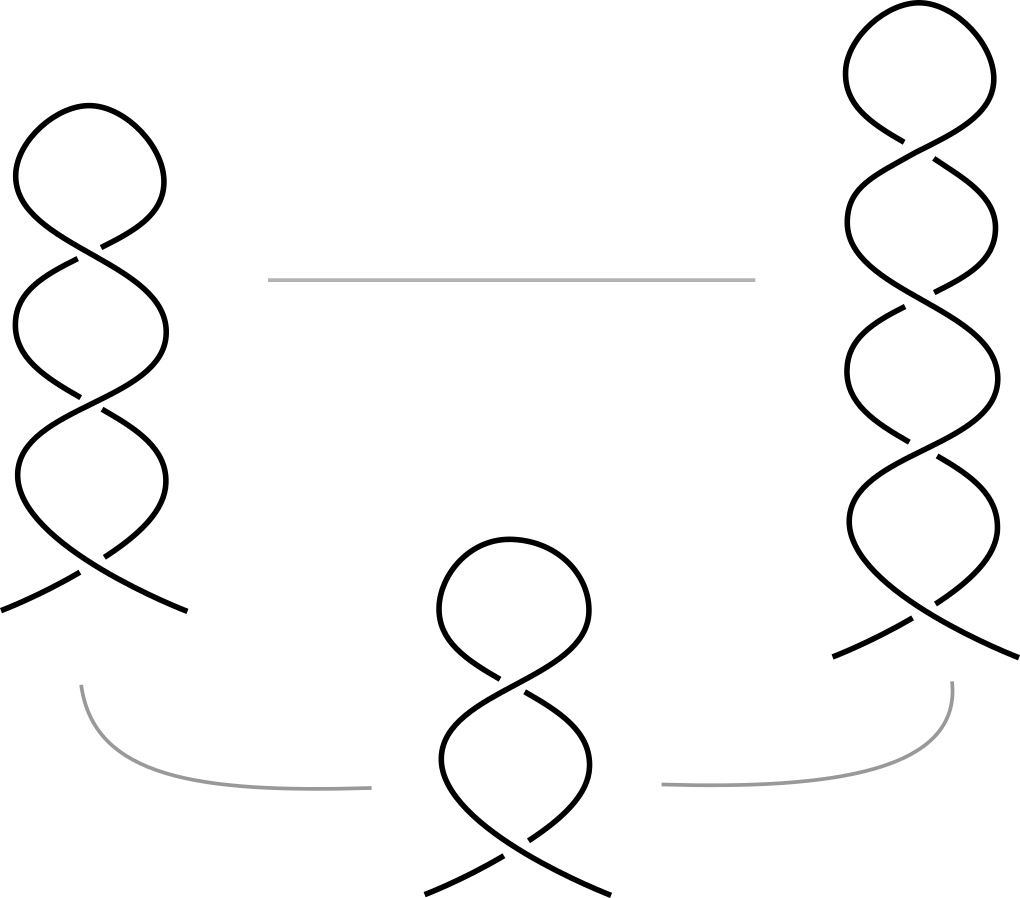}
\caption{A triangle involving some tentacle configurations.}
\label{fig:tentacletriangle}
\end{figure}

\begin{thm}\label{prop:lung3}
If $K$ is a non-trivial knot, the only triangles in its Reidemeister graphs are of the form $\Rs^\pm$-$\Ru^\mp$-$\Ru^\mp$. If instead $K$ is the unknot $\bigcirc$, there are some sporadic exceptions, shown in Figure \ref{fig:controesempio}, of cycles of the form $\Rd^\pm$-$\Ru^\mp$-$\Ru^\mp$.
\end{thm}

\begin{figure}[h!]
\includegraphics[width=7cm]{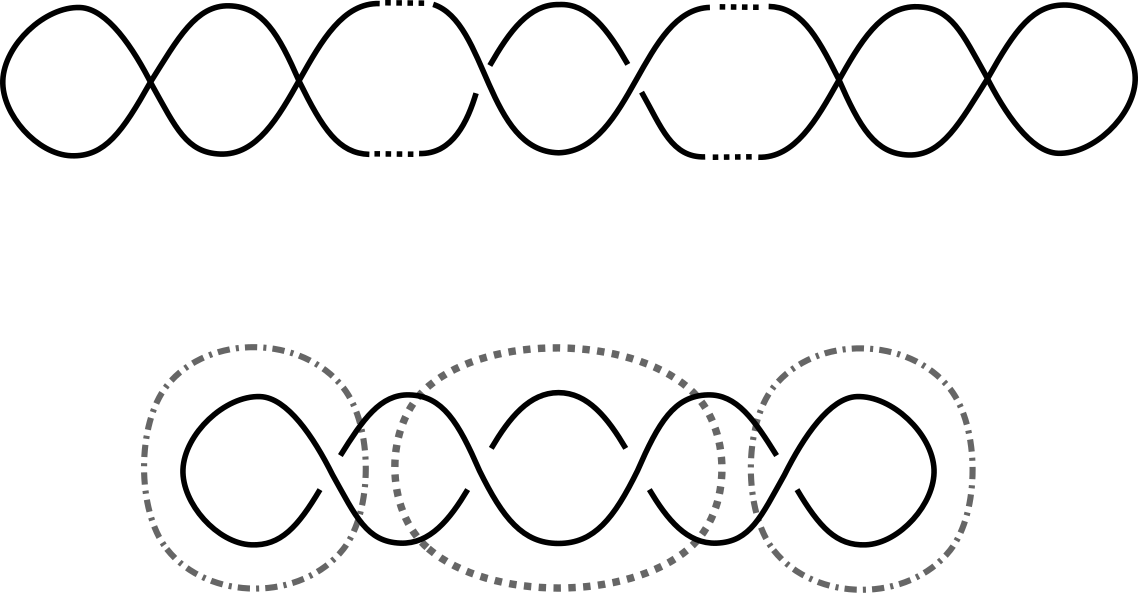}
\caption{On the top, a diagram for the unknot obtained from the crossingless one by a $\Omega_{2/T}$ move followed by succesive $\Omega_1^+$ moves creating crossings of any sign. On the bottom, an example of a triangle involving diagrams of this kind: performing the central $\Omega_2^-$ move or the two lateral $\Omega_1^-$s producec the same diagram.}
\label{fig:controesempio}
\end{figure}

\begin{proof}
Suppose we have a length $3$ cycle, connecting the diagrams $D_0, D_1$ and $D_2$. The total change of crossing number  must be 0, hence we can immediately exclude most cases: a priori the only possible combinations (up to permutations) of $3$ Reidemeister moves  that could work are:
\begin{enumerate}
\item $\Omega_3$-$\Omega_3$-$\Omega_3$
\item $\Omega_3$-$\Omega_2$-$\Omega_2$
\item $\Omega_3$-$\Omega_T$-$\Omega_2$
\item $\Omega_3$-$\Omega_T$-$\Omega_T$
\item $\Omega_3$-$\Omega_1$-$\Omega_1$
\item $\Omega_2$-$\Omega_1$-$\Omega_1$
\item $\Omega_T$-$\Omega_1$-$\Omega_1$
\end{enumerate}

It is easy to exclude cases $(1)$ to $(4)$ using Arnold's $St$ invariant: in any cycle (not containing $\Ru$s) the number of $\Omega_3$ moves must be even. Case $(5)$ can instead be excluded using Hass-Nowik's invariant: the $\Omega_3$ move contributes to $I_{lk}$ with two consecutive terms (that is, of the form $A_n - A_{n+1}$ for $A = X$ or $Y$, and $n \in \Z$), while the $\Omega_1$s can only add some terms of the form $\pm A_0$. Hence the total change in the sum can not be $0$.\\

Finally we can focus on cases $(6)$ and $(7)$ and exclude the former. First notice that, in order to preserve the crossing number, a $\Omega_2^{\pm}$, must be followed by two $\Omega_1^{\mp}$.
Moreover, using Hass-Nowik's invariant we can conclude that the crossings involved in the $\Omega_1$ moves have different signs, and that the $\Omega_2$ is unmatched.

Define the \emph{self-intersection number} $SI(P)$ of a region $P$ in the complement $S^2 \setminus D$, as the number of crossings in the boundary of $P$ that connect $P$ to itself.
We can associate to each diagram $D$ an unordered $N$-tuple
$SI(D) = ( SI(P_1), SI(P_2), \cdots, SI(P_N)  )$ where $N$ is the number of regions in $S^2 \setminus D$.

Performing a $\Omega^-_1$ move always decreases the self-intersection number of a single region by 1, and leaves the self-intersection numbers of the other regions unchanged, see Figure \ref{fig:selfint1}.

\begin{figure}[h!]
\includegraphics[width=6cm]{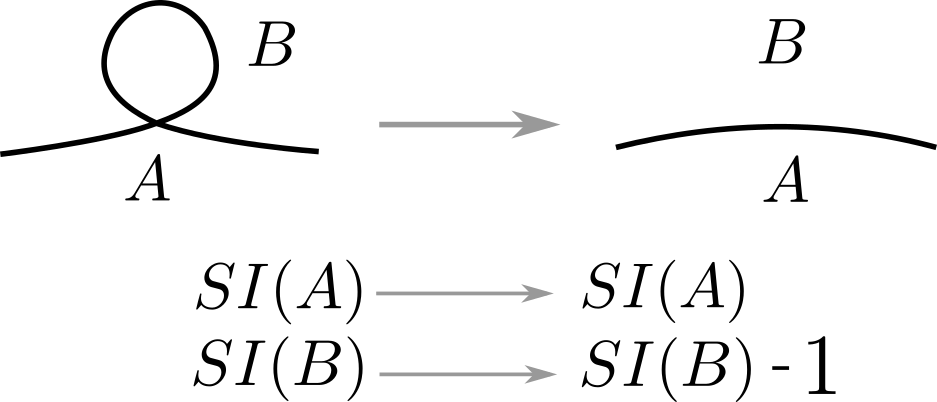}
\caption{Undoing an  $\Omega_1$ move decreases $SI(B)$ by $1$ and leaves $SI(A)$ unchanged.}
\label{fig:selfint1}
\end{figure}
On the other hand, a $\Omega^-_2$ move can change the self-intersection $N$-tuple in two different ways, depending on whether the regions denoted $A$ and $E$ in the lower part of Figure \ref{fig:selfint2} coincide or not\footnote{The regions denoted by $B$ and $C$ are always distinct, otherwise the diagram would represent a two component link.}. If $A$ and $E$ coincide, then the component $SI(A)$ of $SI(D)$ decreases by $2$ when the move is performed (as in the upper part of Figure \ref{fig:selfint2}). 
In the other scenario, the only change in $SI(D)$ comes from the merging of the regions $B$ and $C$; the new region formed has as self-intersection number greater or equal to $SI(B)+SI(C)$ (lower part of Figure \ref{fig:selfint2}).
\begin{figure}
\includegraphics[width=5cm]{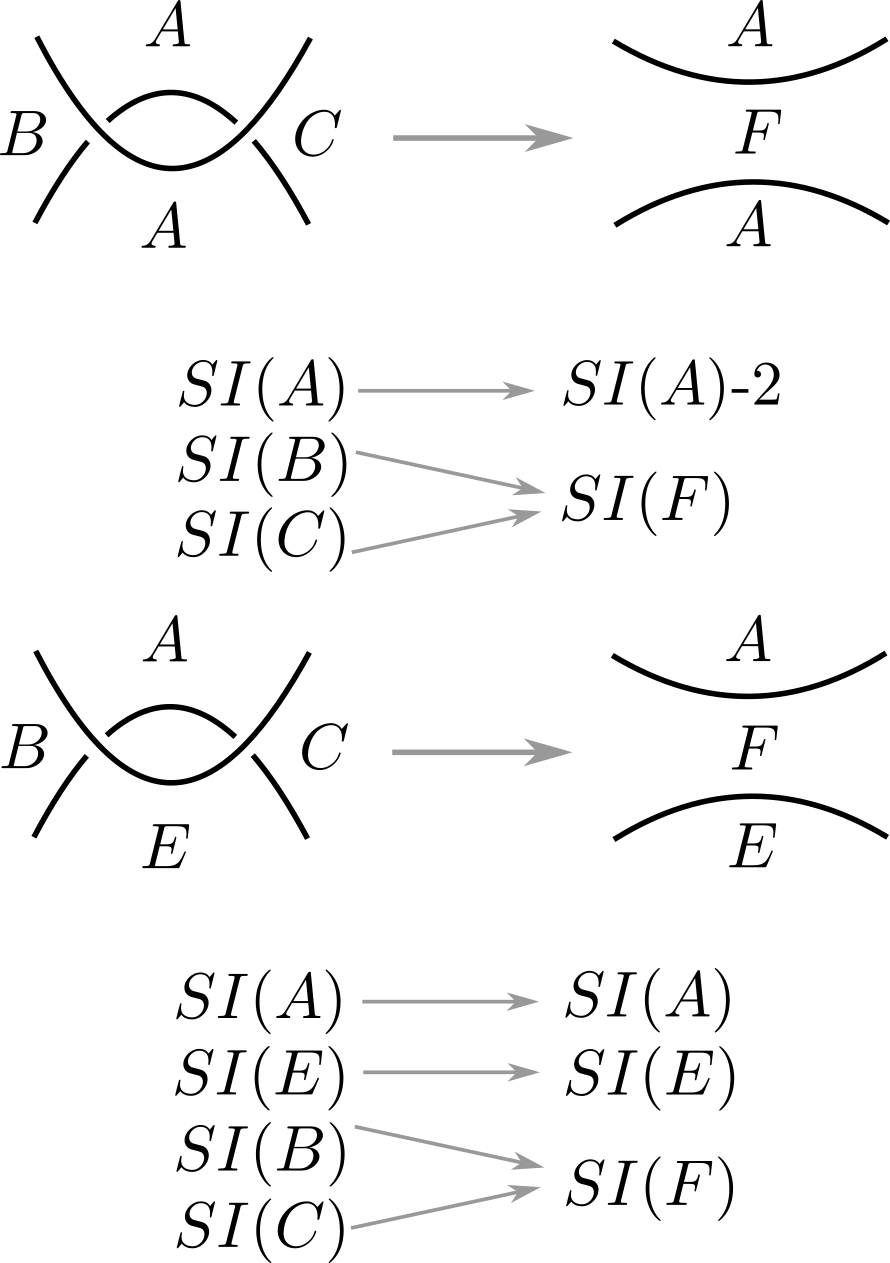}
\caption{Changes in the self-intersection numbers for an $\Omega_2^-$ move.}
\label{fig:selfint2}
\end{figure}

Suppose now by contradiction that there exists a cycle of the form $\Omega_2^{\pm}$-$\Omega_1^{\mp}$-$\Omega_1^{\mp}$.  That means we can obtain a diagram $D'$ from $D$ either by performing an $\Omega_2^-$ move or a sequence of $2$ $\Omega_1^-$ moves on $D$ and that the changes in $SI(D)$ must be the same.
Now observe that, while the self-intersection number of at least one region decreases with two consecutive $\Omega_1^-$s, if in the $\Omega_2$ move the regions $A$ and $E$ are distinct, the sum of the self-intersections over all regions is increased or left unchanged. This fact allows us to exclude the case in which the $\Omega_2$ move is as in in the lower part of Figure \ref{fig:selfint2}.

We will find it useful to divide the discussion in cases, depending on the mutual positions of the curls undone by the $\Omega_1^-$ moves. The relevant portions of the initial diagram $D_0$ are displayed in Figure \ref{fig:cases} for each of these possibilities: in the first and second row we show the mutual positions the curls can have if they do not both appear in $D_0$; in other words, the $1$-region undone by the second $\Omega_1$ appears after undoing the first curl\footnote{Recall that these two moves must have opposite signs.}.

In the third row letters indicate the regions touched by the curls: these regions can either coincide or not.

\begin{figure}
\includegraphics[width=5cm]{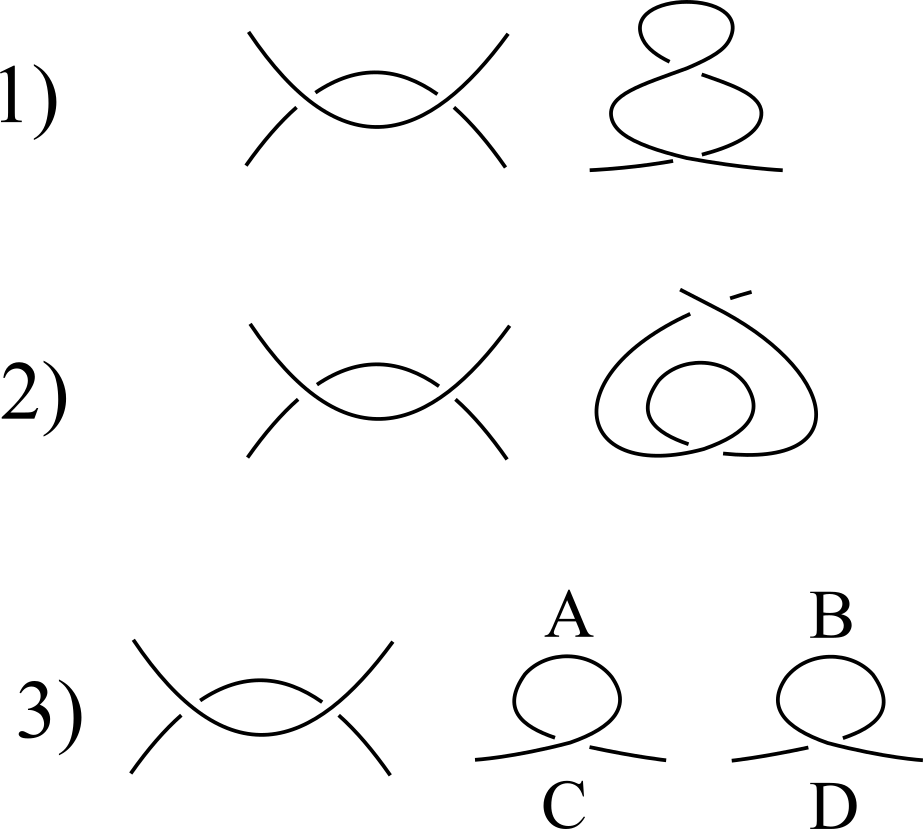}
\caption{In each row we show the portions of $D_0$ involved in the $\Omega_*$ moves. In the first and second row, we assume that the curls undone by the $\Omega_1$ moves do not appear both in the diagram, while in the third case they do. Letters in the latter row indicate the regions touched by the curls.}
\label{fig:cases}
\end{figure}

In what follows, for each one of these cases we will either prove that the $\Omega_2^-$ needs to be a tentacle move (Figure \ref{fig:tentmoves}), or exclude the configuration.

The first case in Figure \ref{fig:cases} can be settled as follows: consider Figure \ref{fig:a4triangle}; we can see that in the diagram $D_2$ there is a tentacle appearing. Since the diagrams $D_1$ and $D_2$ are equivalent by hypothesis, the tentacle in $D_2$ must appear somewhere in $D_1$. Moreover, since they coincide out of the portions of diagram drawn in the figure, the presence of a tentacle in $D_1$ \emph{far away}\footnote{Here and in what follows, by ``far away'' we mean that  the configuration is left untouched by the moves considered.} from the portions drawn would imply the existence of an identical tentacle somewhere in $D_2$, and we would still have one more tentacle in $D_2$ than in $D_1$. 

A similar recursive argument applies if the tentacle appears by undoing the $\Omega_1$-moves, as in Figure \ref{fig:tentacoliappearing}. In fact, in each of the cases shown in Figure \ref{fig:tentacoliappearing}, there is a configuration in $D_2$ which does not appear in $D_1$, and the only way to have $D_1=D_2$ is to find this configuration in $D_1$. Iterating this procedure, one sees that the two diagrams can not be equivalent.   
\begin{figure}
\includegraphics[width=8cm]{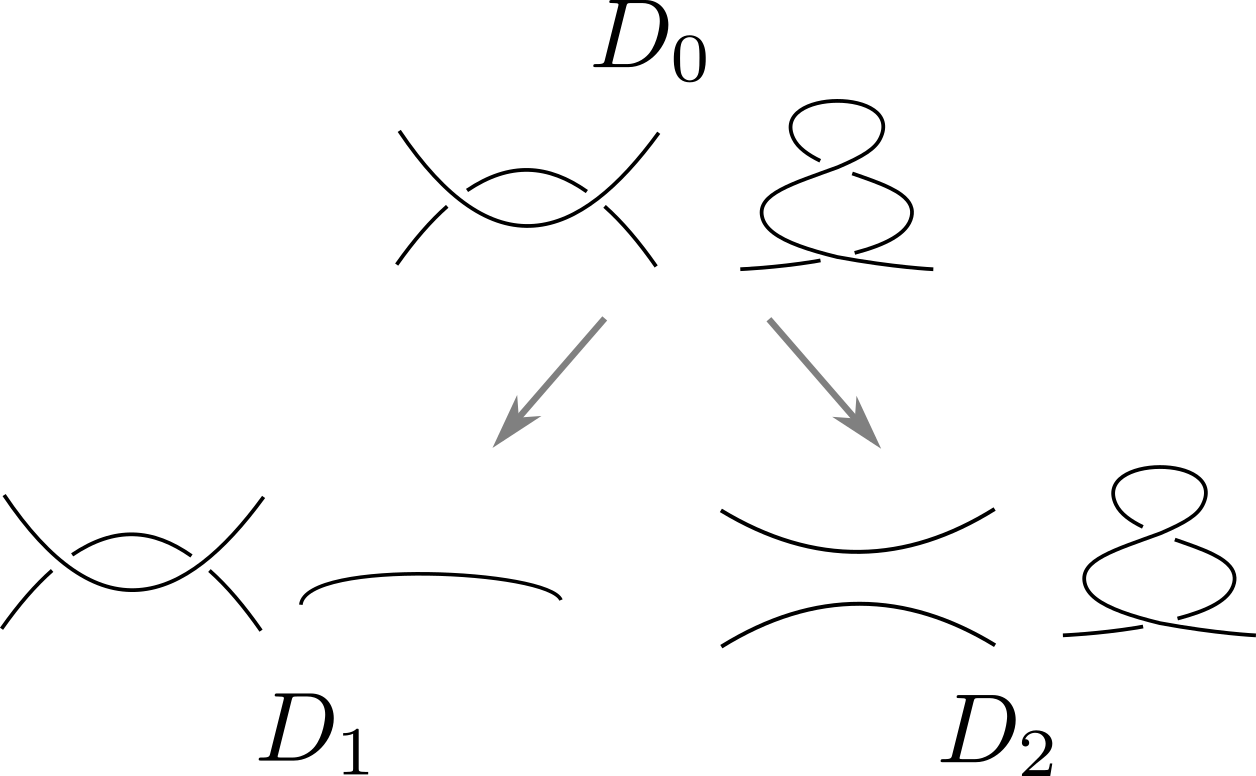}
\caption{$D_1$ is the diagram obtained after performing the two $\Omega_1^-$ moves: together they cancel the tentacle appearing in $D_0$. $D_2$ is the result of undoing the $\Omega_2$ move in the left-hand portion of diagram in $D_0$.}
\label{fig:a4triangle}
\end{figure}
\begin{figure}
\includegraphics[width=11cm]{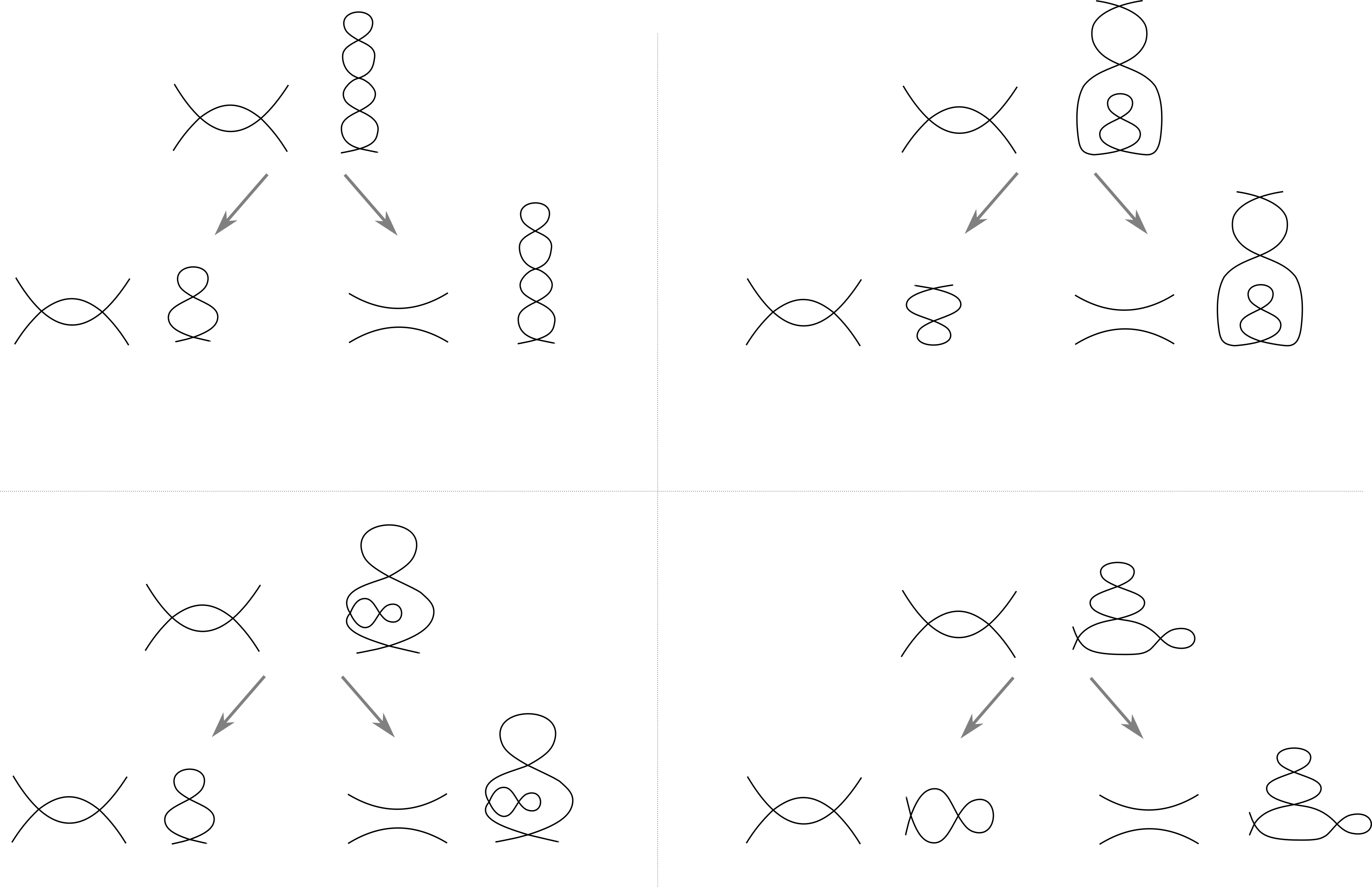}
\caption{The only possible ways a tentacle can appear in case $1)$ after performing two $\Omega_1^-$ moves. In each of these cases we can exclude that the diagrams form a triangle by a recursive argument.}
\label{fig:tentacoliappearing}
\end{figure}

It follows that the only way $D_1$ can be equivalent to $D_2$ is if the $\Omega_2$ is in fact an $\Omega_T$; thus the corresponding part of the diagram is a portion of a tentacle\footnote{Note that if the portion of diagram involved in the $\Rd^-$ is attached to a piece whose projection is the same as a tentacle, but with ``wrong'' crossings, then the diagram does not fit in a triangle.}.

For the second case consider Figure \ref{fig:a5triangle}: we apply the same argumentation of case $1)$. Since the diagrams $D_1$ and $D_2$ are equivalent the ``heart shaped'' configuration in $D_2$ must appear somewhere in $D_1$. Moreover, since the diagrams coincide out of the portions drawn in the figure, the presence of a heart in $D_1$ \emph{far away} from the portions drawn would imply the existence of an identical heart somewhere in $D_2$, and we would still have one more heart in $D_2$ than in $D_1$. The same argument of case $1)$ (as in Figure \ref{fig:tentacoliappearing}) works if we assume that the heart appears after undoing two $\Omega_1$ moves.
\begin{figure}[h!]
\includegraphics[width=8cm]{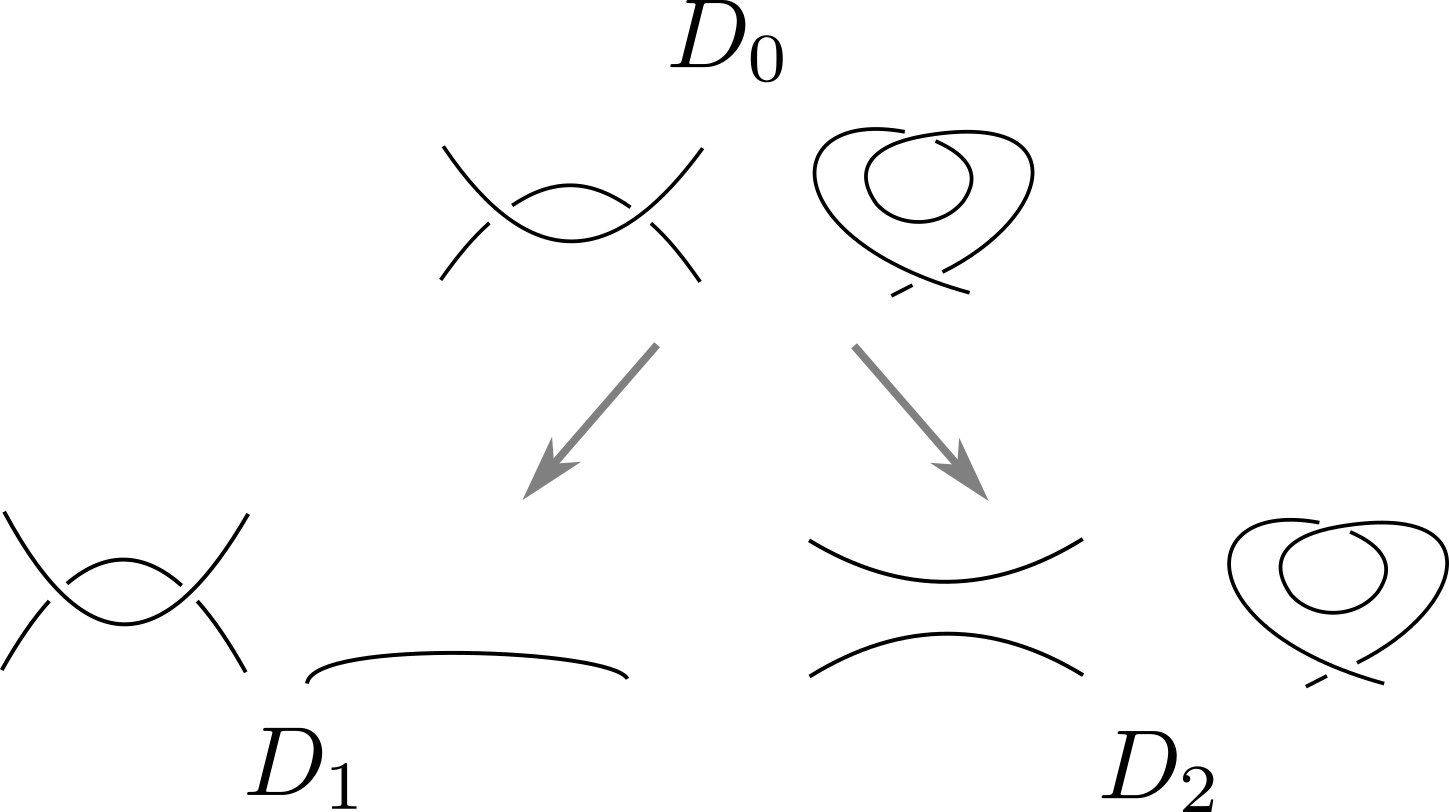}
\caption{$D_1$ is the diagram we obtain after performing the two $\Omega_1^-$ moves: together they cancel the heart configuration appearing in $D_0$. $D_2$ is the result of undoing the $\Omega_2$ move in the left-hand portion of diagram in $D_0$.}
\label{fig:a5triangle}
\end{figure}
It follows that the only possibility is the one depicted in Figure \ref{fig:a3triangletris}.
\begin{figure}
\includegraphics[width=8cm]{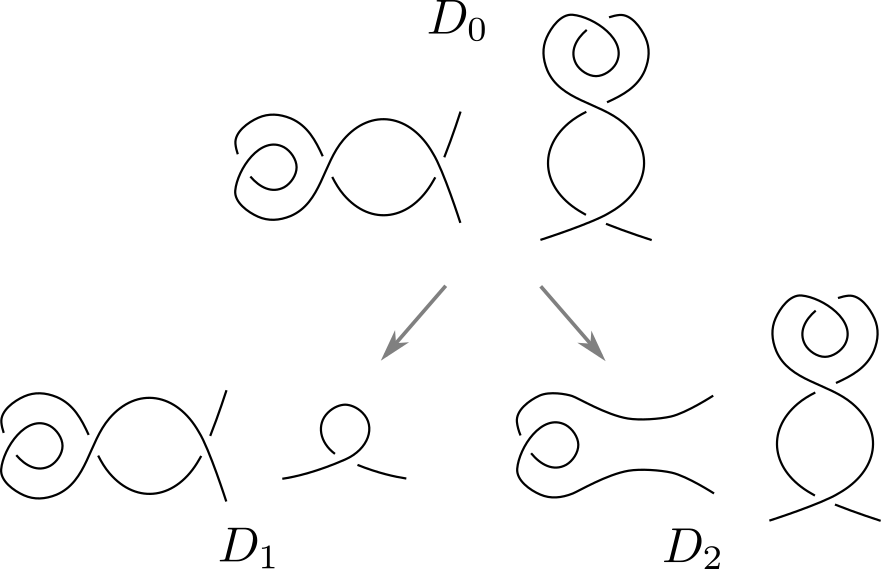}
\caption{$D_1$ is the diagram we obtain after performing the two $\Omega_1^-$ moves: together they cancel the heart configuration appearing in $D_0$, leaving a curl. $D_2$ is the result of undoing the $\Omega_2$ move in the left-hand portion of diagram in $D_0$.}
\label{fig:a3triangletris}
\end{figure}

We can however prove that in this case $D_1$ and $D_2$ can not be equivalent diagrams, and thus exclude it. To this end, consider the blackboard framing of the projection: there are two possibilities to be considered, since we can draw the framing curve on either side of the diagram. Then, since we do not know how the portions of diagrams involving the moves are positioned with respect to each other, we need to consider four different cases, all shown in Figure \ref{fig:a3bistriangle}.
\begin{figure}
\includegraphics[width=13cm]{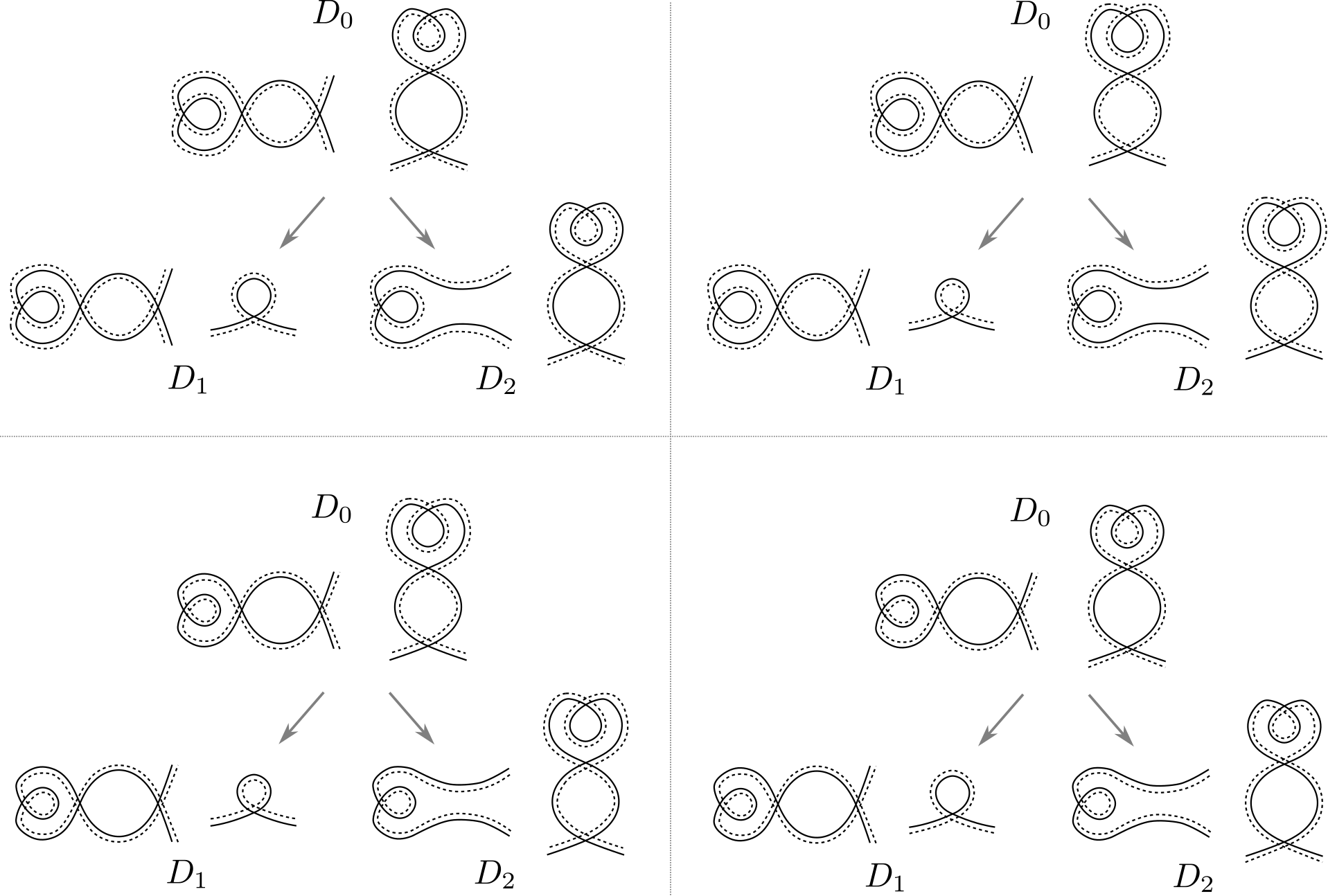}
\caption{The four possible choices for the blackboard framing.}
\label{fig:a3bistriangle}
\end{figure}
It is easy to argue that $D_1$ and $D_2$ can not be equivalent, since the number of curls having the blackboard framing ``inside'' is different in all four cases.

We are now left with case $3)$ from Figure \ref{fig:cases}. As usual, it is convenient to have in mind all the diagrams involved in the triangle, as in Figure \ref{fig:a5triangle9}.
\begin{figure}
\includegraphics[width=11cm]{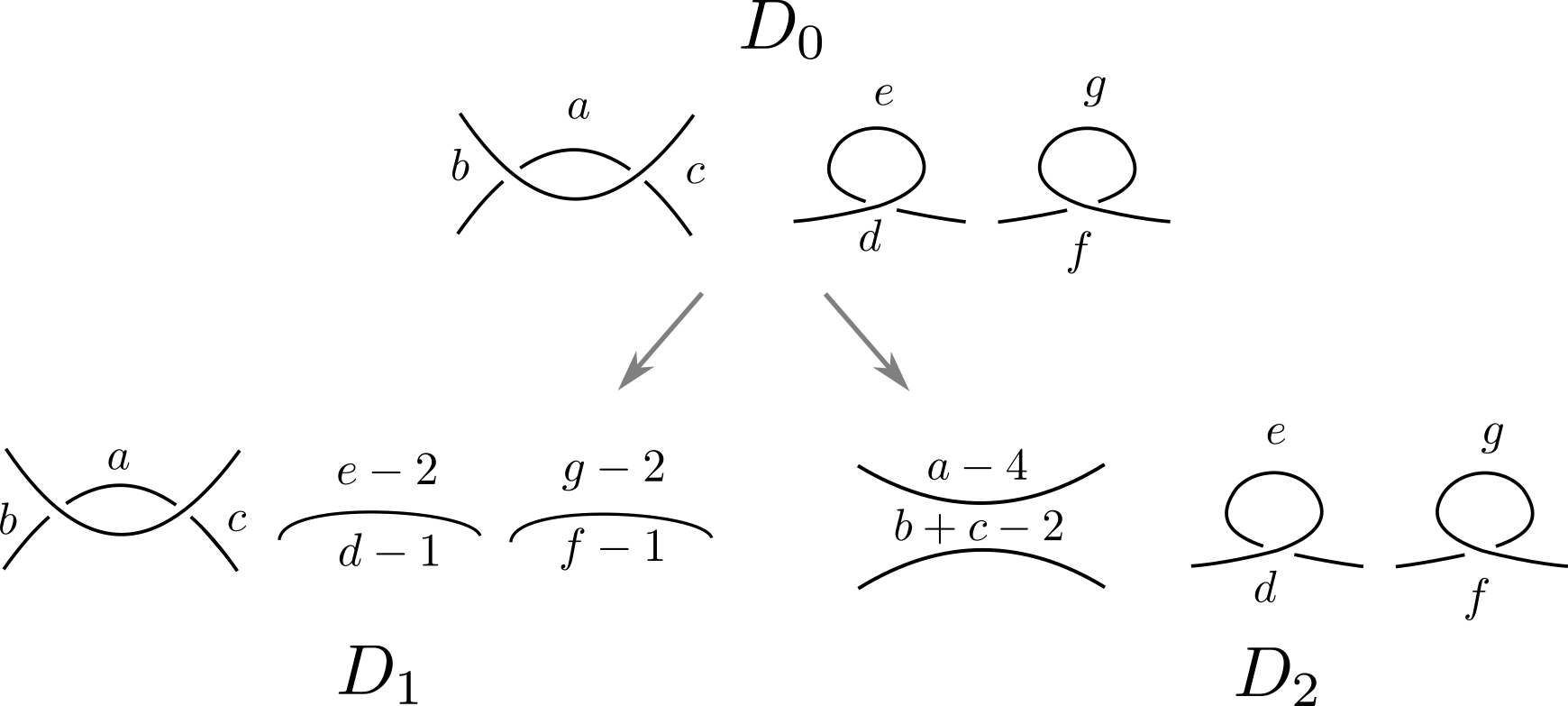}
\caption{Lowercase letters indicate the number of edges in each region. Keep in mind that, even if in the picture all the regions are depicted as different, some of them might coincide.}
\label{fig:a5triangle9}
\end{figure}

From Figure \ref{fig:a5triangle9} it is apparent that there are two more visible $1$-regions in $D_2$ than in $D_1$: since by hypothesis the diagrams are equivalent, there must be two curls in $D_1$ as well. Suppose by contradiction that the $\Omega_2$ is not a tentacle (that is $b,c \neq 1$).
Then, the straight lines in $D_1$ left by undoing the $\Omega_1$ moves must be part of two curls. If we assume  (Figure \ref{fig:a5triangle9}) that the regions touching the curls in $D_0$ are different, this means that at least two among $d-1,e-2,f-1$ and $g-2$ must be equal to $1$. Since the cases $(e,d) = (3,2)$ and $(f,g) = (2,3)$ are impossible, we are in one of the cases described in Figure \ref{fig:a5trianglebis}.
\begin{figure}
\includegraphics[width=9cm]{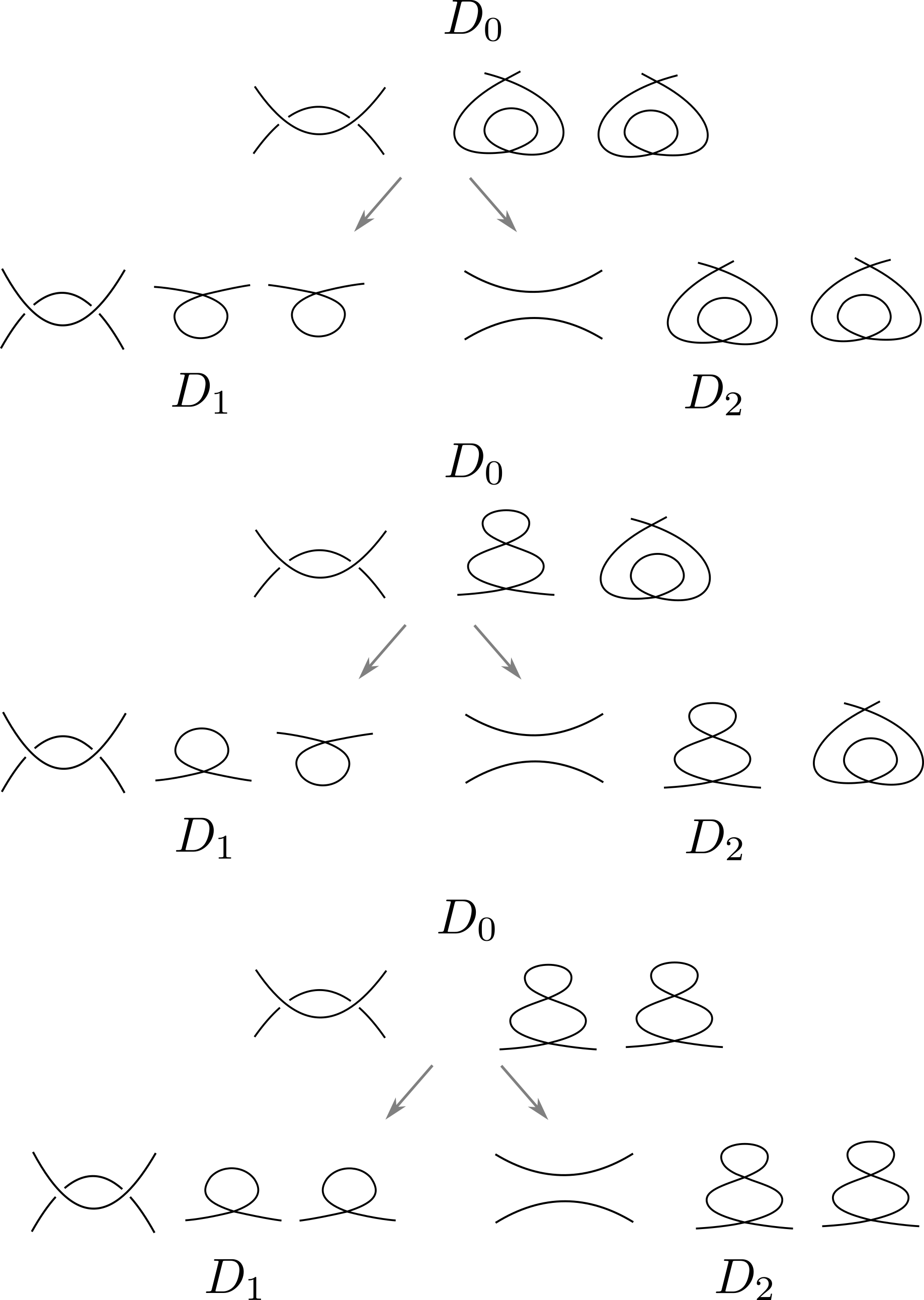}
\caption{The three possible kinds of triangles, assuming that the regions touching the curls undone by the $\Omega_1$ moves do not coincide.}
\label{fig:a5trianglebis}
\end{figure}

Before dealing with the configurations described in Figure \ref{fig:a5trianglebis}, we need to consider the cases in which some of the regions touching the curls coincide, keeping in mind that we are assuming that the $\Omega_2$ is not a tentacle move. We have the following possibilities (capital letters denote regions, as in Figure \ref{fig:cases}):
\begin{itemize}
\item[I] $A=B$ and $C=D$;
\item[II] $A=D$ and $B=C$;
\item[III] $A=B$ and $C\neq D$;
\item[IV] $C=D$ and $A \neq B$;
\item[V] $A=D$ and $B\neq C$;
\item[VI] $C=B$ and $A \neq D$.
\end{itemize}
Note that the upper and lower regions left by undoing a curl can not be both $1$-regions. Thus, cases I and II are straightforward to exclude, since if the regions coincided pairwise it would be impossible to recover two $1$-regions from the straight lines left by undoing the starting curls.
For the same reason, in the third case the only way to have two curls left after the $\Omega_1^-$ moves have been performed is to have a $2$-region below each $1$-region in $D_0$. Thus, case III fits in the bottom configuration described in Figure \ref{fig:a5trianglebis}. Similarly, in case IV, we would necessarily have both the curls in $D_0$ lying inside a $4$-region, forming an \emph{heart} and fitting in the top case shown in Figure \ref{fig:a5trianglebis}.
The latter two cases are symmetric, and it is enough to discuss only the first one. Again, since it is impossible to have both the upper and lower region left by undoing a curl as $1$-regions, it follows that the $\Omega_1^-$ moves must be performed in portions of diagrams identical to the ones drawn in the middle case of Figure \ref{fig:a5trianglebis}.

Let's now discuss carefully Figure \ref{fig:a5trianglebis}.
Consider the configuration at the top of the figure: since the diagrams $D_1$ and $D_2$ are equivalent, the \emph{heart} configurations in $D_2$ must appear somewhere in $D_1$. Moreover, since they coincide out of the portions drawn in the figure, and since using again a recursive argument we can exclude that they are created by undoing the curls in $D_0$, the only possibility is that these hearts are attached to the $\Omega_2$-portion in $D_1$, as shown in Figure \ref{fig:triangolocuori}.

\begin{figure}
\includegraphics[width=7cm]{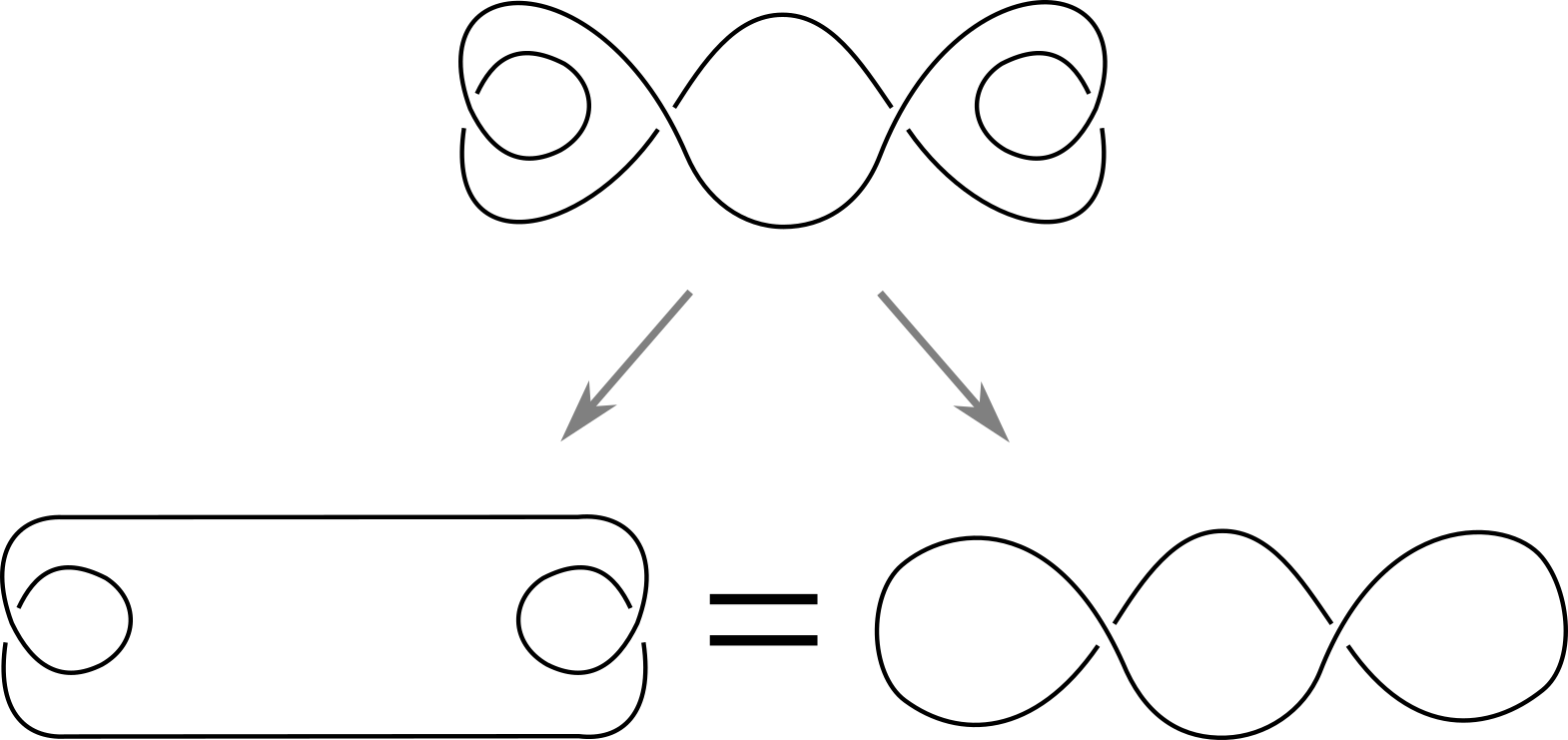}
\caption{A triangle for the unknot fitting in the family described in the statement of the Theorem.}
\label{fig:triangolocuori}
\end{figure}

Note that even if in this case $D_1$ and $D_2$ turn out to be equivalent, the diagrams represent the trivial knot, and more precisely they fit in the family described in the statement of the Theorem. \\
Notice that  this can only happen if we are working with diagrams on $S^2$; if we are working with planar diagrams instead, this configuration does not fit in a triangle.
Now, call \emph{generalised tentacles} the configurations formed by two succesive $\Omega_1$ moves made one on top of the other, as appearing in $D_2$ and $D_0$ on the bottom of Figure \ref{fig:a5trianglebis}. If the crossings are such that the configurations form tentacles, then this implies (as in case $1)$ of Figure \ref{fig:cases}, that the $\Omega_2$ is in fact a tentacle move.\\
Otherwise, by using a similar recursive argument as before, together with the fact that the upper and lower regions involved in the $\Omega_2$ move coincide, we can exclude both the possibility that the configurations appear in $D_1$ by performing the $\Omega_1^-$ moves, and that they appear somewhere \emph{far away} from the portions of diagram shown. Thus, we see that the only possibility for $D_1$ and $D_2$ to be equivalent occurs when the generalised tentacles are attached\footnote{Or they are part of longer generalised tentacles attached to the $\Omega_2$-portion of $D_1$.} to the $\Omega_2$-portion of $D_2$, forming a diagram for the unknot fitting in the family described in the statement of the Theorem (see Figure \ref{fig:controesempio}). Notice that the triangle in $\mathcal{G}_S(\bigcirc)$ involving the heart configurations described before is a special case of this situation. \\
Finally, we are left with the middle configuration in Figure \ref{fig:a5trianglebis}. As usual, since $D_1$ and $D_2$ are equivalent by hypothesis, the tentacle configuration in $D_1$ has to appear somewhere in $D_2$ as well. Assuming that the $\Omega_2$ is not a tentacle move, since the diagrams coincide out of the portions drawn, using yet again a recursive argument we can exclude that the tentacle is created by undoing the curls in $D_0$; hence the only way to have a tentacle in $D_1$ is the one shown in Figure \ref{fig:ultimo2}.
\begin{figure}
\includegraphics[width=9cm]{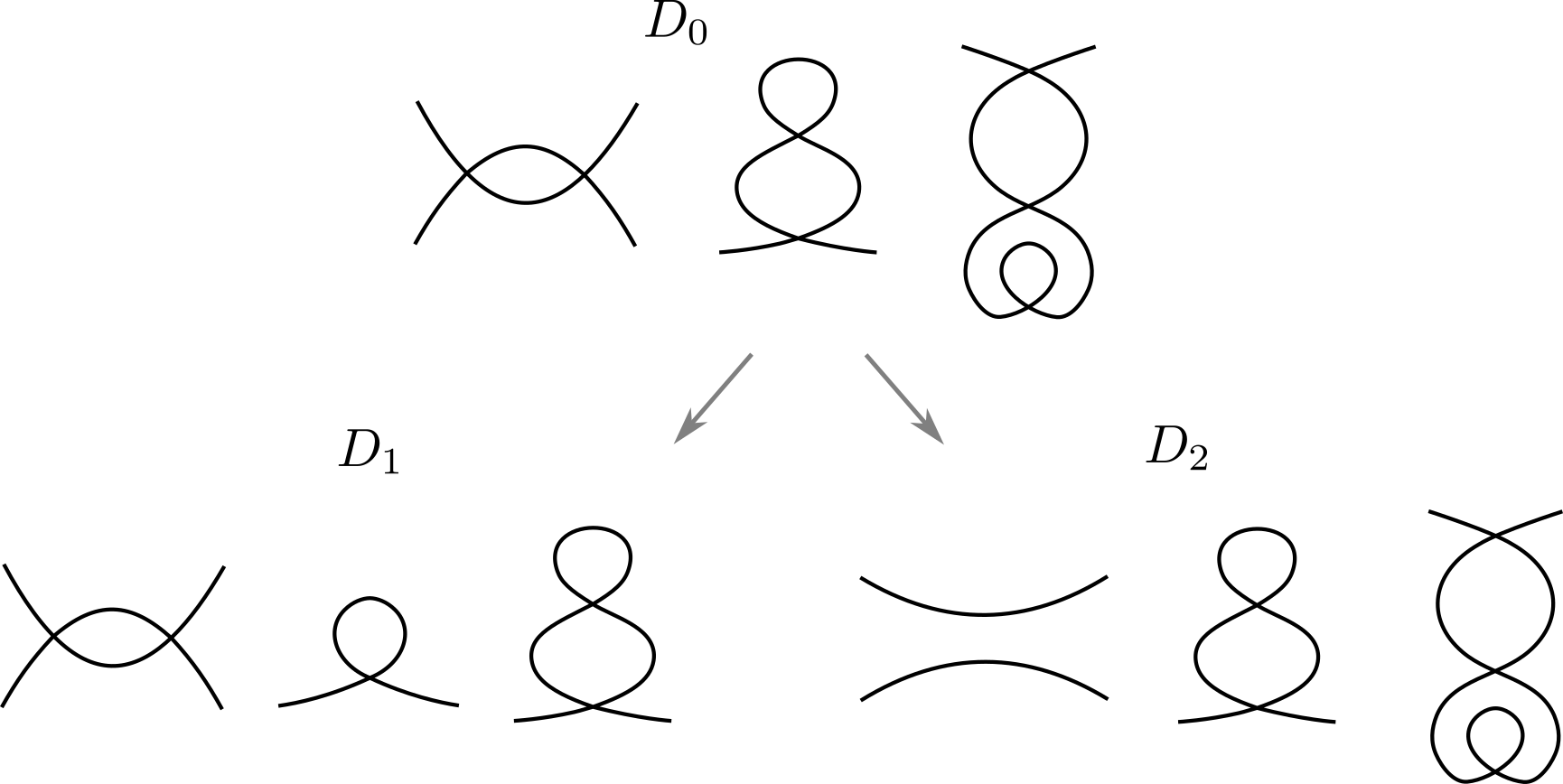}
\caption{$D_1$ is the diagram we obtain after performing the two $\Omega_1^-$ moves: with the first one we cancel the curl inside the heart, while the other has the effect of decreasing the height of the left tentacle by $1$. $D_2$ is the result of undoing the $\Omega_2$ move in the left-hand portion of diagram in $D_0$.}
\label{fig:ultimo2}
\end{figure}
We can however exclude this case as well by adding the blackboard framings. In Figure \ref{fig:ultimo} two of the possible choices of framings are displayed: in both cases $D_1$ and $D_2$ are non-equivalent diagrams, since the framings do not coincide on the tentacles or in the $1$-regions left.
\begin{figure}
\includegraphics[width=9cm]{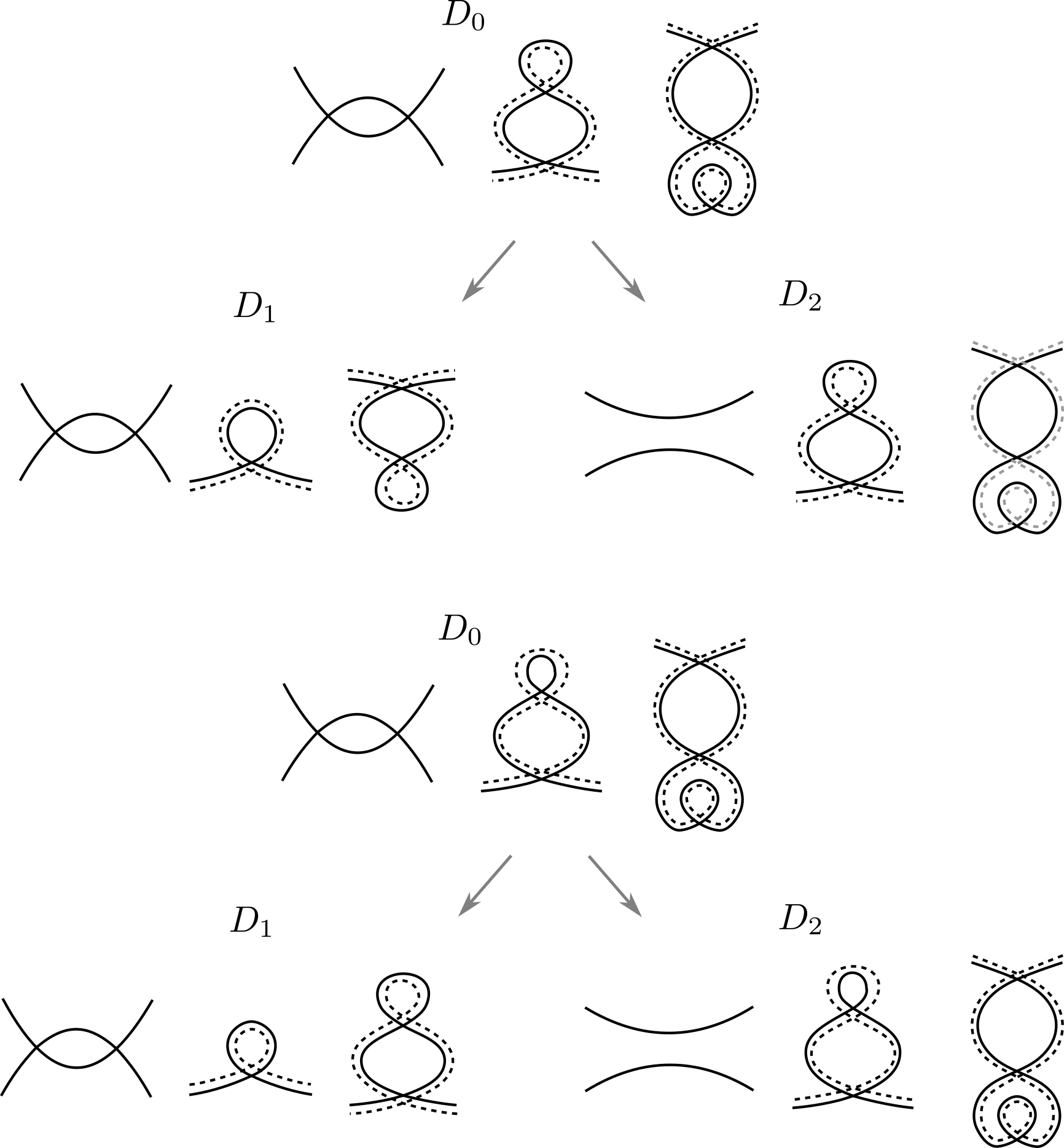}
\caption{Two of the four possible choices of framing. The remaining two can be treated in the exact same way.}
\label{fig:ultimo}
\end{figure}
\end{proof}

\begin{rmk}\label{tuttotrannebanale}
In what follows, unless otherwise specified, all the results will hold for every knot type with the exception of the unknot $\bigcirc$.

\end{rmk}

For each diagram $D \in \diag$, $S(D)$ consists of triangles (possibly attached to one another) with one vertex in $D$, and edges emanating from $D$. Each of these might be a multi-edge.
If we want to study the possible configurations in $S(D)$ involving triangles, by Theorem \ref{prop:lung3}, we just need to restrict to those containing at least one $\Omega_1$ move; various possibilities involving one or more curls/tentacles are shown in Figure \ref{fig:configr1} and \ref{fig:configrmultiple}.
\begin{figure}
\includegraphics[width=14cm]{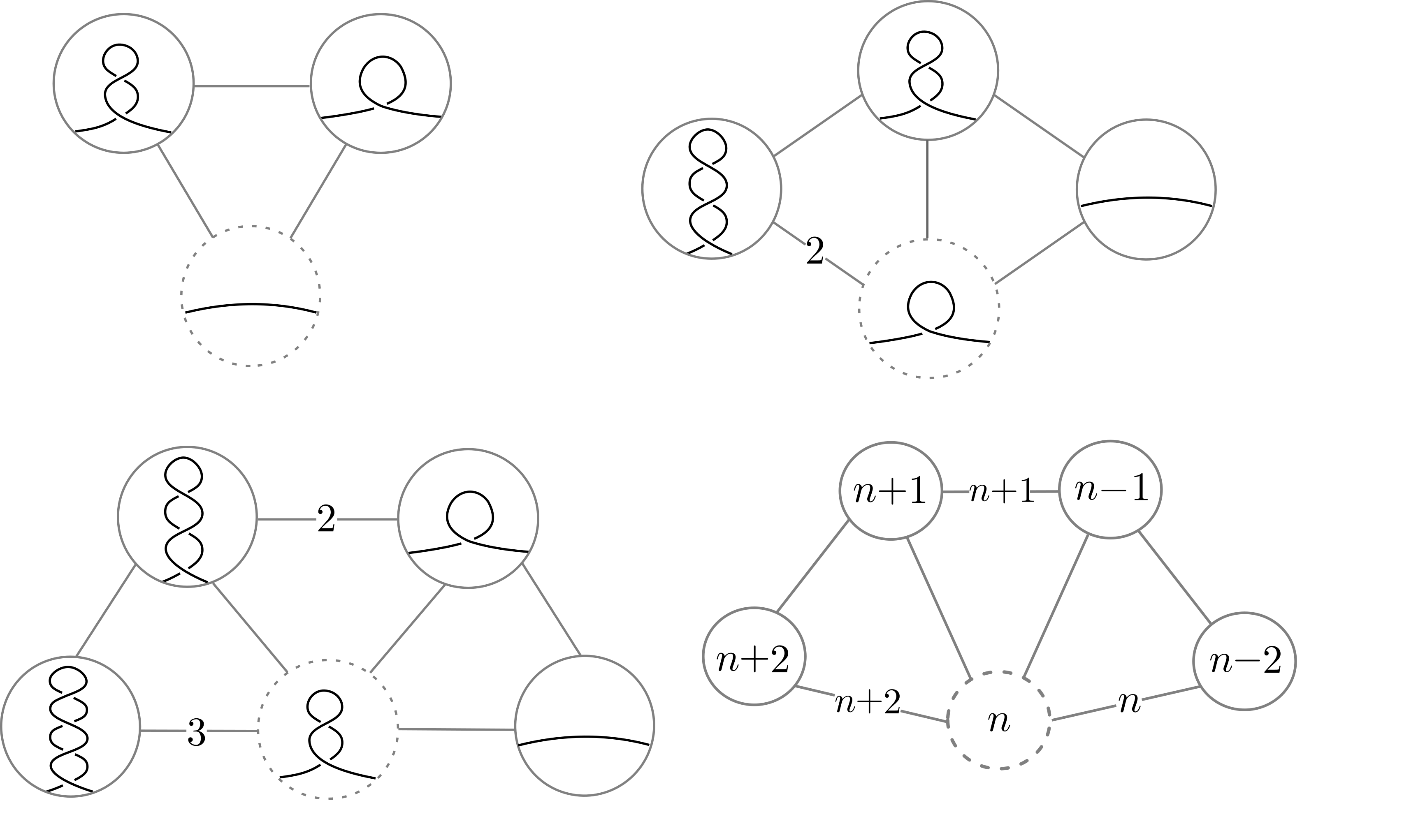}
\caption{Some of the possible configurations in $S(D)$ involving at least one $\Omega_1$ move ($D$ is contained in the dotted circles). The top-left one is present in any $S(D)$, while the others can be found whenever there is a $\Omega_1^-$ (top-right), an height $1$ tentacle (bottom-left), or a tentacle of height $n \ge2$ (bottom-right). Numbered edges denote the valence of the corresponding multi-edge.}
\label{fig:configr1}
\end{figure}
\begin{figure}
\includegraphics[width=14cm]{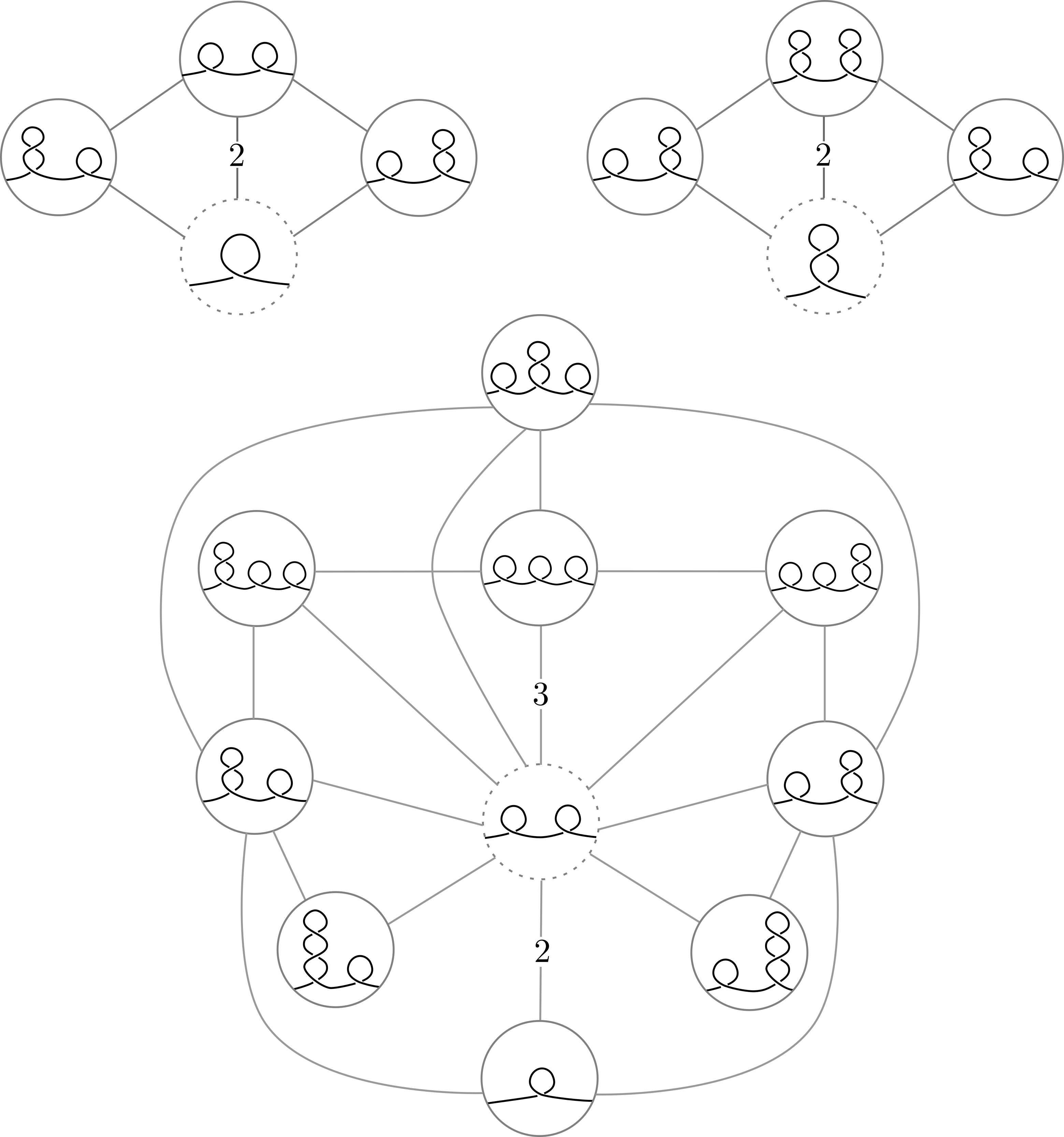}
\caption{Other qualitatively different triangle configurations formed by $\Omega_T$-$\Omega_1$-$\Omega_1$  can be found whenever there are multiple curls or height $1$ tentacles on the same arc.}
\label{fig:configrmultiple}
\end{figure}

So we have a complete description of the short paths that can appear in $\sfera$; note that it makes less sense to pursue a systematic study of longer ($\ge 3$) cycles, since any pair of ``distant'' moves on a diagram produces a cycle of length $4$. In the following we are going to examine more closely the properties and shapes of the various triangles that have been produced during the proof of Theorem \ref{prop:lung3}. This technical analysis is going to be crucial in the proof of the results leading to Theorem \ref{thm:completeness}.

\begin{defi}
We will call  a triangle \emph{normal} if it is of the form described in Figure \ref{fig:triangoloreid}, meaning that all the Reidemeister moves are performed  locally on the same arc.
\end{defi}

\begin{lemma}\label{lemmanormale}
$p_1(D) = 0$ if and only if all the triangles in $S_1(D)$ are normal. 
\end{lemma}
\begin{proof}
If $p_1(D)\neq0$, then there are at least two triangles sharing a $\Omega^+_1$ edge, as shown in the top-right part of Figure \ref{fig:configr1}. This implies that there are at least two non-normal triangles, since one can perform the first $\Omega^+_1$ move on either side of the pre-existing twirl, and complete this edge to a triangle by performing the successive $\Omega^+_1$ and $\Omega^-_T$ on the pre-existing twirl.\\
Viceversa, suppose that $p_1(D)=0$. Thanks to Theorem \ref{prop:lung3} we know that all triangles are made up by $\Rs$-$\Ru$-$\Ru$; moreover, every $\Omega_1$ and $\Omega_T$ is part of at least one triangle. We wish to understand all the possible configurations forming
a triangle. To this end, we can use Figure \ref{fig:cases}, substituting\footnote{We can assume that the $\Omega_T$ move happens on the top of the tentacle.} with $\Omega_T$ configurations the $\Omega_2$s, as in Figure \ref{fig:casestriangle2}.

\begin{figure}
\includegraphics[width=5cm]{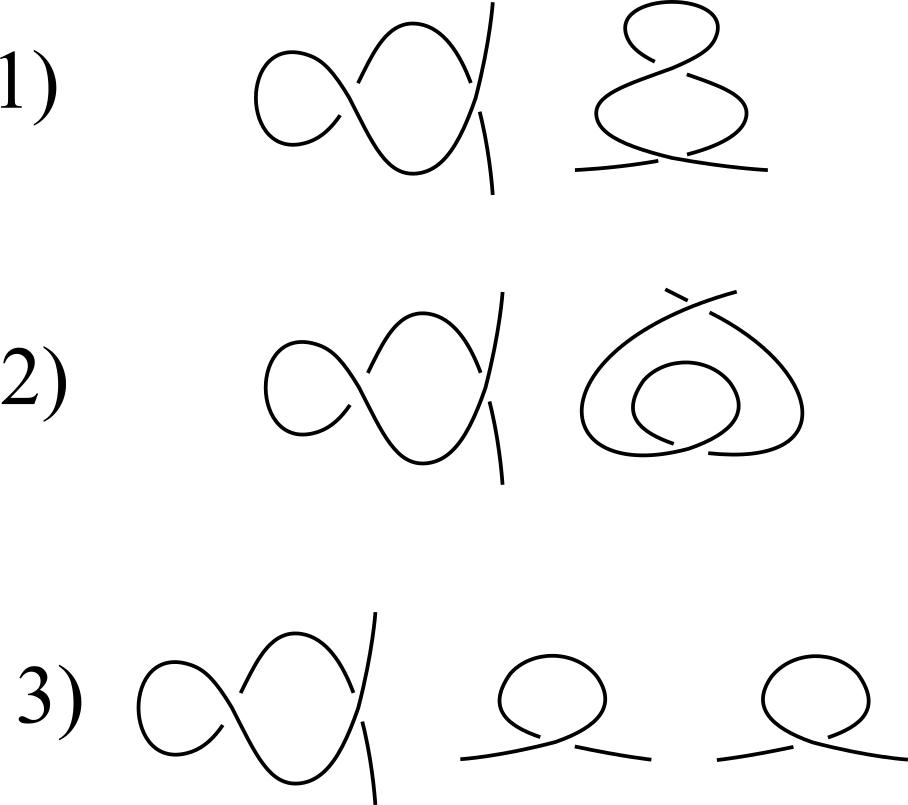}
\caption{In each row the portions of $D_0$ involved in the $\Omega_*$ moves are shown. In the first and second row, we assume that the curls undone by the $\Omega_1$ moves do not both appear in the diagram.}
\label{fig:casestriangle2}
\end{figure}

Since $p_1(D)=0$, we can exclude the occurrence of cases $2$) and $3$) of Figure \ref{fig:casestriangle2}. In fact, in each of these triangles, the diagram with lower crossing number admits at least one $1$-region. Let's suppose that there exists a non-normal triangle fitting in case $1$) of Figure \ref{fig:casestriangle2}. By definition, this means that the moves are not performed on the same arc. Then, in the lower crossing number diagram, there is at least a $1$-region (see Figure \ref{fig:nonchiamarelefigureuguali}), contradicting the hypothesis $p_1(D)=0$.
\end{proof}

\begin{figure}
\includegraphics[width=11cm]{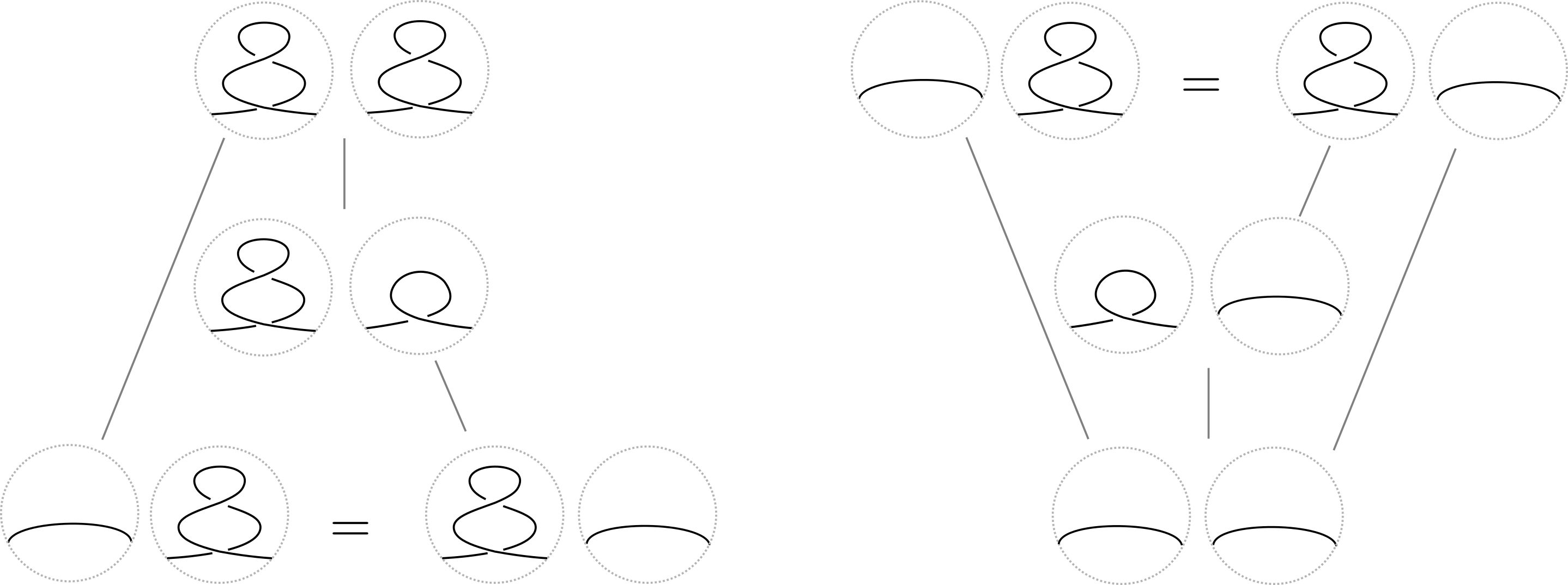}
\caption{On the left, a non-normal triangle fitting in case $1)$. The diagrams with the lowest crossing number (on the bottom) are identified. On the right, a normal triangle fitting in case $1)$. Here the diagrams with the greatest crossing number are identified, and there is a $\Omega_T$ multi-edge. }
\label{fig:nonchiamarelefigureuguali}
\end{figure}

\begin{rmk}\label{rmk:triangolinormali}
If $p_1(D) \neq 0$, then more complicated triangles appear. We show an example of a non-normal triangle fitting in case $1)$ of Figure \ref{fig:cases} in Figure \ref{fig:triangoloperiodico1}.

\begin{figure}
\includegraphics[width=9cm]{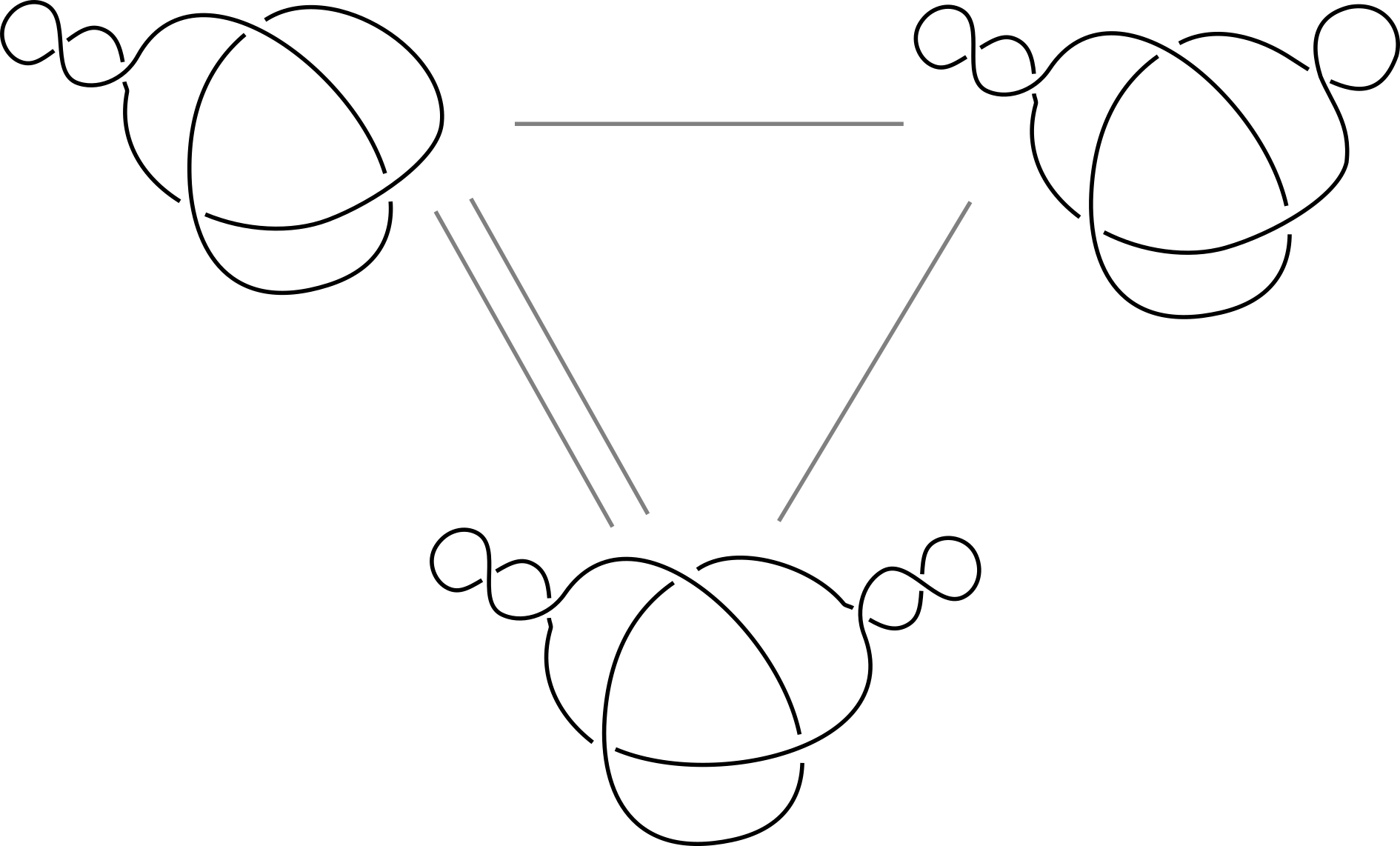}
\caption{A triangle for a periodic knot fitting in case $1)$. The left and lower vertices are connected by a multi-edge.}
\label{fig:triangoloperiodico1}
\end{figure}

In what follows we are going to analyse what happens in the remaining cases. In fact, case $2)$ of Figure \ref{fig:cases} can be excluded as in the proof of Theorem \ref{prop:lung3}.

It is convenient to divide the investigation on triangles fitting in case $3)$ of Figure \ref{fig:cases} in two subcases (denoted by $3$a and $3$b respectively), differing in whether or not one of the $\Omega_1^-$ moves happens on the top part of the tentacle undone by the $\Omega_T$. If it does, then we are in the situation described in Figure \ref{fig:triangoloperiodico3}, and we notice that the diagram with the lowest crossing number contains at least one $1$-region; an example of a non-normal triangle fitting in case $3a)$ is shown in Figure \ref{fig:mmperiodico}.
\begin{figure}
\includegraphics[width=6cm]{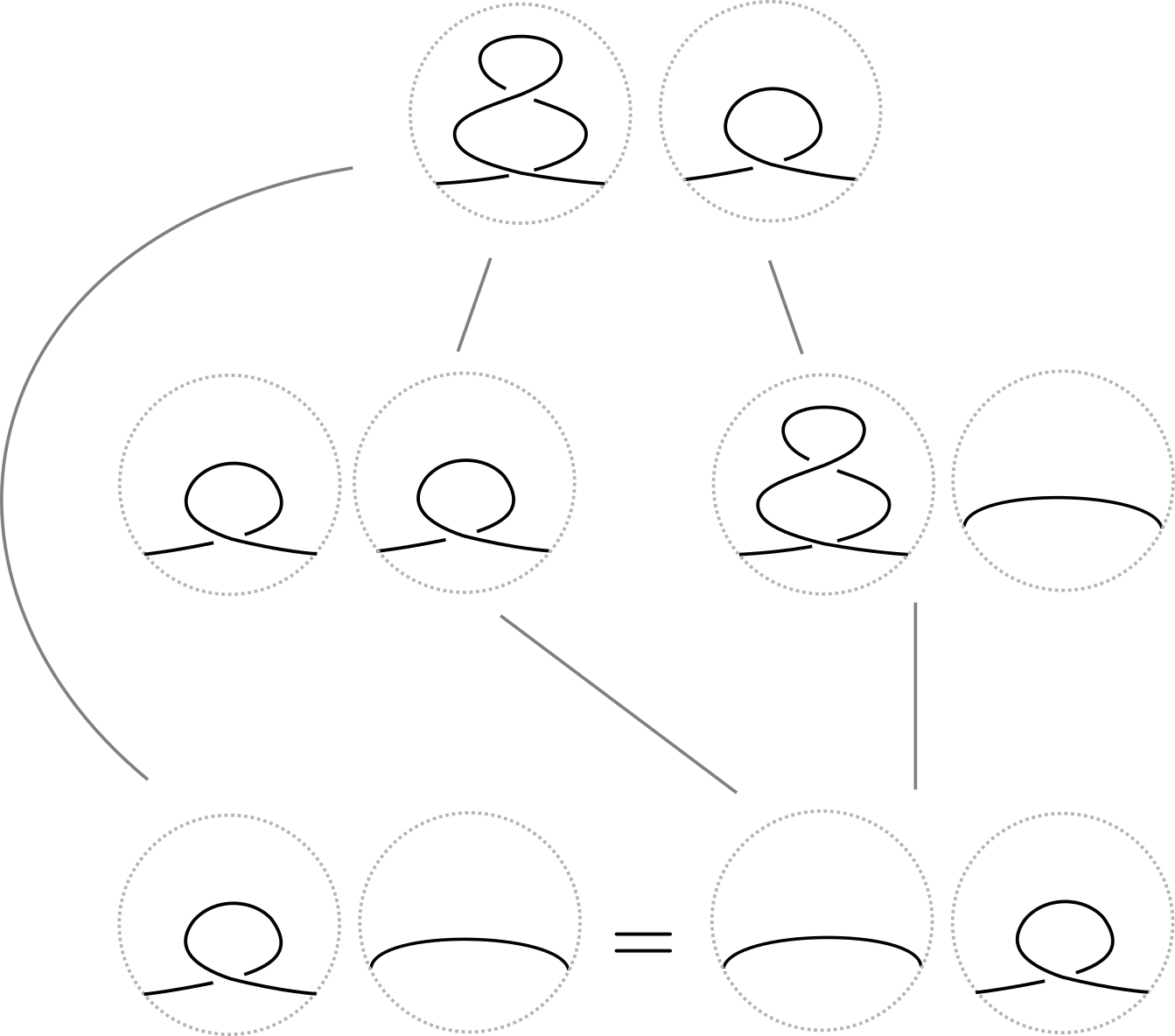}
\caption{A non-normal triangle fitting in case $3a$). Notice that the diagrams with the lowest crossing number (on the bottom) are identified and present at least a curl.}
\label{fig:triangoloperiodico3}
\end{figure}
\begin{figure}
\includegraphics[width=5cm]{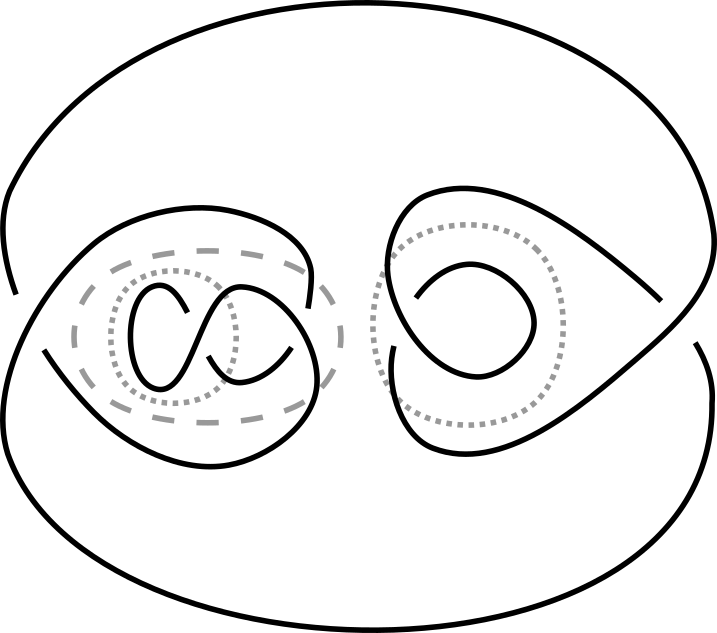}
\caption{A triangle for a periodic (un)knot, fitting in case $3a)$. Dotted circles enclose the $\Ru^-$s, and the dashed one the $\Rs^-$. This specific example was pointed out by M. Marengon.}
\label{fig:mmperiodico}
\end{figure}
Finally, if both the curls undone by the $\Omega_1$ moves are not the top part of the tentacle, then the diagrams appear as in Figure \ref{fig:triangoloperiodico2}.
\begin{figure}
\includegraphics[width=7cm]{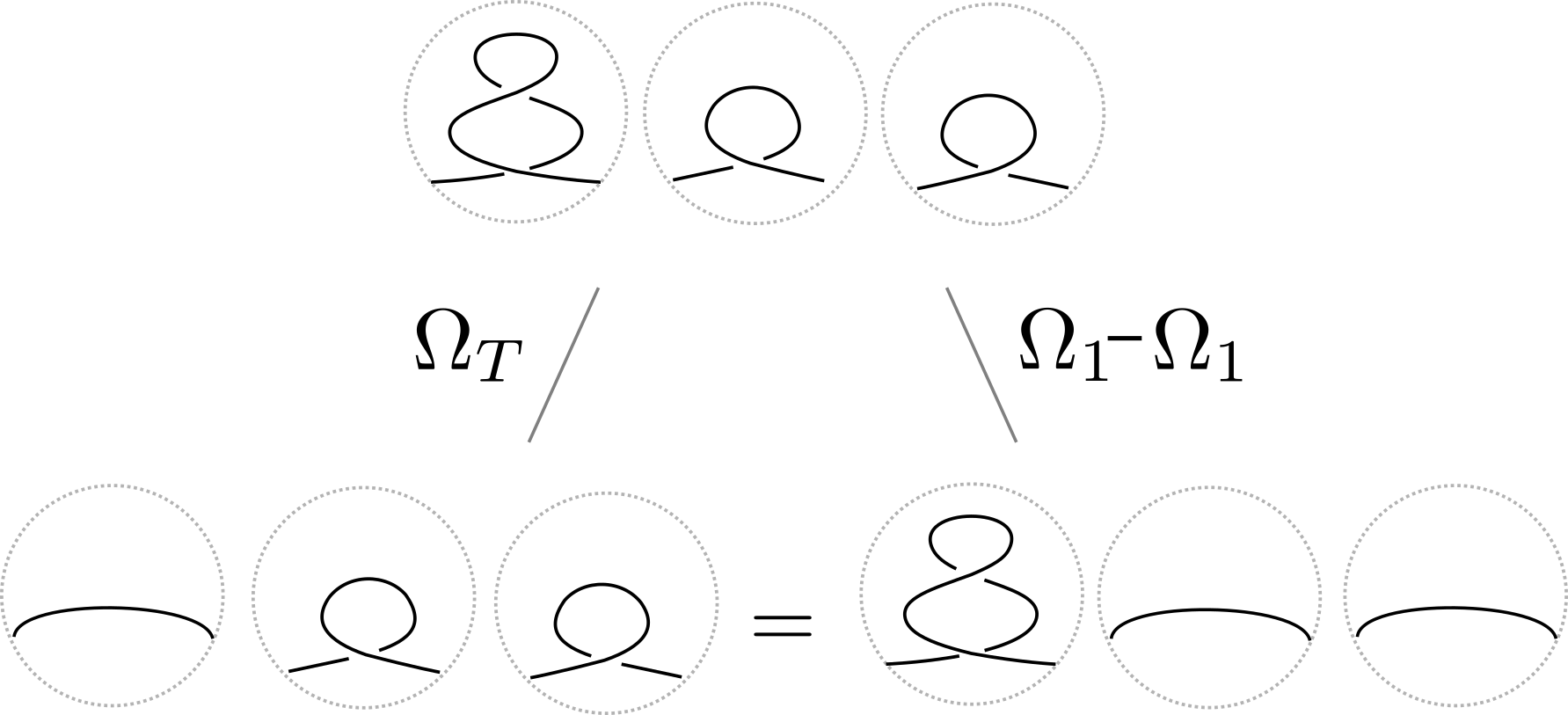}
\caption{A non-normal triangle fitting in case $3b)$. Notice that the diagrams with the lowest crossing number (on the bottom) are identified, and present at least a tentacle configuration. See also Figure \ref{fig:triangoloperiodico5} for an explicit example.}
\label{fig:triangoloperiodico2}
\end{figure}

Again, we can conclude that the diagram with the lowest crossing number presents a tentacle configuration. 
We show an example of a non-normal triangle fitting in case $3b)$ in Figure \ref{fig:triangoloperiodico5}.
\begin{figure}
\includegraphics[width=5.5cm]{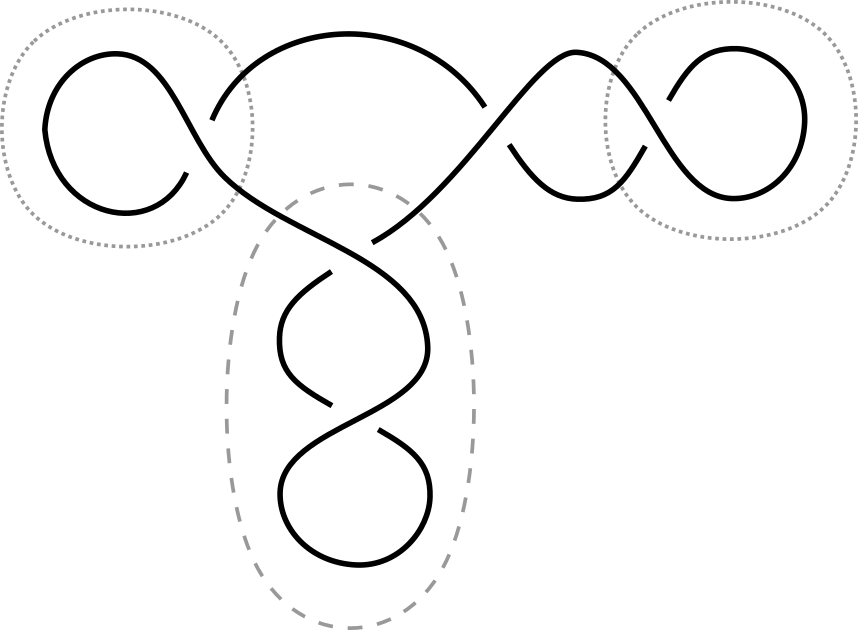}
\caption{A triangle for a periodic (un)knot, fitting in case $3b)$.}
\label{fig:triangoloperiodico5}
\end{figure}
In all the non-normal cases above two diagrams are identified, and this implies either the existence of a periodicity for the knot, or that the moves happen on the same edge, involving adjacent curls or tentacles.

\end{rmk}

\begin{lemma}\label{lemma:miniprop}
Given any knot diagram $D$, there exists an arc on $D$ such that performing either a $\Ru^+$ or a $\Rs^+$ belonging to a normal triangle, the resulting diagrams are non-periodic. Moreover, if we perform another $\Rs^+$ on the top of the tentacle created, the diagram obtained and all of the diagrams in its radius $1$ ball are non-periodic.
\end{lemma}
\begin{proof}
This follows from the fact that the height $h$ tentacle configurations are permuted by any symmetry of the diagram, so if there's only one the diagram can not be periodic. So, just take any diagram $D$; if $p_1(D) =0$, then performing any $\Ru^+$ or the $\Rs^+$ it is paired with will produce a non-periodic diagram. If instead $D$ contains at least one curl, choose the one which appears on the top of the highest tentacle, and perform there the $\Ru^+$/$\Rs^+$ pair (with appropriate signs). Since the new diagram will have only one tentacle of maximal height we can conclude. Finally, if we further increase the length of the tentacle, we are sure that we are at least at distance $2$ from any periodic diagram.
\end{proof}

Unlike the Gordian graph, the Reidemeister graphs are locally finite, even though the valence is not uniformly bounded (Remark \ref{rmk:crossingnumber}). The first invariant we can extract from them is in some sense a measure of the minimal complexity of the diagrams of $K$:
\begin{defi}\label{defi:complessita}
Let $v(D)$ denote the valence of the vertex $D$. The \emph{diagram complexity} of a knot $K$ is $$\delta(K) = \min_{D \in \diag} v(D) .$$
If $v(D) = \delta(K)$ we say that $D$ is a \emph{minimum}. We also define $\# \delta (K)$ as the number of minima of $\reid$; if a knot type $K$ is such that $\#\delta(K) = 1$, we call $K$ \emph{simple}. Both $\delta$ and $\#\delta$ are $\N$-valued knot invariants. There is of course an identical definition for $\sfera$; we denote by $\delta_S$ and $\#\delta_S$ the corresponding invariants.
\end{defi}
\begin{rmk}\label{rmk:crossingnumber}
We will postpone the proof that $\#\delta(K)$ is in fact well defined to Lemma \ref{cor:welldefdelta}. Note that the diagram complexity is not a function of the crossing number, as one might naively think. In Remark \ref{rmk:diagrammanonmin} we are going to provide some examples of this phenomenon. It is however true that, for a fixed knot type $K$, the valence becomes arbitrarily high as the crossing number of the diagrams representing $K$ increases.
\end{rmk}

Given a non-periodic diagram $D \in \diag$, one can enumerate the possible Reidemeister moves on $D$, in order to compute $v(D)$. We start by counting the possible number of $\Omega_1^+$ and $\Omega_T^+$. For each arc in $D$ we can perform $4$ $\Omega_1^+$ moves, as shown in Figure \ref{fig:possibilir1}, and the same holds for $\Omega_T^+$.
\begin{figure}[h!]
\includegraphics[width=8cm]{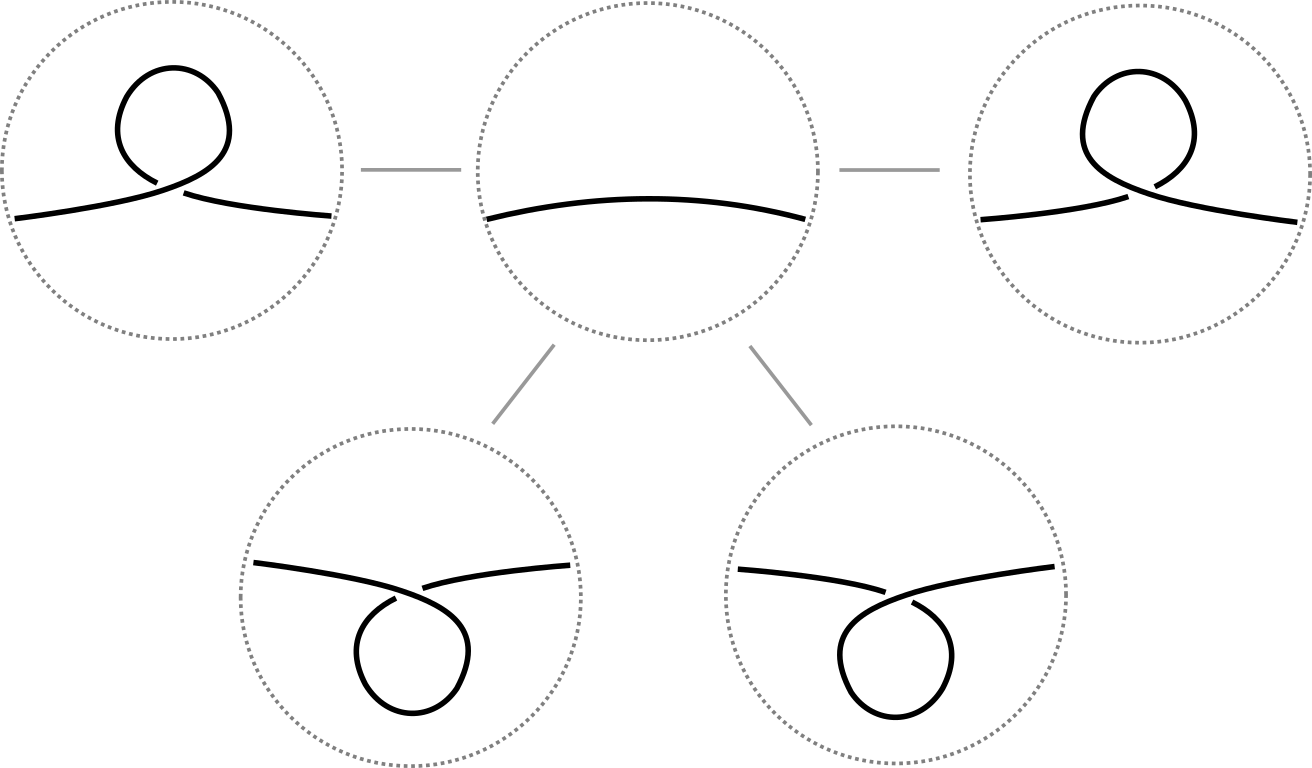}
\caption{The possible $\Omega_1^+$ moves that can be performed on each arc.}
\label{fig:possibilir1}
\vspace{0.5cm}
\includegraphics[width=10cm]{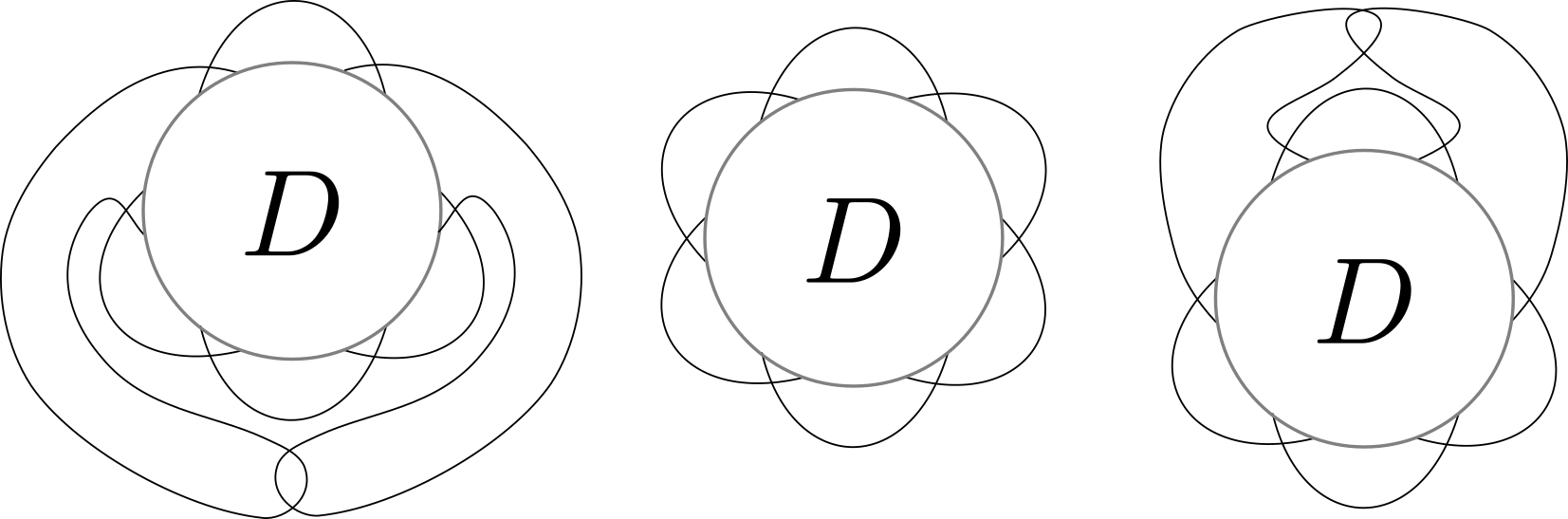}
\caption{The two non equivalent possibilities for a $\Omega_2$ move in the external zone.}
\label{fig:zonaext}
\end{figure}

When working with $\reid$, so diagrams on the plane, we must put a bit of care in counting $\Omega_2^+$ moves, since the number of such possible moves depends on whether we are in the ``external'' polygon or not.
If a polygon $P \in \R^2 \setminus D$ has $k$ edges, there are $2 \binom{k}{2} = k (k-1)$ possible\footnote{In the present discussion we find it convenient to blur a bit the distinction between $\Rd$ and $\Rs$, since we are only interested in the total count.} $\Omega_2^+$ moves we can perform in it (the factor of 2 comes from the two possible choices of which arc passes over the other).
In the external zone however we need to double the previous quantity, since there are two cases to be considered, as shown in Figure \ref{fig:zonaext}. So if we denote by $k_{ext}$ the number of edges of the external zone, we have an extra contribution of $k_{ext} (k_{ext} - 1)$. This extra term does not appear when working with diagrams on the $2$-sphere, as there is no preferential polygon.

Adding all up, we end with this rather unpleasant equation for the valence of a non-periodic planar diagram. Note that multi-edges do not create issues in the sum, as they are counted separately.

\begin{equation}\label{eqn:valenza}
\begin{aligned}
v(D) = 8\alpha(D) + \sum_{k \ge 2} p_k(D) k (k-1) + k_{ext}(k_{ext}-1) +\\ + \#\Omega_3(D) + \#\Omega_2^-(D) + \#\Omega_T^-(D) + \#\Omega_1^-(D)
\end{aligned}
\end{equation}

It follows from Equation \eqref{eqn:valenza} that the valence of any diagram is bounded from above by quantities depending only on the knot projection:
\begin{equation}\label{eqn:boundvalenza}
v(D) \le 8 \alpha(D) + p_1(D) + p_2(D) + p_3(D) + 2 \sum_{k\ge2} p_k (D) k (k-1).
\end{equation}
Equation \eqref{eqn:boundvalenza} is obtained by giving an upper bound on the possible $\Omega^-$ and $\Omega_3$ moves in terms of the number of edges of the regions interested by the moves (\emph{i.e.} on the number of $1$, $2$ and $3$-regions for $\Omega_1^-$, $\Omega_2^-$+$\Omega_T^-$ and $\Omega_3$ moves respectively).

Looking at Equation \eqref{eqn:valenza} we can obtain a lower bound as well, which allows to say that the valence grows at least linearly with the crossing number. Define $P(D)$, the maximal period of a non-trivial diagram $D$, as the maximal order of a finite group acting on the sphere (or the plane), preserving the diagram setwise\footnote{We need to exclude the trivial diagram of the unknot to ensure that $P(D)$ is in fact finite.}. Recall that if $K$ is not the unknot, then $K$ admits finitely many orders of periodicity (see \cite[Thm. 3]{period}).
\begin{lemma}\label{cor:lowerperiodico}
If $D$ is a non-trivial knot diagram with periodicity $P(D)$ (where $P(D)=1$ if $D$ is non-periodic) then
\begin{equation}\label{eqn:lowerboundlinear}
v(D) \ge \frac{8\alpha(D)}{P(D)}.
\end{equation}
\end{lemma}
This follows easily by observing that each fundamental domain for the periodic action must contain at least one arc.

Of course if $D$ is non-periodic, the lower bound
\begin{equation}\label{bounddasopra}
v(D) \ge 8\alpha(D) + \sum_{k \ge 2} p_k(D) k (k-1) + k_{ext}(k_{ext}-1)
\end{equation}
holds as well.

\begin{figure}[h]
\includegraphics[width=12cm]{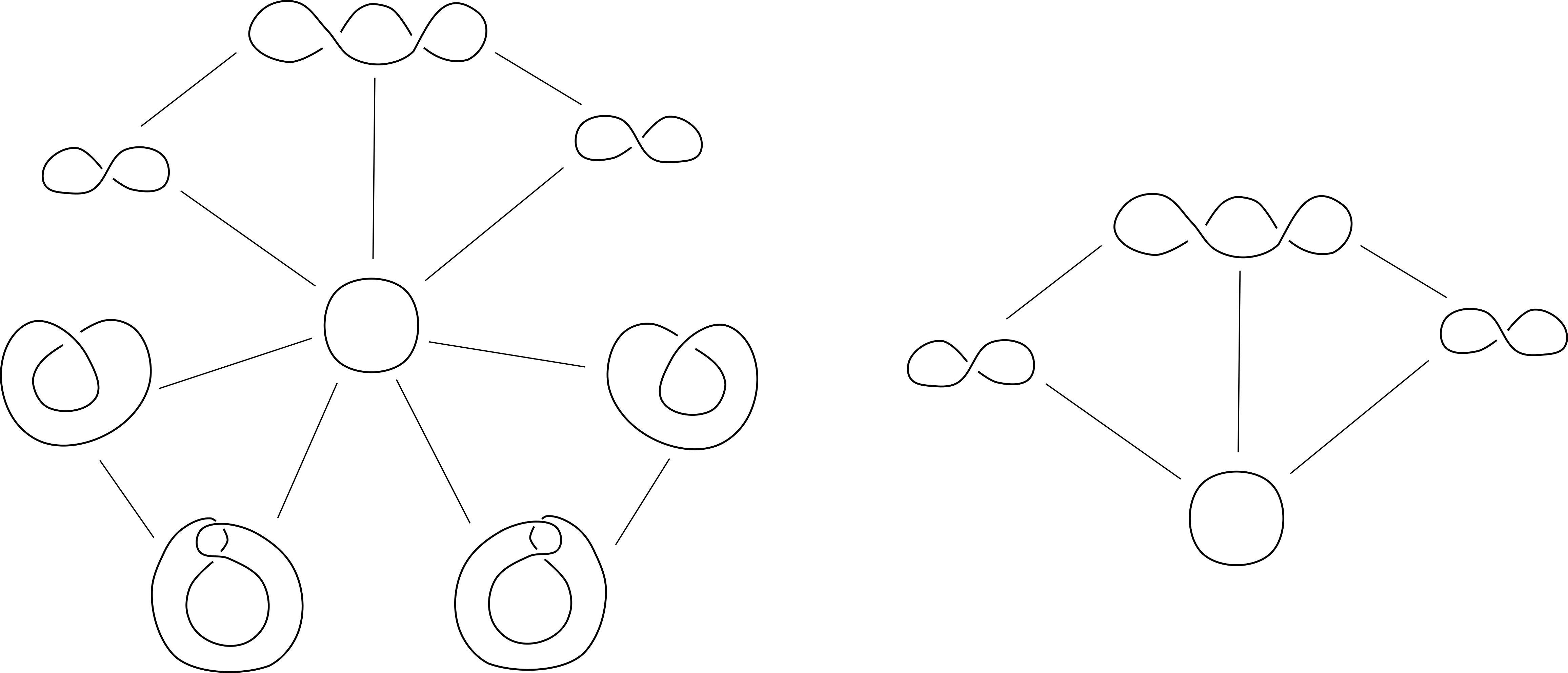}
\caption{The ball $S(\bigcirc)$ in the planar (left) and $S^2$ graphs (right).}
\label{unknot}
\end{figure}

\begin{prop}\label{valenzabanale}
The minimal valence $\delta_S $ detects the unknot $\bigcirc$.
\end{prop}
\begin{proof}
$\delta_{S}(\bigcirc)=3$, as shown in Figure \ref{unknot}, while if $K\neq \bigcirc$, then for every diagram $D$ representing $K$ we have $v(D) \geq 4$, since each fundamental domain for a periodic action must contain at least one arc (as in the proof of Lemma \ref{cor:lowerperiodico}), and for every arc there are at least $4$ (two $\Omega_T^+$ and two $\Omega_1^+$) possible moves.
\end{proof}

\begin{lemma}\label{lemma:pochi}
For each knot $K$, the number of vertices in $\reid$ or $\sfera$ whose valence is bounded by a constant is finite.
\end{lemma}
\begin{proof}
This follows immediately from the fact that there are only finitely many diagrams of a knot with crossing number bounded by a constant, finitely many periods for each knot, and by Equation \eqref{eqn:lowerboundlinear} the valence is bounded from below by a linear function in $cr(D)$.
\end{proof}
In particular, choosing $\delta(K)$ as the constant in the previous Lemma, we get:
\begin{cor}\label{cor:welldefdelta}
$\# \delta(K)$ is well defined.
\end{cor}

Following \cite{kauffman2006hard}, we call a diagram $D \in \diag$ \emph{hard} if  $\#\Omega_1^- (D) = \#\Omega_2^- (D) = \#\Omega_T^- (D) = \#\Omega_3 (D) = 0$.

We can refine \eqref{eqn:valenza} for hard diagrams:
\begin{cor}\label{prop:stimadalbasso}
If $D$ is a hard diagram of a non-periodic knot $K$, then $$v_{S}(D) = 8\alpha (D) + \sum_{k>1}  k(k-1)p_k (D).$$
The analogous result for $\reid$ is obtained by adding $k_{ext}(k_{ext}-1)$.
\end{cor}

In \cite{kauffman2006hard} Kauffman and Lambropoulou exhibit an infinite family of hard unknots. Using their result, it is not difficult to argue that every knot admits (infinitely many) hard diagrams. Take any diagram $D \in \diag$, and choose a (non-trivial) hard diagram $U$ of the unknot. If $D$ is not hard, choose a $\Omega_i^-$ or $\Omega_3$ move and perform a diagram connected sum with $U$ to ``kill it'' as in Figure \ref{kill}.
\begin{figure}[h!]
\includegraphics[width=10cm]{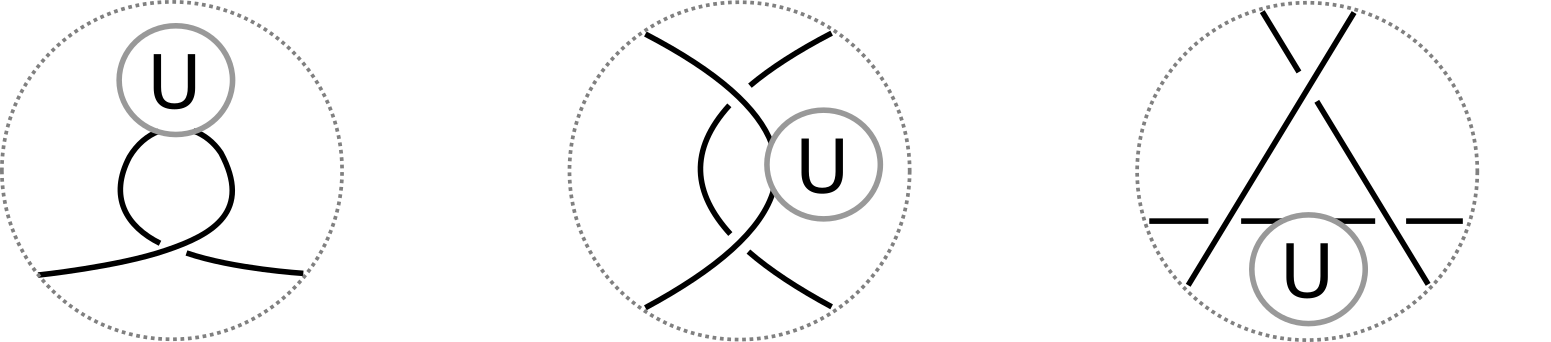}
\caption{How to kill Reidemeister moves.}
\label{kill}
\end{figure}
Generally, hard diagrams of non-periodic knots are interesting from the $\mathcal{R}$-graphs viewpoint, since for them the valence is completely determined by the knot projection, rather than by the diagram. In particular this implies that given a hard diagram, it will have minimal valence among all the diagrams obtained from it by changing any number of crossings\footnote{This is no longer true if one of the diagrams obtained by changing some crossings in a hard one is periodic.}.

\begin{rmk}\label{rmk:trifogliminimi}
It is possible to compute the valence of the two trefoil knots of Figure \ref{fig:trifogli} in $\mathcal{G} (3_1)$. Taking into account the periodicities of the two diagrams (it is of order $3$ for the first and $2$ for the other),
one gets that (as planar diagrams) the first has valence $24$ and the second $32$, so they are set apart in $\mathcal{G}(3_1)$. The valence in $\mathcal{G}_S(3_1)$ instead is $12$.
We will in fact prove in an upcoming paper that $\delta(3_1) = 24$ and $\delta_S(3_1) = 12$, and that in both cases $\#\delta(3_1) = 1$.
\end{rmk}

In order to facilitate the proof of Theorem$\;$\ref{thm:completeness}, understanding how the valence of a diagram can change under the various Reidemeister moves is crucial.

It is of course impossible to \emph{a priori} compute the difference of the valence between two vertices at distance $1$, since this value depends on the crossings and specific configurations in the diagrams involved. It is however possible to pinpoint a quite good bound by accounting for the number of edges of the regions interested by the Reidemeister move.

This last task is a quite tedious exercise; in the following we denote by\footnote{We suppress the dependency of the $\varepsilon_{j,i}$ from the diagrams in the notation for aesthetic reasons.} 
 $\varepsilon_{j,i}$ and $\varepsilon_{j,3}$ the difference in the number of Reidemeister moves of type $\Omega_i^-$ and $\Omega_3$ respectively that can be performed on two diagrams differing by a single Reidemeister move $\Omega_j^+$, with $i,j \in \{1,T,2\}$.\\
If $D^\prime = \Omega_1^+(D)$, then:
\begin{itemize}
\item $\varepsilon_{1,1} \in \{ 0,1\} $; 
\item $\varepsilon_{1,T} + \varepsilon_{1,2} \in \{-2,0,1\} $\footnote{Here we consider the sum $\varepsilon_{1,T} + \varepsilon_{1,2}$ since performing an $\Omega_1^+$ move at the top of a pre-existing tentacle may decrease the number of $\Omega_T$ moves, changing them in $\Omega_2$ moves. };
\item $\varepsilon_{1,3} \in \{-4, \ldots , 4\} $.
\end{itemize}

We denote the sum of the $\varepsilon$ contributions in each case as $\sum_i \varepsilon_{1,i}$; these count the part of the valence of a diagram that is not completely determined by the knot projection.
In particular, we have that 
\begin{equation}\label{epsilon1}
-6 \leq \sum_{i\in \{1,2,T,3\}} \varepsilon_{1,i} \leq 6.
\end{equation}

\begin{figure}[h!]
\includegraphics[width=7cm]{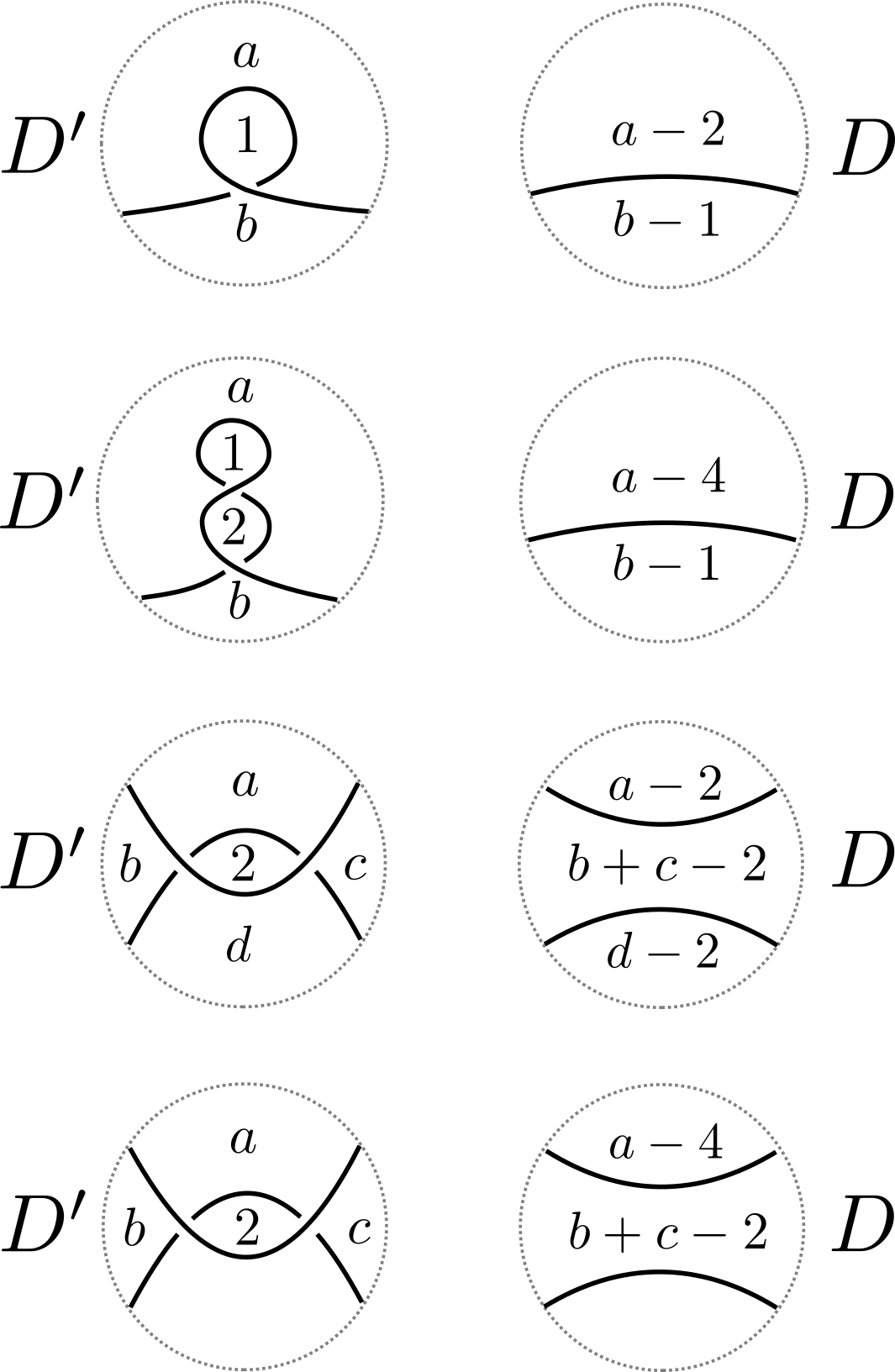}
\caption{The four possible configurations considered in Propositions \ref{prop:changer1} to \ref{prop:changer2bis}.}
\label{modificare}
\end{figure}

\begin{prop}\label{prop:changer1}
If $K$ is a non-periodic knot and $D^\prime \in \diag$ is obtained from $D$ by adding a curl (\emph{i.e.}$\;$performing a $\Omega_1^+$ move, as in the upper part of Figure \ref{modificare}) then $v(D^\prime) > v(D)$. More precisely, if the move is \emph{internal}, that is the two zones involved are not the external one, then:
\begin{equation}\label{eqn:r1}
v(D^\prime) = v(D) + 8 + 4a+2b + \sum_j \varepsilon_{1,j}.
\end{equation}
If the zone with $a$ edges is external:
\begin{equation}\label{eqn:r1ee}
v(D^\prime) = v(D) + 2 + 8a+2b + \sum_j \varepsilon_{1,j}.
\end{equation}
And finally if the zone with $b$ edges is external:
\begin{equation}\label{eqn:r1e}
v(D^\prime) = v(D) + 6 + 4a+4b + \sum_j \varepsilon_{1,j}.
\end{equation}
Moreover, we have $ \sum_j \varepsilon_{1,j} \in \{-6,\ldots,6\}$. Thus, performing an $\Omega_1^+$ move always increases the valence.

\end{prop}
\begin{proof}
After performing a $\Omega_1^+$ move, the number of arcs in $D'$ increases by 2, that is $\alpha(D')=\alpha(D)+2$. Moreover, assuming that $a,b,a-2,b-1$ are pairwise distinct, we have the following changes in the $p_k$s:
\begin{itemize}
\item $p_a(D')= p_a(D)+1$
\item $p_{a-2}(D')= p_a(D)-1$
\item $p_b(D')= p_b(D)+1$
\item $p_{b-1}(D')= p_b(D)-1$
\end{itemize}

Adding all up, and keeping in mind Equation \eqref{eqn:valenza}, we obtain $$ v(D')-v(D)= 8\cdot 2 + a(a-1) -(a-2)(a-3)+$$
$$ + b(b-1)-(b-1)(b-2) +  \sum_j \varepsilon_{1,j}.$$
That is precisely $$ v(D^\prime) = v(D) + 8 + 4a+2b + \sum_j \varepsilon_{1,j}.$$
Notice that even if $a,b,a-2,b-1$ are not pairwise distinct, the same computation holds.
All other $\Omega_j^-$ moves (that do not depend solely on the knot projection) add up to $\sum_j \varepsilon_{1,j}$.\\
To obtain Equations \eqref{eqn:r1ee} and \eqref{eqn:r1e} it is enough to add the contribution of the external region, which is $a(a-1)-(a-2)(a-3)=4a-6$ in the first case, and $b(b-1)-(b-1)(b-2)=2b-2$ in the second.
\end{proof}

The proof is identical in the other cases considered below, and we are going to omit it.

\begin{prop}\label{prop:changesr2}
Let $D, D^\prime \in \diag$ be two non-periodic diagrams differing by a $\Omega_T$ creating a tentacle of length $1$ as in the upper-middle part of Figure \ref{modificare}. Then, if the zones involved are not external:
\begin{equation}\label{eqn:changesr2}
v(D^\prime) = v(D) + 12 + 8a + 2b +  \sum_j \varepsilon_{T,j}.
\end{equation}
If the zone with $a$ edges is external:
\begin{equation}
v(D^\prime) = v(D) -8 + 16a + 2b + \sum_j \varepsilon_{T,j}.
\end{equation}
And finally if the zone with $b$ edges is external:
\begin{equation}
v(D^\prime) = v(D) + 10 + 8a + 4b + \sum_j \varepsilon_{T,j}.
\end{equation}

\end{prop}

\begin{prop}\label{prop:changer2}
If two non-periodic diagrams $D, D^\prime \in \diag$ differ by a $\Omega_2$ move in which the regions with $a$ and $d$ edges do not coincide, as in the middle part of Figure \ref{modificare}, then if the move is internal:
\begin{equation}\label{eqn:changer2}
v(D^\prime) = v(D) + 16 + 4(a+b+c+d) - 2bc + \sum_j  \varepsilon_{2,j}.
\end{equation}

\end{prop}

\begin{prop}\label{prop:changer2bis}
If two non-periodic diagrams $D, D^\prime \in \diag$ differ by a $\Omega_2$ move in which the regions with $a$ and $d$ edges coincide, as in the lower part of Figure \ref{modificare}, then if the move is internal:
\begin{equation}\label{eqn:changer2bis}
v(D^\prime) = v(D) + 8 + 4(2a+b+c) - 2bc + \sum_j \varepsilon_{2,j}.
\end{equation}

\end{prop}

\begin{rmk}
A $\Omega_T$ creating a tentacle of length greater than $1$ is a special case of \ref{prop:changer2bis} in which $c=2$. Thus, in this case we obtain $$v(D^\prime) = v(D) + 8 + 4(2a+b+2) - 4b + \sum_j \varepsilon_i=$$ $$=  v(D) + 16 +8a +\sum_j \varepsilon_{T,j}.$$
\end{rmk}

It is worth to remark that, when dealing with $\sfera$, the change of the valence is determined by Equations \eqref{eqn:r1}, \eqref{eqn:changesr2} and \eqref{eqn:changer2} in the respective cases.

\begin{rmk}\label{rmk:valenzapositiva}
We will find useful to divide the valence of every vertex in two parts, namely the positive valence $v^+(D)$ and the negative valence $v^-(D)$. The positive valence is defined as the number of edges emanating from $D$ which correspond to $\Omega_*^+$ moves, where $* \in \{1,2,T\}$. Note that $v^+(D)$ only depends on the projection of $D$.
If we wish to consider only the positive valence, Equations \eqref{eqn:r1}, \eqref{eqn:changesr2}, \eqref{eqn:changer2} and \eqref{eqn:changer2bis} can be rewritten as:
\begin{equation}\label{eqn:r1+}
v^+(D^\prime) = v^+(D) + 8 + 4a+2b
\end{equation}
\begin{equation}\label{eqn:changesr2+}
v^+(D^\prime) = v^+(D) + 12 + 8a + 2b
\end{equation}
\begin{equation}\label{eqn:changer2+}
v^+(D^\prime) = v^+(D) + 16 + 4(a+b+c+d) - 2bc
\end{equation}
\begin{equation}\label{eqn:changer2bis+}
v^+(D^\prime) = v^+(D) + 8 + 4(2a+b+c) - 2bc
\end{equation}
\end{rmk}

\begin{rmk}\label{rmk:diagrammanonmin}
Proposition \ref{prop:changesr2} suggests how to produce examples of knots in which the minimal complexity is not realised by a diagram minimising the crossing number. From Equation \eqref{eqn:changer2} it is apparent that if $b$ and$\slash$or $c$ are sufficiently big, then the diagram $D^\prime$ (with higher crossing number than $D$) obtained by performing a $\Omega_2^+$ move, will have a lower valence. An easy example of this phenomenon is given in Figure \ref{fig:nodoesempio}. This is the twist knot with $17$ crossings; note that the example shown is also alternating, reduced and non-periodic.
\begin{figure}[h!]
\includegraphics[width=5cm]{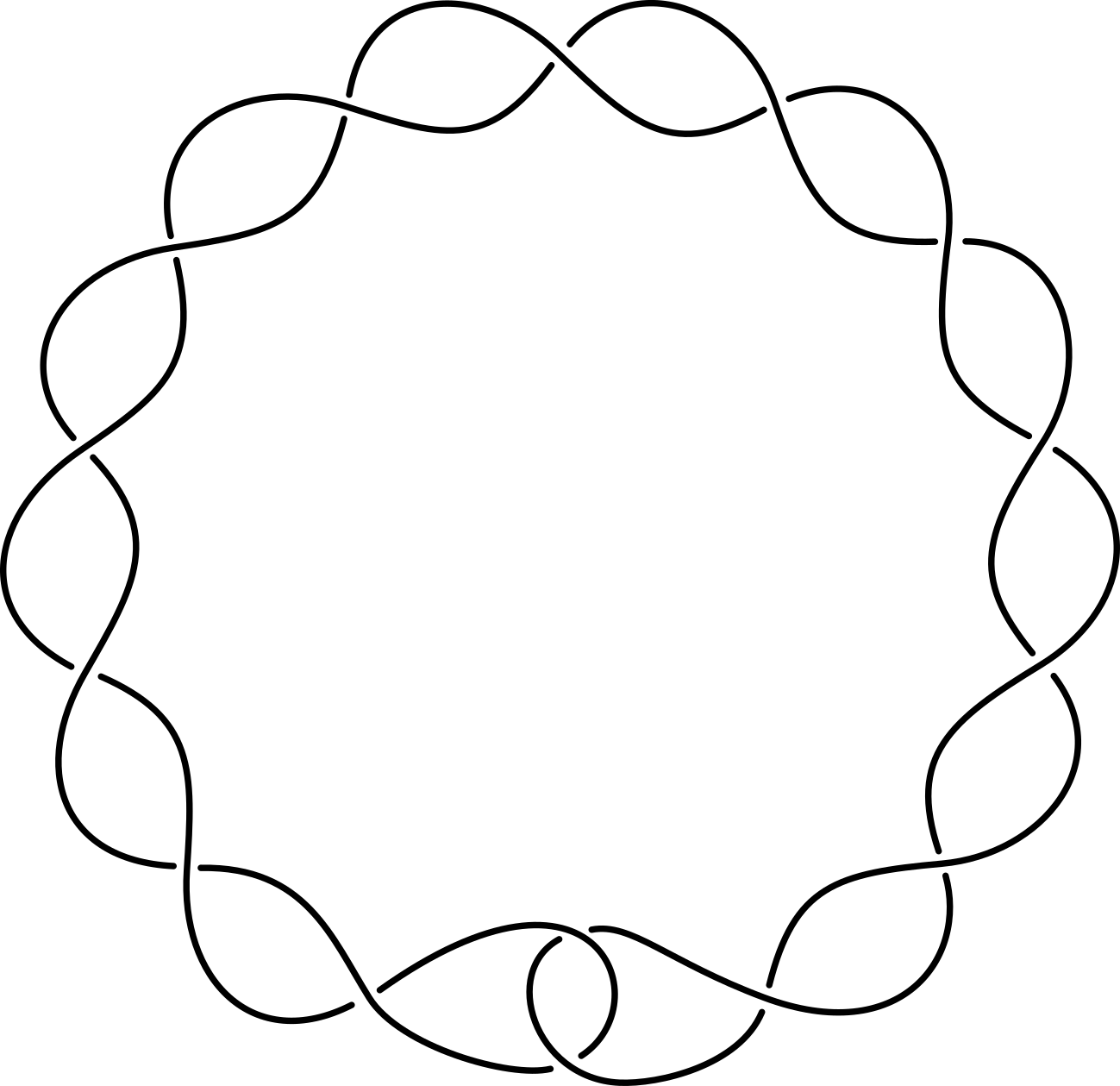}
\caption{By performing a $\Omega_2^+$ move in the central (or external) region, we obtain a diagram with lower valence.}
\label{fig:nodoesempio}
\end{figure}
When the internal (and external) region has more than $12$ faces, performing a $\Rd^+$ move decreases the valence, according to Equation \eqref{eqn:changer2} (with $a=d=4$ and $b=c=8$).

In particular, any knot in which all diagrams realizing the crossing number have many regions with a sufficiently high number of edges provides an example where the minimal valence is not realised in the diagram with minimal crossing number.
\end{rmk}

We prove here some facts that are going to be useful in the next sections.
First of all we show that the graph can \emph{distinguish} between the different Reidemeister moves. This means that by looking at a neighborhood of an edge of $\sfera$, we can tell which Reidemeister move it represents; furthermore this will provide a way to read the crossing number of a diagram $D$ from the combinatorial structure of $S(D)$.

\begin{thm}\label{distinguemosse}
The $S^2$-graph distinguishes the Reidemeister moves, and detects the crossing number of a diagram.
\end{thm}
\begin{proof}
In the interest of clarity we are going to start by examining the non-periodic case. Fix a diagram $D \in \diag$ for a  non-periodic knot $K$. The combinatorics of $S(D)$ will allow us to distinguish the various moves.

Since by Theorem \ref{prop:lung3} all $\Omega_1$ moves are paired with at least one $\Omega_T$ move in a triangle, it is easy to argue that the graph can tell apart the two sets of moves $M_1 = \{\Omega_1^{\pm}, \Omega_T^{\pm}\}$ and $M_2 = \{\Omega_2^{\pm}, \Omega_3\}$.

To further separate the elements of $M_1$ we can thus restrict to triangles in $S(D)$.
Choose an edge emanating from a vertex $D$, which is part of a triangle. There are $3$ possibilities, shown in Figure \ref{fig:distinguetriangoli}.
\begin{figure}
\includegraphics[width=10cm]{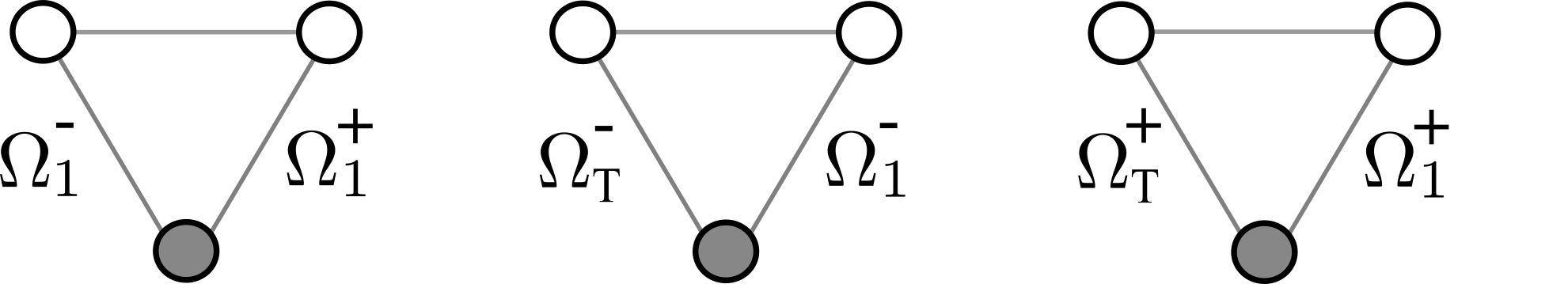}
\caption{The grey dots represent the diagram $D$ we start from.}
\label{fig:distinguetriangoli}
\end{figure}

From this, using Proposition \ref{prop:changer1}, it is easy to argue that $\sfera$ can tell apart the elements in $M_1$; indeed, if only one of the two moves decreases the valence, then they are both $\Omega_1$s, and the one which decreases it is the $\Omega_1^-$. If both moves decrease the valence, then the one that decreases it most is the $\Omega_T^-$, and the other is a $\Omega_1^-$. Lastly, if both moves increase the valence, then the one that increases it most is a $\Omega_T^+$, and the other is a $\Omega_1^+$.

Now, since the number of $\Omega_1^+$ moves is a multiple of the arc number of the diagram (cf. Figure \ref{fig:possibilir1}), the crossing number $cr(D)$ corresponds to $\frac{1}{8}\# \Omega_1^+ (D)$.
Hence, since we can distinguish and count such moves, we can read the crossing number of a diagram from $S(D)$.

Using this information we can tell apart the elements of $M_2$ as well and conclude: the only remaining moves are $\Omega_2^\pm$s and $\Omega_3$s, all of which are not part of a triangle. These appear as edges connected only to the center of $S(D)$. We can distinguish between them by counting the crossing number of the vertices they connect $D$ to; one then just needs to recall that  $\Omega_3$ moves do not increase it, while $\Rd$ increase or decrease it by $2$. Hence it follows that we can distinguish among $\Omega_2^+$,$\Omega_2^-$ and $\Omega_3$ moves as well.

In all the previous discussion, in order to determine $cr(D)$, we only used the fact that all diagrams in $S(D)$ were non-periodic; this fact will allow us to compute it in the periodic case as well.

If $K$ is periodic we can not use directly the various equations relating the valence of two neighbouring vertices, since one\footnote{Or both, see also Question \ref{quesperiodici}.} could be periodic.

Instead of trying to directly detect from the structure of the graph whether a diagram is periodic, we can use Lemma \ref{lemma:miniprop} to bypass most complications. For every vertex $D$, define the generalized triangle number $n_{tr}(D)= \#\Omega_1^{\pm}(D) +\#\Omega_T^{\pm}(D)$. This quantity is computable from the graph, since by Theorem \ref{prop:lung3} it coincides with the number of edges emanating from $D$ which are part of at least one triangle. 
By Lemma \ref{lemma:miniprop}, at least two diagrams appearing in a triangle, reached by a $\Omega_1^+$ and a $\Omega_T^+$ respectively will
be non-periodic, and we claim that such diagrams maximise $n_{tr}$ among all the diagrams reached by edges starting from $D$ that are part of at least one triangle. Let $a_h(D)$ denote the number of (maximal) height $h$-tentacles in $D$, and define $$n(D)= p_1(D) +  \sum_{h \geq 1} h a_h (D).$$
Note that $n(D)$ is equal to the sum $\Omega_1^- (D) +\Omega_T^- (D)$ when $D$ is not periodic.
Then, for the diagrams $D''$ and $D^\prime$ in Lemma \ref{lemma:miniprop}, the following equalities hold: 
$$ n_{tr}(D'')= 16cr(D'') + n(D'') = 16(cr(D)+2) + n(D)+2 $$
$$ n_{tr}(D^\prime_{1}) = 16cr(D') + n(D')= 16(cr(D)+1) +  n(D)+1$$ 
This follows since we are performing the curls on the top of a tentacle (or on any arc, if there are no curls in $D$), and this fact ensures that the number of $\Omega_*^-$ moves is equal to $n(D)+1$ when $*=1$ and to $n(D)+2$ when $*=T$.
On the other hand, for any diagram $D^+_T$ and $D^{+}_1$ reached from $D$ by a $\Omega_T^+$ move and a $\Omega_1^+$ respectively, the following inequalities hold: $$ n_{tr}(D^{+}_T) \leq 16(cr(D)+2) + n(D)+2 $$ 
$$ n_{tr}(D^{+}_1) \leq 16(cr(D)+1) +  n(D)+1 $$\\
The presence of periodicity in $D^+_1$ or $D^+_T$ can only decrease the value of $n_{tr}$, and the same holds if the moves are not performed (with the appropriate sign) on the top of a pre-existing tentacle. In other words these moves maximise $\#\Omega_1^- + \#\Omega_T^-$. If we consider moves that decrease the crossing number, disregarding the possible periodicities, the numbers $n_{tr}$ we obtain have no chance of being greater than $ n_{tr}(D'')$.
So, choose the diagrams in $S(D)$ maximising this quantity; they correspond to vertices reached by $\Omega_T^+$ moves. Consider all the edges that form triangles with them: these have to correspond to diagrams reached by $\Omega_1^+$ moves. Choose between them one maximising $n_{tr}$. Notice that $D$ is non-periodic if and only if $n_{tr}(D)=n_{tr}(D')-17$.
Now, choose  $D'''$ in $S(D')$ forming a triangle with $D''$, with $S(D''')$ totally non-periodic, and such that it maximises $n_{tr}$ in $S(D')$. We know that such a diagram exist by Lemma \ref{lemma:miniprop}, and we can check the hypothesis on the non-periodicity of $S(D''')$ thanks to the above criteria. Then, we can recover $cr(D''')$, and obtain $cr(D)$ as $cr(D''') -3$.

Hence, using the crossing number as in the non-periodic case, we can tell apart the various types of moves, and we are done.
\end{proof}
This last result will allow us to say ``perform a $\Omega_i^\pm$ move on a diagram'' in a way that is meaningful also at the level of the graph. In other words, we just proved that the $\mathcal{R}$-graphs intrinsically contain the same amount of information as the same graphs with edges decorated according to which $\Omega_i^\pm$ move we are performing.\\

By the previous result we know that the crossing number of a diagram can be read by looking at $S(D)$. Thus if a knot is non-periodic, taking the minimum of $\frac{1}{8}\# \Omega_1^+ $ among all vertices of the corresponding Reidemeister graphs gives back $cr(K)$, the crossing number of the knot. For periodic knots, this procedure produces a slightly different invariant, which can be regarded as crossing number up to periodicities.
More precisely define $$\widehat{cr}(K) = \frac{1}{8}\min_{D \in \diag} \# \Omega_1^+ (D).$$  $\widehat{cr}(K) = cr(K)$ if $K$ is non-periodic, while in general $\widehat{cr}(K) \le cr(D)$. As an example, we have $\widehat{cr}(3_1) = \frac{1}{2}$.

Note that a similar consideration for the other kinds of moves does not yield useful invariants: it is possible to show that the minimal number of $\Omega_2^+$ moves is simply related to the combinatorics of the number of regions in the complement of the diagram on $\R^2$ or $S^2$, and the minimal number of $\Omega_2^-$ and $\Omega_3$ moves one can perform within a knot type is always 0 (as can be seen by ``killing'' all the $\Omega_3$ moves with a $\Omega_1$ in the region with $3$ edges, similarly to what was done in Figure \ref{kill}). Nonetheless one might obtain some meaningful invariants by restricting diagrams not minising the valence.\\

The knowledge of the crossing number from the graph also implies that we can use Coward and Lackenby's result \cite{coward2014upper} to give some upper bounds on the path distance between two diagrams.

\section{Global properties}\label{global}

This section is devoted to the analysis of some global properties of the $\mathcal{R}$-graphs. We begin by proving that each $\mathcal{R}$-graph is not hyperbolic.

\begin{prop}\label{prop:latticeembed}
The $\mathcal{R}$-graphs are not hyperbolic.
\end{prop}
\begin{proof}

Choose a non-periodic diagram $D \in \diag$ not containing $1$-regions, an arc on $D$, and a polygon $P$ having this arc as a face.  We can embed isometrically the rank $2$ lattice graph as follows: to the pair $(a,b) \in \Z^2$ associate the configuration on the arc composed by $a$ positive curls in the region $P$ if $a > 0$, and in the other region touching the arc if $a<0$; do the same for $b$, this time with negative crossings on the right of the previous ones. An example is shown in Figure \ref{fig:nonhypspri}.  The fact that the embeddings are isometric follows \emph{e.g.} from the analysis of the $I_{lk}$ invariants of the diagrams: $I_{lk}(D_{a,b}) = I_{lk}(D) + aX_0 + bY_0$, where $D_{a,b}$ is the diagram corresponding to the element $(a,b)$.

\begin{figure}[h!]
\includegraphics[width=4cm]{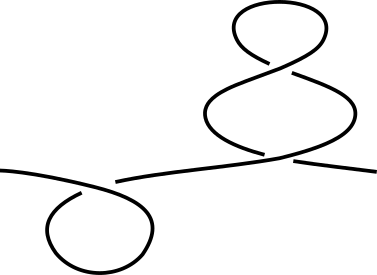}
\caption{This configuration represents $(-1,2) \in \Z^2$.}
\label{fig:nonhypspri}
\end{figure}
\end{proof}

\begin{prop}\label{prop:neverplanar}
The Reidemeister graphs $\reid$ and $\sfera$ are not planar.
\end{prop}
\begin{proof}
We are going to prove that for every knot $K$ we can find a $K_5$ minor\footnote{As is customary, $K_n$ denotes the complete graph on $n$ vertices.} contained in each $\mathcal{R}$-graph of $K$. This is achieved by considering the local construction shown in Figure \ref{fig:nonplanar}.
\begin{figure}
\includegraphics[width=7cm]{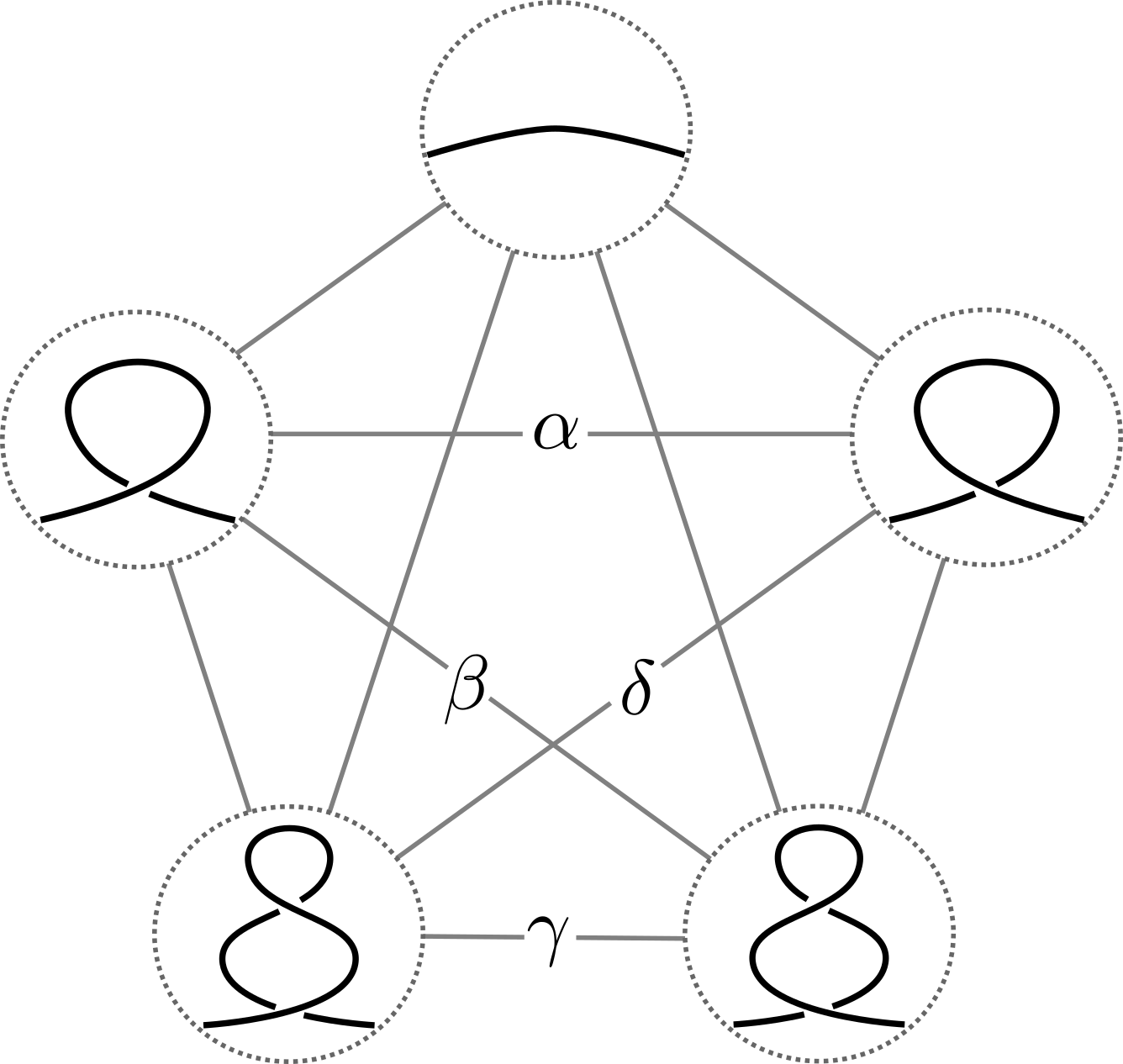}
\caption{A local embedding of $K_5$ as a minor of any $\reid$.}
\label{fig:nonplanar}
\end{figure}
The edges denoted with a Greek letter are length $2$ paths; as shown in Figure \ref{fig:latopentagono}, these can be obtained by putting alongside the two moves, and then resolving either one.

\begin{figure}
\includegraphics[width=8cm]{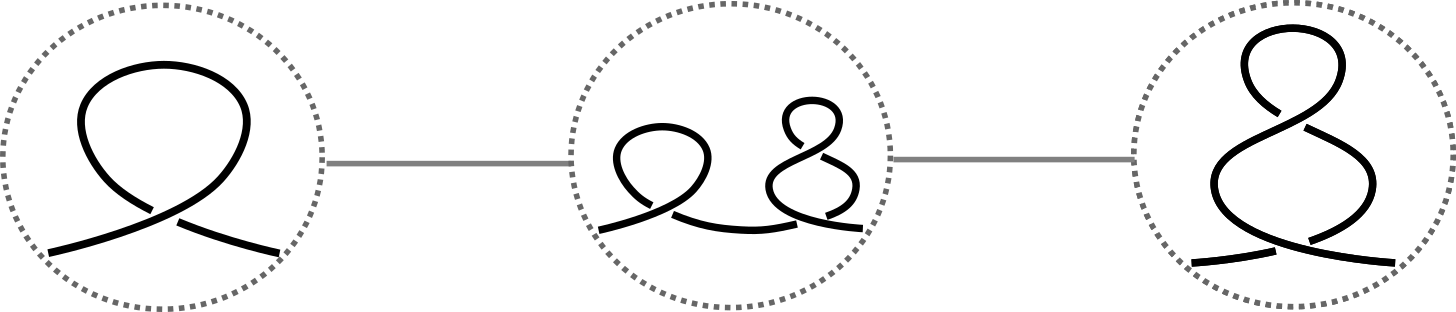}
\caption{The path corresponding to the $\beta$ edge in Figure \ref{fig:nonplanar}.}
\label{fig:latopentagono}
\end{figure}
\end{proof}

In graph theory, it is customary to consider \emph{coarse} properties of a (infinite and locally finite) graph. One way to do this is to study the quasi-isometry class of the graph, often through related invariants.

A \emph{ray} of a locally finite graph $G$, is a semi-infinite simple path in $G$; two rays $r_1, r_2 \subset G$ are regarded as equivalent if there exists a third ray $r_3$, containing infinitely many vertices of both $r_1$ and $r_2$. \\An \emph{end} is an equivalence class of rays, and it is called \emph{thick} if it contains infinitely many pairwise disjoint rays.

\begin{prop}\label{thickend}
Each $S^2$-Reidemeister graph has only one thick end.
\end{prop}
\begin{proof}
It is immediate to show (\emph{e.g.}$\;$using paths as those in Figure \ref{fractal1} or tentacle configurations) that there are infinitely many disjoint rays in $\reid$ and $\sfera$ for each choice of $K \in \nodi$.
To show that there is only one end, we will prove that removing any ball with arbitrary radius does not disconnect the graphs into two pieces, each containing infinitely many vertices. This in turn would immediately imply that there is only one equivalence class of rays in the graph.\\
Consider a diagram $D \in \diag$ for a knot $K$, and the radius $R$ ball $S_R(D)$ in $\reid$ (or equivalently in $\sfera$). Call $H$ the maximal height among the tentacles of the diagrams contained in  $S_R(D)$, and take any two diagrams $D_0, D_1 \in \diag$ which do not belong to $S_{R+H+1}(D)$; we need to find a path in $\reid \setminus S_R(D)$ connecting $D_0$ to $D_1$. Choose an arc on $D_0$ and on $D_1$, and create on each a tentacle of height greater than $H$. These two new diagrams $D_0^\prime$ and $D_1^\prime$ can be connected through moves that avoid the newly created tentacles\footnote{Remember that we are working on $S^2$.}, and this path $\gamma$ from $D_0^\prime$ to $D_1^\prime$ will not intersect $S_R(D)$, thanks to the hypothesis on $H$.
Attaching to the ends of $\gamma$ the two paths $\widehat{\gamma}_i$ from $D_i$ to $D_i^\prime$, induced by the creation of the tentacles, gives the desired path from $D_0$ to $D_1$. Note that the hypothesis on the height of the tentacle allow us to say that the paths $\widehat{\gamma}_i$ do not intersect $S_R(D)$, and $S_{R+H+1}(D) \setminus S_R(D)$ contains only finitely many vertices.
\end{proof}

The $S^2$-Reidemeister graphs contain only one end, but infinitely many disjoint rays, hence by Halin's grid Theorem \cite{halin}, each must contain a subdivision of the planar hexagonal tiling.

\begin{rmk}
One might find reasonable to assume that the graphs $\reid$ and $\sfera$ are quasi-isometric; it is however easy to see that the ``natural'' map\footnote{Where we map a planar diagram $D$ to its equivalence class in $\sfera$.} $\reid \hookrightarrow \sfera$ between the two graphs fails to be a quasi-isometry. This can be seen from Figure \ref{fig:zonaext}: the two diagrams on the left and right can have arbitrarily large distance in $\reid$, but are identified in $\sfera$.
\end{rmk}

These graphs also exhibit a fractal behavior, which can be observed \emph{e.g.}$\;$by considering sequences of $\Omega_1$ moves as in Figure \ref{fractal1}.
\begin{figure}[h]
\includegraphics[width=10cm]{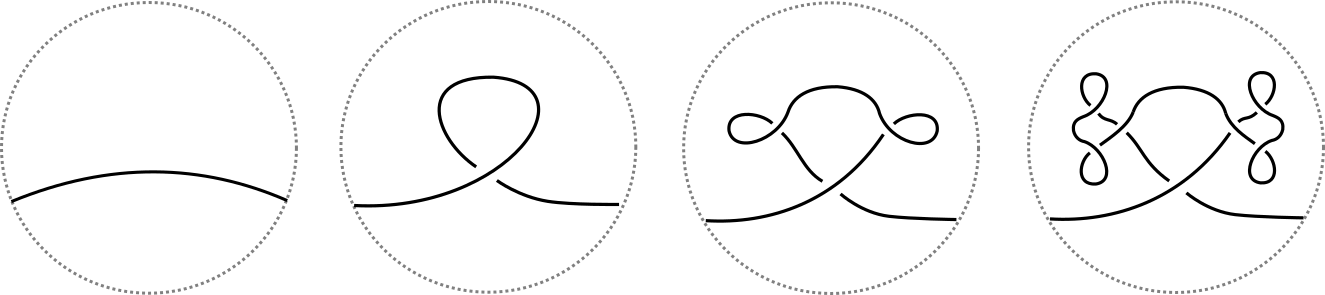}
\caption{The ``fractal behavior'' of $\reid$.}
\label{fractal1}
\end{figure}
The corresponding subgraph can be embedded (infinitely many times) in each $\mathcal{R}$-graph for any knot $K$.\\

The $\mathcal{R}$-graphs can be filtered in several ways; the easiest one is to consider the filtration induced by the distance from the vertices with minimal valence.

Given a knot $K$, denote by $\mathcal{F}_m (K)$ the subgraph spanned by those vertices whose distance from the minimal diagrams of $K$ is $\le m$, and denote by $\# \mathcal{F}_m (K)$ the number of vertices it contains.

We can extract some numerical invariants from this filtration on $\reid$:
\begin{defi}\label{def:filtrazione1}
Define $$f_K : \N \longrightarrow \N $$ $$ f_K (m) = \# \mathcal{F}_m (K),$$
and $$ M(K) = \min_{m \ge \delta(K)} \{\pi_0 (\mathcal{F}_m (K))= \Z\}.$$
In other words, $M(K)$ measures the minimal distance between the diagrams of minimal complexity in $\reid$.
In particular, $M(K) = 0$ if and only if a knot type is simple.
\end{defi}

Recalling the proof of Lemma \ref{lemma:pochi}, we can also define another filtration on $\reid$:
\begin{defi}
Let $\widetilde{\mathcal{F}}_K$ be the filtration of $\reid$ whose $m$-level consists of the vertices of $\reid$ with valence less or equal to $m$.
Let also $$g_K : \N \longrightarrow \N$$ be defined as the associated counting function
$$g_K(m) = \#\{ D \in \reid \;|\; v(D) \le m\} .$$
Clearly $g_K(m) = 0 \; \forall m < \delta(K)$, and $g_K(\delta(K)) = \# \delta(K)$.
\end{defi}

Both these filtrations $\mathcal{F}$ and $\widetilde{\mathcal{F}}$, together with the associated integer valued counting functions $f_K$ and $g_K$ are knot invariants, and it is not hard to show that they both distinguish the unknot. Moreover one can consider the homology groups of the various level sets and obtain yet other knot invariants.\\

In \cite{miyazawa2012distance} Miyazawa computes the homology groups of the \emph{Reidemeister complex}, which he denotes by $ M(K:P_5,1)$, in the case of oriented diagrams with a minimal generating set of Reidemeister moves. Along these lines we can define a slightly different version of Reidemeister complex, denoted by $\mathcal{CG}(K)$  and defined as follows: a $n$-simplex $\Delta_n = \langle D_0, \ldots, D_n\rangle$ is given by a string of $n+1$ distinct diagrams such that\footnote{Here $\delta_{i,j}$ denotes Kronecker's delta function, and $d$ is the path distance.} $d(D_i,D_j) = 1 -\delta_{i,j}$, considered up to permutations of the indices.

Define $C_n(\mathcal{CG}(K))$ as the free abelian group generated by $n$-simplices, with the obvious boundary operator induced by simplicial homology:
\begin{equation}
\partial (\langle D_0, \ldots, D_n\rangle) = \sum_{i=0}^n (-1)^i \langle D_0, \ldots, \widehat{D_j},\ldots, D_n\rangle.
\end{equation}
From this perspective, $\sfera $ is the $1$-skeleton of $\mathcal{CG}(K)$. Miyazawa proved that $H_0(M(K:P_5,1);\Z) = \Z$ (which follows from Reidemeister's Theorem), and that $H_n(M(K:P_5,1);\Z) = 0$ for every $n\ge 2$ and $K \in \nodi$.\\ Our situation is slightly different; with the methods developed in Section \ref{sec:grafo} we can easily enstablish the triviality of $H_n(\mathcal{CG}(K);\Z)$ for $n\ge 3$:
\begin{prop}
For any knot $K$ we have $H_n (\mathcal{CG}(K);\Z) = 0$ for $n\ge 3$ (in both the planar and $S^2$ case).
\end{prop}
\begin{proof}
Assume there exists a tetrahedron $\Delta_3$ in $\mathcal{CG}(K)$; then it follows from Theorem \ref{prop:lung3} that all faces have to be made up triples $\Omega_T^\pm$-$\Omega_1^\mp$-$\Omega_1^\mp$. Up to symmetries, there is only one possibility to be considered, shown in Figure \ref{fig:notetra}. However this can be excluded as well, by taking into account the signs of the moves composing the tetrahedron. In particular this shows that there are no simplices of dimension $n \ge 3$, hence all the corresponding homology groups vanish.
\begin{figure}
\includegraphics[width=5cm]{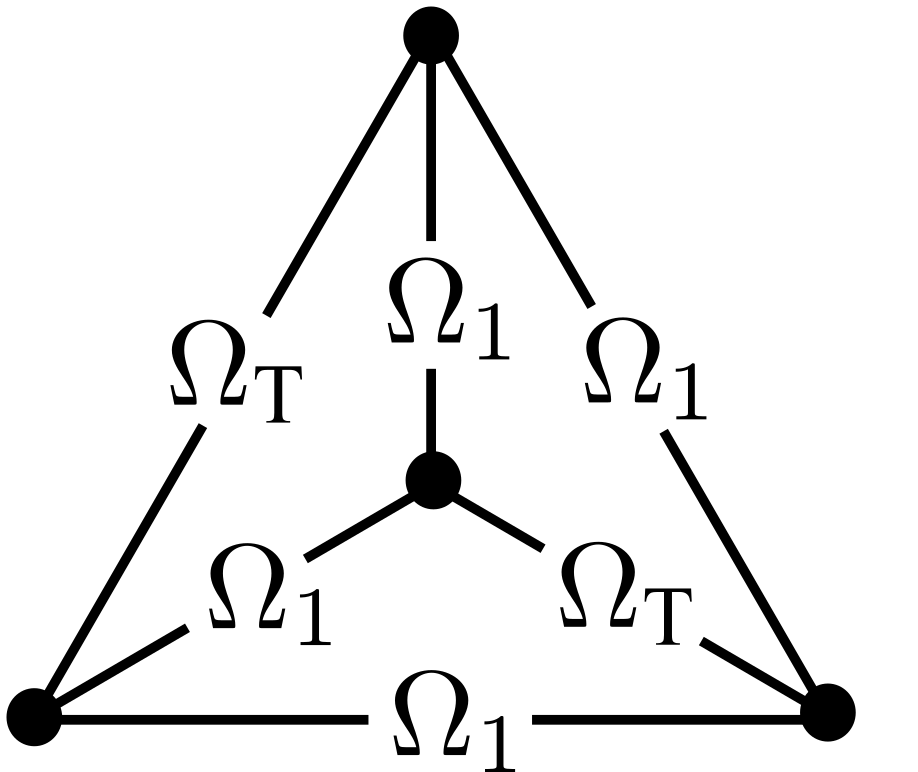}
\caption{The only tetrahedron with compatible faces. There is no way of coherently orienting the signs on the edges of its faces.}
\label{fig:notetra}
\end{figure}
\end{proof}
In particular it follows that $\mathcal{CG}(K)$ is just $\sfera$ with all triangles capped by $2$-simplices. It is not hard to prove that $H_1 (\mathcal{CG}(K);\Z)$ is an infinitely generated free abelian group, as any pair of distant $\Omega_3$-$\Omega_2$ moves does not bound any union of $2$-simplices. We can now conclude the computation of the homology groups of $\mathcal{CG}(K)$ with the next proposition.
\begin{prop}
The following holds:
\begin{equation}
H_2(\mathcal{CG}(K);\Z) \cong \Z^\infty .
\end{equation}
\end{prop}
\begin{proof}
The two configurations in the top part of Figure \ref{fig:configrmultiple} show a topologically embedded $2$-sphere in $\mathcal{CG}(K)$. As these are local configurations, they can be found infinitely many times on any $\sfera$.
\end{proof}

So, as in Miyazawa's case, the global homology does not provide useful invariants. Nonetheless, some properties of the diagrams can be inferred from the local homology of the complexes.
Denote by $S^{cpx}(D)$ the ball of radius $1$ centered in $D$, seen as a subcomplex of $\mathcal{CG}(K)$. 
\begin{lemma}
If $H_1 (S^{cpx}(D);\Z) = \Z^m$, then $m \ge \widehat{cr}(D)$. Moreover $H_2(S^{cpx}(D);\Z)=0 $ if and only if $p_1(D)=0$.
\end{lemma}
The first part of this lemma follows easily from \ref{fig:multiedge1}, while the second is a consequence of Lemma \ref{lemmanormale} together with Theorem \ref{prop:lung3}.

\section{Completeness of the \texorpdfstring{$S^2$}--graph invariant}\label{s:towards}

This section is devoted to the proof of Theorem \ref{thm:completeness} (recalled below). The proof will rely solely on results from Section \ref{local}, and exploits rather large portions of the graph.

\begin{thm*}
The $S^2$-Reidemeister graph is a complete knot invariant up to mirroring. That is  $\sfera \equiv \mathcal{G}_{S} (K^\prime)$ \emph{iff} $K^\prime = K$ or $\overline{K}$.
\end{thm*}
We will prove the Theorem by breaking it down in smaller parts, which are the content of the following Propositions.

Suppose we have a knot $K \in \nodi$, and that $D \in \diag$ is any diagram. Write $P(D) = (p_1(D), \ldots, p_m(D))$ where $m$ is the greatest coefficient with a non-zero entry (or equivalently the maximal number of sides among the regions in the complement of $D$ in $S^2$).

\begin{prop}\label{prop:determinaP}
The $S^2$-graph of a knots determines $P(D)$ for each vertex $D$ such that all diagrams in $S(D)$ are non-periodic and $p_1(D)=0$.
\end{prop}
\begin{proof}
Thanks to Lemma \ref{lemmanormale} we know that if a diagram $D$ does not contain any curls, then all the triangles in $S(D)$ which admit $D$ as the vertex with lower crossing number are normal. Moreover, since there are $8cr(D)$ $\Omega_1^+$ moves and $\Omega_T^+$ moves, we can conclude that in $S(D)$ there are exactly $8cr(D)$ triangles.
For each $\Omega_1^{+}$ move, choose the corresponding $\Omega_T^+$ move;
call $D^\prime$ and $D''$ respectively the diagrams obtained by performing these $\Omega_1^+$ and $\Omega_T^+$ moves on $D$. By Equations \eqref{eqn:r1+} and \eqref{eqn:changesr2+}, the difference of the positive valences is  $$v^+(D'') - v^+(D^\prime) = 4a + 20,$$
where $a$ is the number of edges of the region in which the tentacle and the curl will appear.

If we do the same for all possible $\Omega_1^+$ moves applicable to $D$, we get a set of numbers $\{n^\prime_i\}_{i \in \{0, \ldots ,\#\Omega_1^+(D)\}}$;  define a new set $\{n_i\}_{i \in \{0, \ldots ,\#\Omega_1^+(D)\}}$, where $n_i = \frac{n^\prime_i - 20}{4}$. 
Each region with $a$ sides contributes to this new list with exactly\footnote{Corresponding to the two possible $\Omega_1^+$ moves performed with the curl contained in the region on one of its edges.} $2a$ entries equal to $a$. It is thus immediate to show that we can compute each $p_a (D)$ from $S(D)$.
\end{proof}

However the knowledge of $P(D)$ on a subset of vertices does not immediately guarantee the completeness of $\sfera$. A priori there might be two distinct knots (up to mirroring), whose diagrams have the same number and types of Reidemeister moves, and such that their complement has the same number of regions. We first need to detect the structure of $D$ as a $4$-valent graph on $S^2$.

\begin{prop}\label{riconosceproj}
The $S^2$-graph recognises the projections of the knot corresponding to diagrams without $1$-regions.
\end{prop}
\begin{proof}
Let us deal only with non-periodic knots for the moment, and come back to the periodic case afterwards.

Choose a vertex $D \in \diag$ with $p_1(D)=0$. This is possible by Theorem \ref{distinguemosse}, since the condition $p_1(D)=0$ is equivalent to the absence of $\Omega_1^-$ moves emanating from $D$.

To obtain the structure of $D$ as a graph, we need to be able to tell which regions are adjacent to one another in $S^2$, or in other words, we want to determine the dual graph of the projection.

We need to look for this information outside of $S(D)$; begin by assuming for the moment that $D$ only has one region $R$ with a certain number $k$ of edges (that is, $p_{k} (D) = 1$), and $k$ is such that there are no regions with $k \pm l$ sides, for any $l \leq L$ where $L$ is a suitably big integer.

To determine the number of edges of the regions adjacent to $R$, perform a $\Omega_1^+$ move on one edge of\footnote{This is in fact well defined on the graph, since the valence of the diagrams obtained in this fashion will be different from any other obtainable by making a $\Omega_1^+$ anywhere else.} $R$, so that the curl is contained in the interior of $R$. We can then compute the number of edges of the other region 
involved in the move as follows. The  $\Omega_1^+$ move is associated to a unique $\Omega_T^+$ move, connecting $D$ with a diagram $D''$ with which they form a triangle. By counting the difference of the positive valences between $D'$ and $D''$ 
and by using Equations \eqref{eqn:r1+} and \eqref{eqn:changesr2+} we can compute the number of edges of $R$. Once we have that, the difference in the positive valence between  $D$ and $D'$ gives us the number of edges of the other region involved. Moreover, note that knowing the number of edges in these two regions is enough to compute $P(D')$ from $P(D)$.

Repeating for all edges\footnote{Thanks to the hypothesis on $R$, we can recongise from the graph all the $\Omega_1^+$ which create a curl in $R$.} in $R$
we get the number of sides of each region which shares an edge with $R$.

From this last paragraph it follows that if we could find a diagram of $K$ such that, with the exception of regions with one side, 
all the regions have a different and sufficiently spaced number of sides, then we could infer how they are globally ``patched together'' to form the corresponding $4$-valent graph.

There is an easy way to achieve such a configuration in a controlled way. Start with a diagram $D$ with $p_1(D)=0$, and perform a $\Omega_1^+$ move; again by the previous line of thought\footnote{Since triangles in $S(D)$ are normal, we can compute $P(D)$ by considering the difference in the valence between diagrams reached by triangles. Moreover, by considering one triangle at the time, and using the differences in the positive valences between $D''$ and $D'$, and between $D'$ and $D$, we can compute the number of edges involved in the corresponding $\Omega^+_1$.} we can tell that the move has been made with the curl contained in a region with $a$ edges, which is adjacent to a region with $b$ edges (as in the top of Figure \ref{modificare}).\\
Call $D_1$ the diagram obtained; we can recover $P(D_1)$ from $P(D)$, since we know the number of edges of the regions involved.
There are only two possible choices to perform another $\Omega_1^+$ move on this new diagram, in such a way that a $\Ru$-multi-edge with valence $2$ is created (given by performing an identical $\Omega_1^+$ move on the left or right of the previous one).
We can repeat this process $N_1 \gg 0$ times, obtaining a new diagram $D_{N_1}$ with only one region with $a+2N_1$ edges, and a region with $b + N_1$ edges, and such that there is only one multi-edge of order $N_1$, and exactly $ N_1$ regions with $1$ edge.
Notice again that at each step we can recover $P(D_i)$ form $P(D_{i-1})$, since at each step we already know the number of edges of one region involved, and from the difference in the positive valence between $D_{i-1}$ and $D_i$ we can recover the second one. Eventually, we are able to compute $P(D_{N_1})$. 

Now we have a more complicated diagram with two distinguished adjacent regions; we can then iterate the process: choose another edge of the first region\footnote{Such that it is not on the top of the $\Omega_1^+$s we just made.} and make $N_2$ identical $\Omega_1$s, with $N_1 \gg N_2 \gg0$, in such a way that the curls are contained in the first region. Again, the first step is well defined, since such $\Omega_1^+$ moves are the only ones that reach diagrams whose positive valence 
is increased by approximately $8N_1$ and do not increase the multiplicity of the multi-edge.
From the second step on, the lack of periodicity (see Remark \ref{rmk:triangolinormali}) ensures that making a curl close to the previous one is the only way to create a $\Omega_1^+$-multi-edge. We still can recover $P(D_i)$ at every step.\\
We can fill up every edge of the region who once had $a$ edges in the same fashion, and then move to another region. If at each step we start making curls on an edge bounding two regions whose number of edges is different enough\footnote{This can be achieved by moving through adjacent regions. Notice that every time an edge is filled with $N_i$ curls, we can ``remember'' the number of edges of the regions involved, and the number of curls made.}, and if we keep track of the number of curls added, we are sure that the moves are well defined on the graph and that we can compute the $n$-tuples.

Notice that we need to choose the numbers $N_i$ incredibly big and suitably distant, with $N_i \gg N_{i+1}$ in order to avoid confusion and ultimately get a diagram $\widetilde{D}$ such that it has region sequence of the form $$P(\widetilde{D}) = (N, 0,\ldots,0,1,0\ldots,0,1,0,\ldots),$$ where $N = \sum_i N_i$, and the minimal gap between two non-zero entries is $\gg0$. Call a diagram with these properties \emph{sparse}. \\If the number of edges of the various regions are sufficiently spaced, then the previous claim applies\footnote{Notice that even if the $p_1(\widetilde{D}) \neq 0$, the previous claim applies anyway, thanks to the sparseness of $\widetilde{D}$ and to the fact that the only $1$-regions in $\widetilde{D}$ are the ones created by the construction.}, and we can explicitly see which regions are adjacent to one another. However thus far we have only determined the dual graph to the knot projection as an \emph{abstract graph}; 
we need a bit more work to find out the specific planar embedding of the dual, in order to get back the projection of $D_0$. 
It is well known that an abstract finite planar graph $G$, together with a \emph{rotational system}, uniquely determines an embedding of $G$ and thus $G^*$, which is the diagram projection we want. A local rotational system for a vertex $v \in G$ is just a choice of a cyclic order for the edges emanating from $v$. A rotational system for $G$ is such a choice for each $v \in G$, and it is said to be \emph{coherent} if all the local systems are coherently oriented\footnote{That is, given any two adjacent region in the (embedded) dual of $G$, the orientations induced on the common edges do not coincide, as in Figure \ref{fig:sistciclico}.}.\\
\begin{figure}[h!]
\includegraphics[width=11cm]{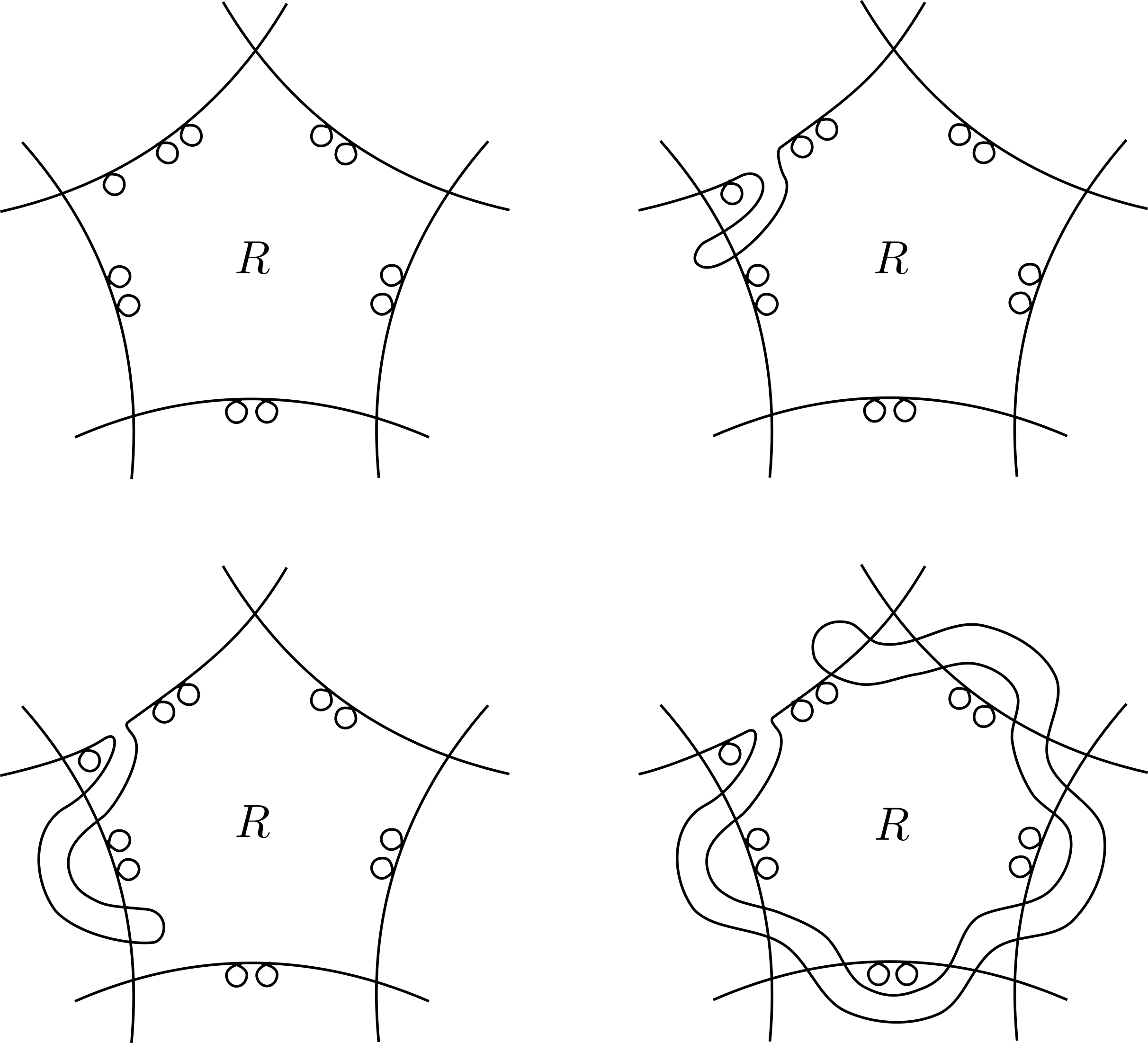}
\caption{The choice of a local orientation system for the dual. By the previous construction the curls on the arcs can be on either side of an edge.}
\label{fig:sistorient}
\end{figure}
Choose a region $R$ in the sparse diagram $\widetilde{D}$ and suppose that $R$ is bounded by $r$ edges.
Choose a $\Omega_2^+$ move that: creates a new bigon and a new $4$-region, 
and increases the number of edges of two regions adjacent to $R$ by $2$; one example is shown in the top-right part of Figure \ref{fig:sistorient}. Notice that this choice is well defined on the graph thanks to the sparseness of $\widetilde{D}$. 
Indeed, we know the numbers $r_i$ of edges of all the regions adjacent to $R$, and the $\Omega_2^+$-moves of that kind are the only ones that change the positive valence by a value of $ 56+ 4r_j + 4r_i -4r$ (as in Equation \eqref{eqn:changer2+}).

Then there are only two choices for a second $\Omega_2^+$ move that: creates a new bigon adjacent to $R$ leaving all the $1$-regions on the edge on the left of the newly created bigon, eliminates one bigon adjacent to one of the two regions whose vertices increased in the previous move (let us call this region $M$), and replaces it with a quadrilateral. Again, thanks to the sparseness of $\widetilde{D}$ we can identify such a move on the graph, by considering the valence of the diagram reached. In fact, we can identify all the $\Omega_2$ moves creating a new bigon adjacent to $R$ and eliminating the other bigon, since these are the only ones changing the valence by
\footnote{The following expression is valid if the $1$-regions are inside $R$. To consider the other case it is enough to change every $n_1$ in the expression with $2n_1$.} $40+ 4r -2n_1 r_j +2(n_1+2)(n_1-2)$ (as in Equation \eqref{eqn:changer2+}) where $n_1$ is the number of $1$-regions on the left of the newly created bigon. In order to detect the correct $\Omega_2$ moves it is sufficient to choose the one minimising the coefficient of $r_j$ in the previous expression\footnote{We can do that thanks to the sparseness of the diagram.}. 
These two options correspond to the possible choices of over/under passing for the first $\Omega_2^+$ move in Figure \ref{fig:sistorient}.
These moves might also decrease by a lot the number of edges of $M$ (according to how many curls are contained on the edge between $R$ and $M$). Now we can repeat the process following\footnote{We only need to follow moves that do not \emph{separate} curls lying on the same edge in different regions, and we can do that again using the difference in the valence together with the sparseness of the initial diagram.} Figure \ref{fig:sistorient}, until we get back to the first region that had its edges increased by $2$ in the first move (but was not $M$). Keeping track of the various regions encountered during this process allows to reconstruct a local orientation system about the vertex corresponding to $R$ in the dual graph
Since we know the numbers $r_i$ of edges of all the regions adjacent to $R$, and thanks to the sparseness of the diagram, this construction works even if two distinct edges of $R$ are shared with the same region.\\
Again, this sequence of $\Omega_2^+$ is only well-defined up to a choice of over/under passing at each step, but this indeterminacy does not affect the result.

Finally, in order to get a proper orientation system for the dual, we need to be able to have a coherent way of orienting these local rotational systems we obtained. The process is shown in Figure \ref{fig:sistciclico}.
\begin{figure}[h!]
\includegraphics[width=7cm]{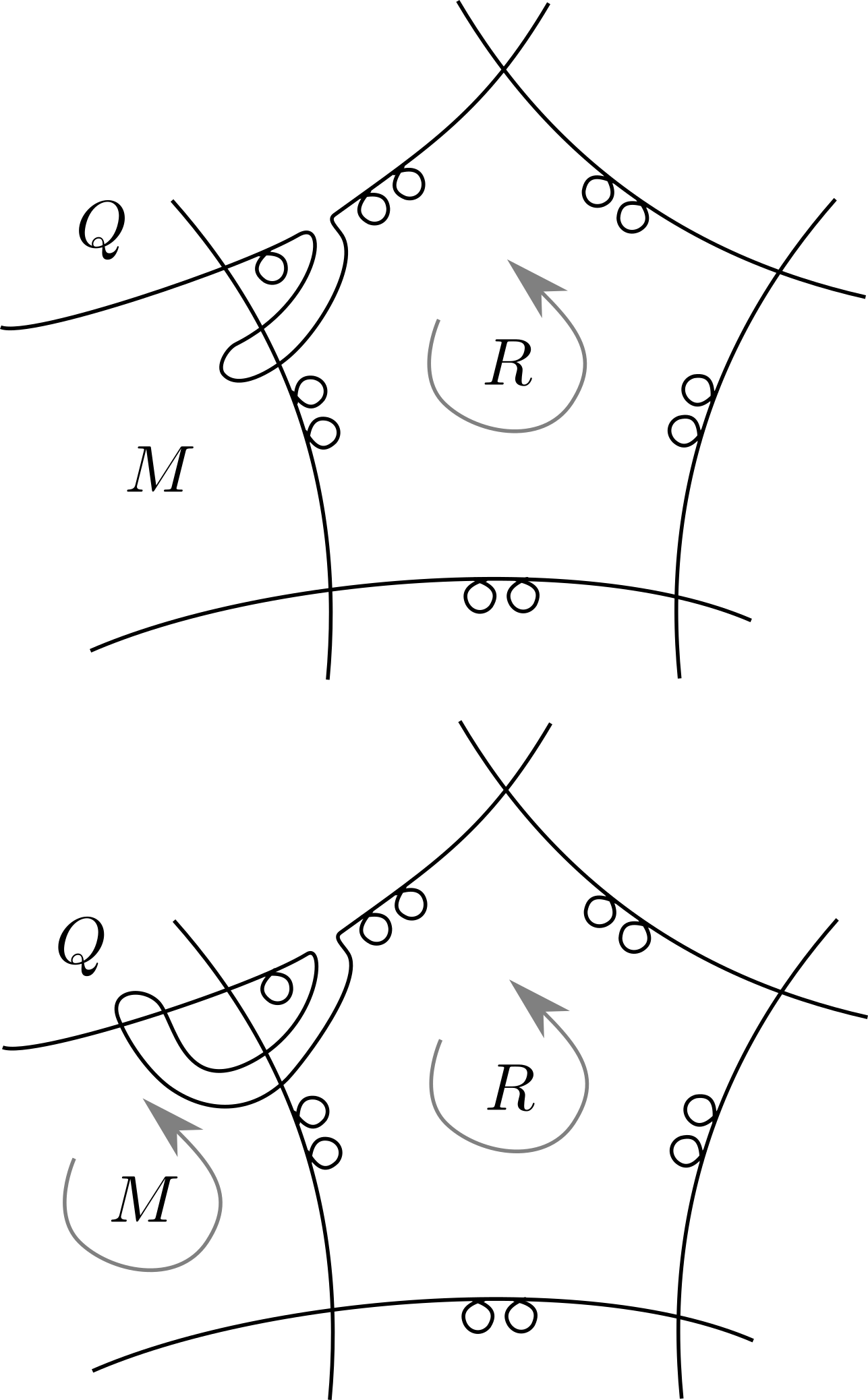}
\caption{How to choose a coherent cyclic ordering for the orientation system.}
\label{fig:sistciclico}
\end{figure}
Once we have made the first $\Omega_2^+$ move of Figure \ref{fig:sistorient} (and thus chosen a clockwise or counterclockwise orientation), there are only $2$ other $\Omega_2^+$ moves that increase the number of regions with $3$ sides by one and change the valence by approximately\footnote{Lower case letters denote the number of edges in the corresponding regions.} $4q -2m$. This move will also increase by $2$ the number of edges of the region denoted by $Q$; we are going to choose the only cyclic orientation based at the vertex in the dual, corresponding to the region $M$, that has $Q$ after $R$. Repeating this process for all regions produces a well-defined and coherent orientation system for each vertex in the dual graph, hence uniquely determines the embedding of the dual and consequently the knot projection.

Now suppose we have a periodic knot type $K$; in order to repeat the previous strategy we need to be able, for each $D \in \diag$ with $p_1(D)=0$, to find in a controlled way a sparse diagram $D^\prime \in \diag$, and a sufficiently large ball, centred in $D^\prime$, such that all diagrams in this ball are non-periodic. Since we can verify whether a diagram is periodic\footnote{As explained in the proof of Theorem \ref{distinguemosse}.}, by changing the order of the $\Omega_1^+$ multi-edges appearing in the previous construction, and/or their valence, we can achieve a sparse configuration with these properties.
\end{proof}

To conclude the proof, we only need to be able to say that the only possible knots sharing all projections without curls are mirrors of one another. 
\begin{prop}\label{distinguediagramma}
The $S^2$-graph $\sfera$ detects some diagrams of $K$ up to mirroring.
\end{prop}
\begin{proof}
Suppose we have two knots $K$ and $K^\prime$ sharing the same graph. Take a vertex $D$ of $\sfera$ such that $\#\Omega_1^-(D) = 0$. The corresponding vertex $D^\prime$ in the isomorphic graph $\mathcal{G}_{S}(K^\prime)$ will have the same knot projection as $D$ by the previous Proposition. Hence if $K \neq K^\prime$ the diagrams must differ in at least one crossing. If they differ in \emph{all} crossings, then $K^\prime$ is the mirror of $K$, and we are done. Otherwise there must be a pair of crossings that up to mirroring looks like the pair in the top part of Figure \ref{fig:lacciostrano} in $D$ and $D^\prime$.\\
\begin{figure}[h!]
\includegraphics[width=10cm]{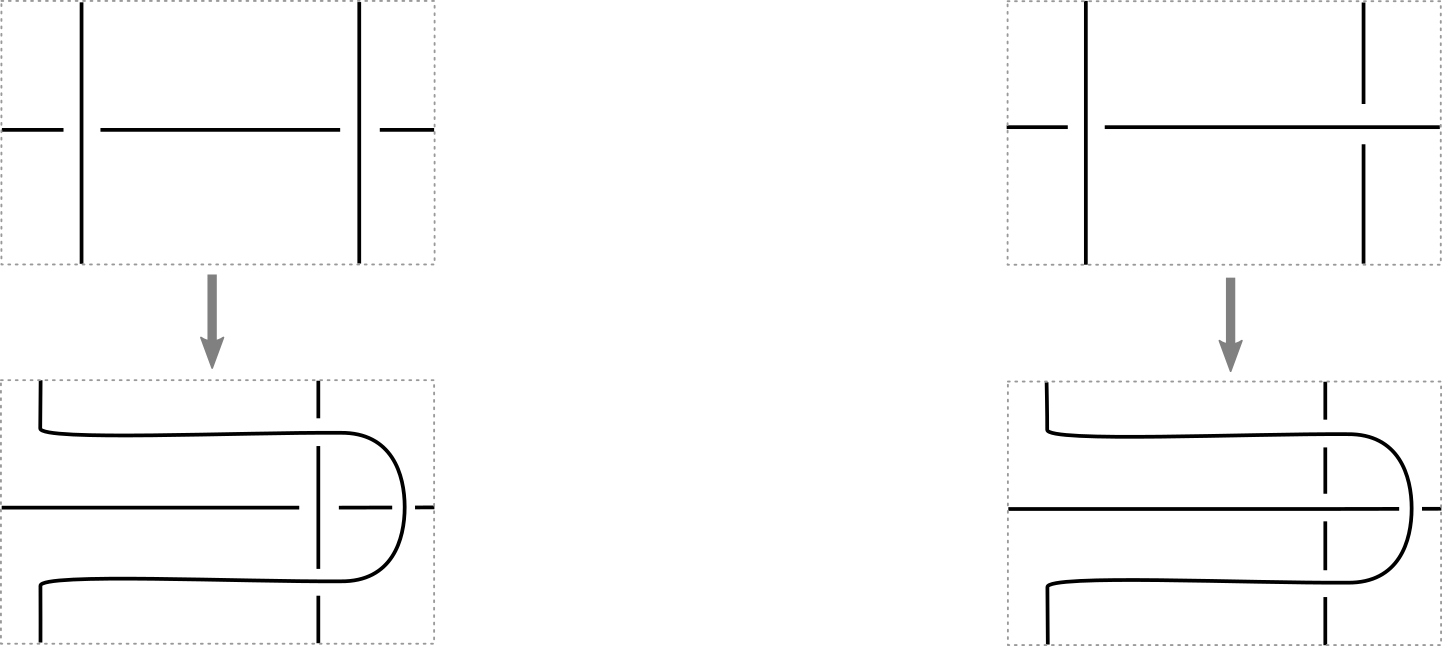}
\caption{The two paths in $\sfera$ and $\mathcal{G}_{S}(K^\prime)$. The grey arrows denote the sequence of $\Omega_2^+$-$\Omega_3$-$\Omega_2^-$ moves connecting the two diagrams.}
\label{fig:lacciostrano}
\end{figure}
 
Now, perform the sequence of $\Omega_2^+$-$\Omega_3$-$\Omega_2^-$ moves that takes the upper diagrams in Figure \ref{fig:lacciostrano}, and ends in the lower ones.
Note that these paths are well defined, since all the diagrams involved respect the condition $p_1=0$; thus we are able to actually determine the effect of these moves on the projections by the previous Proposition.
Consider Figure \ref{fig:lacciostrano2}; there are two distinct sequences of $\Omega_2^+$-$\Omega_3$ (differing by the choice of over/under passing for the first $\Omega_2^+$ move) starting from the diagram on the top-left of Figure \ref{fig:lacciostrano2} and ending in two different diagrams sharing the same projection. On the other hand, there is only one way to perform an $\Omega_2^+$ from the diagram on the top-right of the Figure in order to be able to complete the sequence with an $\Omega_3$ and obtain a diagram with the same projection as the other two. Again, these moves are all well defined thanks to Proposition \ref{riconosceproj}.

\begin{figure}[h!]
\includegraphics[width=13cm]{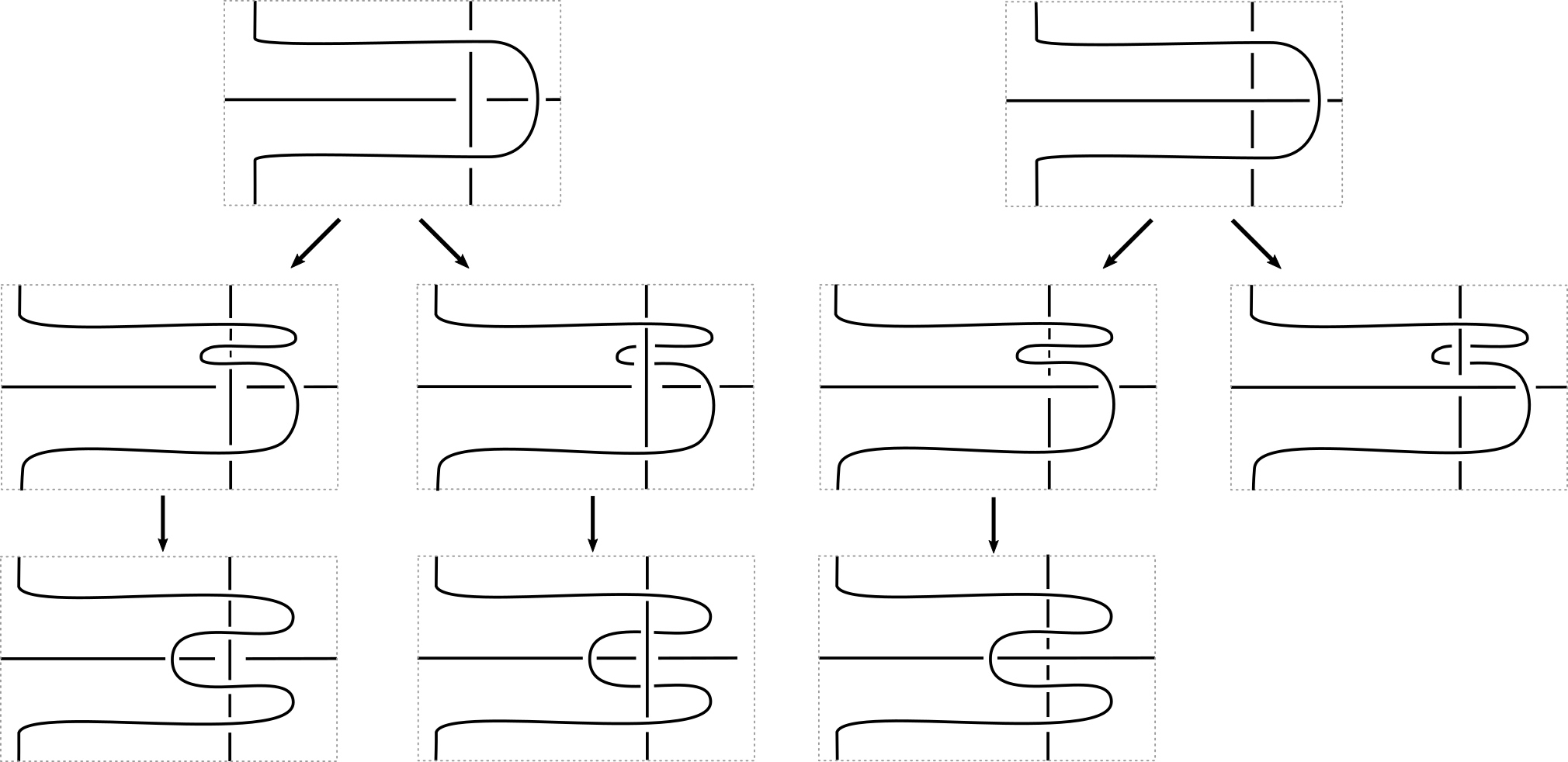}
\caption{The two paths in $\sfera$ and $\mathcal{G}_{S}(K^\prime)$. The top/middle arrows denote a $\Omega_2^+$ and $\Omega_3$ move respectively.}
\label{fig:lacciostrano2}
\end{figure}
Hence the two graphs can not coincide, since there is a path in one of the graphs which is not present in the other one, and we can conclude.
\end{proof}

The three previous propositions together with Proposition \ref{valenzabanale} can be easily seen to imply Theorem \ref{thm:completeness}, but as a matter of fact the result proved is even stronger, since it allows to recover the actual diagrams\footnote{For knots $K \neq \bigcirc$.} represented by some specific vertex of the graph, and not only its knot type. From the proof of Proposition \ref{riconosceproj} we are actually obtaining an embedding for the graph which is dual to the knot projection corresponding to the diagram $D_0$. Hence, this proves that we can actually get back the shape of any diagram not containing any region with $1$ edge (in the non-periodic case).\\ 

This next result follows directly from the proofs of the previous three Propositions:

\begin{cor}\label{cor:pallachar}
Let $K$ be a knot. For every vertex $D \in \sfera$ there exists an integer $R>0$ such that  $S_R(D)$ is \emph{characterizing}, meaning that this graph can only appear in $\sfera$. Moreover, in the non-periodic case, $R$ is computable.
\end{cor}

A similar argument should guarantee the completeness of the planar $\mathcal{R}$-graphs, even though the whole process is complicated by the fact that the presence of the external region does not allow a straightforward adjustment of Proposition \ref{prop:determinaP}.

\section{The blown-up Gordian graph}\label{blownup}

We can unify the Gordian and Reidemeister graphs in a single object, by a sort of ``blowup'' construction; just replace each vertex of the Gordian graph with the corresponding $\reid$. The edges between two knots in the Gordian graph can be split into edges between the diagrams realising the crossing changes.
\begin{defi}
Define the \emph{blown-up Gordian graphs} $\mathcal{G}^*$ and $\mathcal{G}^*_S$ respectively, as the graphs whose vertices are knot diagrams in the plane (respectively in $S^2$) up to the corresponding notion of diagram isotopy; there is an edge between two vertices if and only if they are connected by a single Reidemeister move or a crossing change.
\end{defi}

As in the previous setting, the valence of each vertex is finite.
For non-periodic diagrams we have $$v_*(D) = v(D) + cr(D),$$
where $v_*(D)$ denotes the valence of $D$ in $\mathcal{G}^*$, and $v(D)$ is the valence of $D$ in the corresponding $\mathcal{R}$-graph. For non-periodic diagrams we only get an inequality.

Note that $\mathcal{G}^*$ admits an order $2$ automorphism, induced by changing all crossing of each diagram, \emph{i.e.}$\;$taking the mirror image. The only fixed points of this automorphism are the diagrams of amphichiral knots which are equivalent to their mirror up to planar isotopy.

\begin{rmk}
There are embeddings of $\mathcal{G} (K) \hookrightarrow \mathcal{G}^*$ and $\mathcal{G}_S (K) \hookrightarrow \mathcal{G}^*_S$ for each $K \in \nodi$, and there are many crossing-change edges in both $\mathcal{G}^*$ and $\mathcal{G}^*_S$ connecting two diagrams in the same isotopy class; according to the \emph{cosmetic crossing conjecture} (\cite[Problem 1.58]{kirbylist}) all these should correspond to nugatory crossings. It would be interesting to explore the possible applications of these graphs to the conjecture.
\end{rmk}

If we look at the ball of radius $1$ in $\mathcal{G^*}$ about a diagram $D$, we find all the length $3$ paths of Theorem \ref{prop:lung3}, together with a new configuration, shown in Figure \ref{lungh3ingordian}. The fact that the only new triangle\footnote{That is, triangles that contain at least one crossing change.} appearing is actually this one, follows easily by considering Arnold's and Hass-Nowik's invariants, together with crossing number and writhe, as in Theorem \ref{prop:lung3}. More precisely, using these invariants we can restrict to cycles of the form $\Omega_1$-$\Omega_1$-$C$, where the $\Omega_1$s create crossings of opposite sign, and $C$ denotes a crossing change. Then, with the same line of thought as in Theorem \ref{prop:lung3}, we can prove  that the curls must lie in the same region by taking in account the self-touching number. It follows that the regions under the two curls have the same number of edges. However, we are not able to prove that the crossing change happens exactly on the curls.\\

\begin{figure}[h!]
\includegraphics[width=6cm]{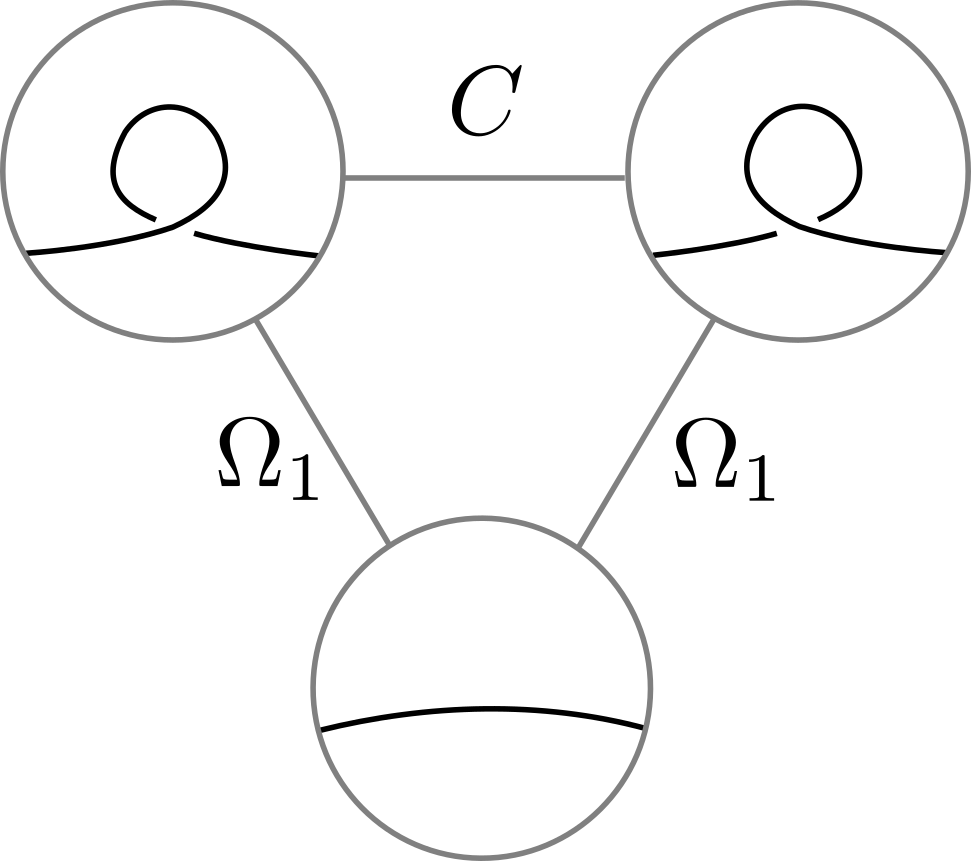}
\caption{The only other possible length $3$ path in $\mathcal{G^*}$ which is not a path in any $\reid$.}
\label{lungh3ingordian}
\end{figure}

By extending the proof of Theorem \ref{distinguemosse} to the blown-up graph it is possible to prove the analogous result; namely that $\mathcal{G}_*$ detects the crossing changes and Reidemeister moves.

It is also possible to employ the blown-up graph in quite different contexts, like modelling DNA pathways or consider walks on it to produce cryptographic protocols. These applications will be the subject of upcoming works.\\

Other related ongoing projects include a translation of the concepts outlined in this paper to a plethora of other settings; the rough idea is the following: given a recipe to present knots, and a finite set of moves to pass between equivalent presentations, one obtains a related graph. In an upcoming paper 
we are going to study what happens in the case of grids, braids, tangles, pointed and framed diagrams. We will also provide computations for the corresponding diagram complexity invariants, for low-crossing knots and some infinite family. Moreover, we are going to explore some of the connections between the $\mathcal{R}$-graphs defined here and the topology of the discriminant hypersurfaces in Arnold's and Vassiliev's constructions (\cite{arnold}, \cite{vass}).\\

We conclude with some questions.
\begin{ques}
Does there exists a periodic knot type $K$ such that the minimum of the complexity (in either graph) is not attained at a periodic diagram?
\end{ques}
\begin{ques}
Are all knot types simple (as in Definition \ref{defi:complessita})?
\end{ques}
\begin{ques}\label{quesperiodici}
If a diagram $D\subset S^2$ of a non-trivial knot is periodic, is it true that all other diagrams in $S(D)$ are not periodic?
\end{ques}
\begin{ques}\label{ques:minimo}
Is the $S^2$-graph obtained from a minimal set of Reidemeister moves a complete invariant?
\end{ques}
\begin{ques}
To what extent do the filtrations $\mathcal{F}(K)$ and $\widetilde{\mathcal{F}}(K)$ classify knot types?
\end{ques}
The following (hard) question was suggested by M. Lackenby:
\begin{ques}
We have shown that the isomorphism class of the $S^2$-Reidemeister graph is a complete knot invariant. Is the \emph{quasi-isometry} class of the graph is a complete invariant, or if not, to what extent does it distinguishes inequivalent knots, or detects interesting properties of the knot?
\end{ques}
The following question was asked by D. Cimasoni:
\begin{ques}
Which knot invariants (such as genus, absolute value of the writhe, polynomials...) can be extracted from the $\mathcal{R}$-graphs?
\end{ques}
\begin{ques}
Does the connectivity of the $\mathcal{R}$-graphs coincide with the minimal complexity?
\end{ques}
\begin{ques}
We have exhibited an infinite set of spheres representing non-trivial elements of $H_2(\mathcal{CG}(K);\Z)$ in Section \ref{global}. Are there any embedded closed surfaces of higher genus (note that such surfaces would automatically be non-trivial in homology)?
\end{ques}

\bibliographystyle{amsplain}

\begin{thebibliography}{10}
\bibitem{alexbriggs}
James Waddell Alexander and Garland Baird Briggs, \emph{On Types of Knotted Curves}, The Annals of Mathematics, Second Series, Vol. 28, No. 1/4 (1926 - 1927), pp. 562-586.
\bibitem{arnold}
Arnold, Vladimir I., \emph{Plane curves, their invariants, perestroikas and classifications}, Advances in Soviet Mathematics 21 (1994): 33-91.
\bibitem{baadercross}
Baader Sebastian, \emph{Note on crossing changes}, The Quarterly Journal of Mathematics vol. 57 n.2 2006.
\bibitem{coward2014upper}
Coward, Alexander and Lackenby, Marc, \emph{An upper bound on Reidemeister moves}, American Journal of Mathematics, vol. 136, n. 4, 2014.
\bibitem{ekholm2016complete}
Ekholm, Tobias and Ng, Lenhard and Shende, Vivek, \emph{A complete knot invariant from contact homology}, V. Invent. math. (2017).
\bibitem{period}
Flapan, Erica, \emph{Infinitely periodic knots}, Canad. J. Math, volume 37, 1985.
\bibitem{gordon1989}
Gordon, Cameron McA. and Luecke, John, \emph{Knots are determined by their complements}, Bull. Amer. Math. Soc. (N.S.), n. 1, vol 20, 1989.
\bibitem{halin}
Halin, Rudolph, \emph{\"{U}ber die {M}aximalzahl fremder unendlicher {W}ege in {G}raphen}, {M}athematische {N}achrichten 30.1-2 (1965): 63-85.
\bibitem{hass2008invariants}
Hass, Joel and Nowik, Tahl, \emph{Invariants of knot diagrams}, Mathematische Annalen, vol 342, nr. 1, 2008.
\bibitem{itobook}
Ito, Noboru, \emph{Knot Projections}, CRC Press, 2016.
\bibitem{joyce1982classifying}
Joyce, David, \emph{A classifying invariant of knots, the knot quandle}, Journal of Pure and Applied Algebra, vol. 23, n. 1, 1982.
\bibitem{kauffman2006hard}
Kauffman, Louis H. and Lambropoulou, Sofia, \emph{Hard unknots and collapsing tangles}, Introductory Lectures on Knot Theory-Selected Lectures presented at the Advanced School and Conference on Knot Theory and its Applications to Physics and Biology ICTP, Trieste, Italy, 11--29 May 2009.
\bibitem{kirbylist}
Kirby Rob, \emph{Problems in low dimensional manifold theory}, In Algebraic and geometric topology, Part 2, Proc. Sympos. Pure Math., XXXII, pages 273-312. Amer. Math. Soc., Providence, R.I., 1978.
\bibitem{aboutgordian}
March\'{e}, Julien, \emph{About the {G}ordian graph at infinity}, Comptes Rendus Mathematique 340.5 (2005): 363-368.
\bibitem{markoff1936uber}
Markoff, Andrei, \emph{\"{U}ber die freie \"{A}quivalenz der geschlossenen Z\"{o}pfe}, Recueil Math. Moscou, vol 1, 1935.
\bibitem{miyazawa2012distance}
Miyazawa, Yasuyuki, \emph{A distance for diagrams of a knot}, Topology and its Applications, vol 159, n. 4, (2012).
\bibitem{polyak2010minimal}
Polyak, Michael, \emph{Minimal generating sets of {R}eidemeister moves}, Quantum Topology, 2010, 1.4: 399-411.
\bibitem{reidemeister}
Reidemeister, Kurt, \emph{Elementare Begr{\"u}ndung der Knotentheorie}, Abhandlungen aus dem Mathematischen Seminar der Universit{\"a}t Hamburg, vol 5,  nr. 1, 1927.
\bibitem{vass}
Vassiliev, Victor Anatolyevich, \emph{Cohomology of knot spaces}, Theory of singularities and its Applications (Providence), V.I. Arnold (Ed.), Amer. Math. Soc, Providence (1990).
\bibitem{budneyMO}
\emph{Structure of Gordian graph of knots}, MathOverflow, URL: \url{https://mathoverflow.net/q/255140} (version: 2016-11-21).


\end{thebibliography}

\end{document}